\documentclass[11pt]{amsart} 
\usepackage{amsmath,amssymb,a4wide} 
\usepackage{mathabx}
\usepackage{color}
\usepackage{bm}
\usepackage{dsfont}
\usepackage{tikz}
\usepackage{subfigure}
\usepackage{mathrsfs}
\usepackage{ulem}
\usepackage{enumitem}
\definecolor{refkey}{gray}{0.75}
\colorlet{labelkey}{blue}
\usepackage[percent]{overpic}
\usepackage{grffile}%

\usepackage{tikz,pgfplots}
\pgfplotsset{compat=1.16}
\newtheorem{theorem}{Theorem} 
\newtheorem{lemma}[theorem]{Lemma}

\newtheorem{corollary}[theorem]{Corollary}

\theoremstyle{remark}
\newtheorem*{merci}{Acknowledgments}

\providecommand{\abs}[1]{|#1 |}

\newcommand{\wh}{\widehat }

\newcommand{\R}{{\mathbb R}} 
\newcommand{\C}{{\mathbb C}} 
\newcommand{\N}{{\mathbb N}}

\newcommand{\dd}{{\rm d}}
\newcommand{\gq}{\mathfrak{q}}
\newcommand{\boE}{\mathcal{E}}
\newcommand{\sech}{\operatorname{sech}}

\newcommand{\W}{\mathcal{W}}
\newcommand{\Emin}{E_\textup{min}}
\newcommand{\rot}{\textup{rot}}
\usepackage{color}

\title[Numerical solitons for NLS]
{Numerical computation of dark solitons of a nonlocal nonlinear Schr\"odinger
equation}

\author[A. de Laire]{Andr\'e de Laire}
\author[G. Dujardin]{Guillaume Dujardin}
\author[S. L{\'o}pez-Mart{\'i}nez]{Salvador L{\'o}pez-Mart{\'i}nez}

\begin{document}


\begin{abstract}
  The existence and decay properties of dark solitons for a large class of nonlinear
  nonlocal Gross-Pitaevskii equations with nonzero boundary conditions in dimension one
  has been established recently \cite{dLMar2022}.
  Mathematically, these solitons correspond to minimizers of the energy at fixed momentum
  and are orbitally stable.
  This paper provides a numerical method to compute approximations of such solitons
  for these types of equations, and provides actual numerical experiments for several
  types of physically relevant nonlocal potentials.
  These simulations allow us to obtain a variety of dark solitons, and to comment on their
  shapes in terms of the parameters of the nonlocal potential.
  In particular, they suggest that, given the dispersion relation,
  the speed of sound and the Landau speed are important values to understand
  the properties of these dark solitons.
  They also allow us to test the necessity of some sufficient conditions
  in the theoretical result proving existence of the dark solitons.



\end{abstract}


\maketitle
{\small\noindent 


\section{Introduction}
\label{sec:cont}

We consider the one-dimensional nonlocal Gross--Pitaevskii equation for
$\Psi : \R\times\R \rightarrow \C$
\begin{equation}
  \label{eq:GP}
  i\partial_t \Psi = \partial_x^2 \Psi + ({\mathcal W}\star (1-|\Psi|^2)) \Psi,
\end{equation}
with nonvanishing  boundary conditions at infinity, i.e.\ 
\begin{equation}
  \label{eq:BCGP}
   \lim_{|x|\to +\infty} |\Psi(x,t)| = 1, \quad \text{for all }t,
\end{equation}
and appropriate initial conditions, where $\star$ denotes the convolution in space with a tempered distribution $\W$.
This equation  appears, for instance,  as a model for the evolution of a one-dimensional optical beam of intensity $|{\Psi}|^2$ in a self-defocusing nonlocal Kerr-like medium, where $\W$ characterizes the nonlocal response of the medium \cite{nikolov2004,krolikowski2000}. 
In this case, condition \eqref{eq:BCGP} is natural when studying dark optical solitons.
In all the physical situations, $\mathcal W$  is assumed to be a real-valued
and symmetric distribution, so that  \eqref{eq:GP} is a Hamiltonian equation,
and the energy and momentum 
	\begin{align}
		\label{energy}
	E(\Psi(t))=&\frac12 \int_{\R}\abs{\partial_x\Psi(t)}^2\,\dd x +\frac 14 \int_{\R}(\W\star (1-\abs{\Psi(t)}^2))(1-\abs{\Psi(t)}^2)\,\dd x,\\
			\label{momentum}
	P(\Psi(t))&=\int_{\mathbb{R}}\langle i \partial_x \Psi(t),\Psi(t) \rangle_{ \C}\left(1-\frac{1}{|\Psi(t)|^2}\right)\dd x,
	\end{align}
are formally conserved, where $\langle z_1,z_2 \rangle_{\C} =\Re(z_1 \bar z_2)$, for $z_1$, $z_2\in \C$. 

In the most typical first approximation, $\mathcal W$ is considered
as a Dirac delta function $\delta_0$, which, from \eqref{eq:GP}, leads to the standard
Gross--Pitaevskii equation
\begin{equation}
	\label{eq:GP_local}
	i\partial_t \Psi = \partial_x^2 \Psi +  (1-|\Psi|^2) \Psi,
\end{equation}
with nonvanishing conditions at infinity.      
We refer to \cite{delaire-mennuni,dLMar2022,de2010global} for more details about the theory and applications of equations \eqref{eq:GP} and  \eqref{eq:GP_local}.



Dark solitons play a fundamental role in the  dynamics of finite-energy solutions to \eqref{eq:GP} propagating in a constant background. Thus, these solutions  
have brought  theoretical and experimental attention in the last decades (see e.g. \cite{frantzeskakis2010,becker2008,BethGrav2015}).
Precisely, these solitons correspond to smooth traveling wave solutions of the form $\Psi(x,t)=u(x-ct)$,
for some speed $c\in\R$, with localized derivative.
Therefore, we focus on finding functions $u : \R\rightarrow \C$ solving
\begin{equation}
  \label{eq:TWc}
  \tag{$\textup{TW}_c$}
    icu'+u''+({\mathcal W}\star (1-|u|^2)) u = 0,
\end{equation}
and belonging to the energy space
	$$\mathcal{E}(\mathbb{R})=
\{v \in H^{1}_{\text{loc}}(\mathbb{R}) : 1-|v|^{2}\in L^{2}(\mathbb{R}), \ v' \in L^{2}(\mathbb{R})\}.$$
Notice that this is justified by assuming that the Fourier transform of $\W$ is bounded, i.e.\ that   $\wh\W\in L^\infty(\R)$ (see \cite{dLMar2022}). 
Note that, in this paper, we use the convention that for  $f\in L^1(\R)$,
\[\wh f(\xi)=\int_\R e^{-ix\xi}f(x) {\rm d}x,\quad \xi\in\R,\] 
so that $\wh \delta_0=1$.
Moreover, we only consider distributions $\W$ such that $\wh \W$ is (bounded and) continuous,
so that we can, and will, also assume the normalization condition $\wh \W(0)=1$.

Notice that any constant function $u$ of modulus $1$ is a solution to \eqref{eq:TWc} with zero energy.
We refer to such functions as trivial solutions.
Observe that there are also oscillating solutions to \eqref{eq:TWc} with infinite energy such as 
\begin{equation*}
	u_{r,c}^\pm(x)=r \exp\Big( ix \Big( {\frac{-c\pm\sqrt{c^2+4(1-r^2)}}{2}}\Big) \Big),\quad\text{ for all }r\in (0,1) \text{ and } c\in\R.
\end{equation*}

Another remark is that, by taking the complex conjugate of $u$ in equation \eqref{eq:TWc},
we only need to consider $c\geq 0$. 

In the sequel, when we refer to a solution
or a soliton to \eqref{eq:TWc}, we implicitly assume that it is a \textit{nontrivial finite-energy} solution to \eqref{eq:TWc}. Let us recall the following result concerning 
the properties of these solutions.

\begin{lemma}[Decay properties of dark solitons (see \cite{dLMar2022})]
  \label{lemma:regularity}
  Let  $\W$ be a real-valued even tempered distribution such that $\widehat\W\in L^\infty(\R)$.
  Assume that $c\geq 0$ and that $u\in \mathcal E(\R)$ is a solution to \eqref{eq:TWc}.
  Then $u$ is bounded and of class $\mathcal C^\infty(\R)$.
  Moreover, if $c>0$, then $u$ does not vanish on $\R$ and there exists a smooth lifting of $u$.
  More precisely, there is real-valued function $\theta\in\mathcal C^\infty(\R)$
  and another $\mathcal C^\infty$ function $\eta$ over $\R$ with values in $(-\infty,1)$ such that
  $u=\sqrt{1-\eta} e^{i\theta}$ on $\R$,
  with for all $k\in\N$, $\theta',\eta \in H^{k}(\R)$, and 
\begin{equation}
	\label{limits-infty}
	\abs{u}(\pm \infty)=1,\  D^j u(\pm \infty)=D^j\theta(\pm \infty)=D^j\eta(\pm\infty)=0, \ \ \text{for all } j\geq 1.
\end{equation}
\end{lemma}

In addition, the \textit{existence} of such solutions is also proved in  \cite{dLMar2022},
and we can rewrite this result as follows.
For the sake of completeness, we provide a proof in Appendix.
\begin{theorem}\label{thm:dLMar}
  Let  $\W$ be a real-valued even tempered distribution such that $\widehat\W\in L^\infty(\R)$.
  Assume that there is $\sigma \in(0,1]$ such that 
	\begin{equation}
		\label{def:sigma}
		\inf_\R \big(     \widehat \W(\xi)+\xi^2/2\big)=\sigma.
	\end{equation}	
  Then, for almost every $c\in (0,\sqrt{2\sigma})$,
  there exists a nontrivial solution $u\in \mathcal E(\R)$ to \eqref{eq:TWc}.
\end{theorem}

Furthermore, by imposing some extra conditions on $\W$, 
it is proved in \cite{dLMar2022} that the existence  holds true \textit{for every} $c\in (0,\sqrt{2\sigma})$.
These kinds of conditions are satisfied, for instance, by the Dirac delta function $\W=\delta_0$, so that $\widehat \W=1$.
This agrees with the results in 
\cite{bethuel2008existence}, where the authors show that for $c\in[0,\sqrt2 )$, the unique solutions, up translation and multiplication by a complex constant of modulus one, are explicitly given by
	\begin{equation}
	\label{sol:1D}
	u_{c}(x)=\sqrt{\frac{2-c^2}{2}}\tanh\Bigg(\frac{\sqrt{2-c^2}}{2}x\Bigg)-i\frac{c}{\sqrt{2}}.
\end{equation}
In addition, they show that if $c\geq\sqrt{2}$, the only solutions to \eqref{eq:BCGP} are the trivial ones.

Concerning the nonlocal equation \eqref{eq:GP},
as explained in \cite{delaire-mennuni},
the Bogoliubov dispersion relation  is given by 
\begin{equation}\label{bogo}
	\omega(\xi)=\sqrt{\xi^4+2\widehat \W(\xi) \xi^2}.
\end{equation}
Thus, assuming that $\widehat \W$ is continuous and non-negative in $0$,
we get $\omega(\xi)\approx {(2\wh \W(0))}^{1/2}\abs{\xi}$, for $\xi\approx 0$.
We infer that the value of the sonic speed is
\begin{equation}
  \label{eq:sonicspeed}
  c_s(\W) \underset{{\rm def}}{:=} \lim_{\xi\to 0} \frac{\omega(\xi)}{\xi} = (2\wh \W(0))^{1/2}.
\end{equation}
With our normalization $\wh \W(0)=1$, the sonic speed is hence set to $c_s=\sqrt 2$. 

In \cite{dLMar2022}, the authors proved that 
if   $\wh \W$ 	is smooth in a neighborhood of the origin, with  $\wh\W\geq 0$ a.e.\ on $\R$,   	$\wh \W(0)=1$ and  $(\wh \W)''(0)\neq -1$, then  \eqref{eq:TWc} admits no nontrivial solution for the sonic speed $c_s=\sqrt 2$.

Another important concept in the study of superfluids is the Landau's critical speed,
that explains the existence of excitations in a superfluid called {\it rotons}.
Mathematically, the Landau speed is defined as
\begin{equation}
	\label{def:landau}
c_L(\W)\underset{{\rm def}}{:=}\inf_{\R}\frac{\omega (\xi)}{\abs{\xi}}, 
\end{equation}
and the Landau's criterion is that there exists a point $\xi_{\rot}>0$,
that we call {\it roton minimum}, such that $c_L(\W)={\omega (\xi_\rot)}/{\abs{\xi_\rot}}.$
Thus, if $\omega$ is differentiable at $\xi_\rot>0$,
a roton minimum $\xi_\rot$ is a point where the group velocity $\omega'(\xi_\rot)$
and the phase velocity $\omega(\xi_\rot)/\xi_\rot$ are equal.
In classical weakly interacting BEC, the Landau's speed and sonic speed are equal
$c_L(\W)=c_s(\W)$,
while in some superfluids such as  $\phantom{}^4$He,
the Landau's speed is smaller than the speed of sound \cite{barenghi2001quantized}.
In the physical literature, it is expected that when  $c_L(\W)<c_s(\W)$, 
solitons solution should include some 
additional density oscillations around the vortex, due to the presence of a roton minimum. 
Moreover, Berloff and Roberts provided some formal arguments in higher dimensions,
indicating that the condition $c<c_L(\W)$ should be necessary to obtain nontrivial solitons
\cite{berloff0}.

To give an idea of the existence of a roton minimum, let us consider
the Gaussian potential $\wh \W_\lambda(\xi)=e^{-\lambda^2\xi^2}$.
If $\lambda\in [0,\sqrt{2}/2)$, the dispersion curve $\xi\mapsto \omega(\xi)$ is convex,
we have $c_L(\W_\lambda)=c_s=\sqrt 2$,
and the aspect of this curve is similar to the one depicted in the left panel
of Figure~\ref{fig:dispersion}. 
In contrast, if $\lambda>\sqrt{2}/2$, $\omega$ is concave near the origin and there is a roton minimum
$\xi_\rot>0$ as depicted in the right panel in Figure~\ref{fig:dispersion},
where the slope of the green line corresponds to the Landau speed. 

In this context, we can 
recast Theorem~\ref{thm:dLMar} as follows
 \begin{corollary}
\label{cor:dLMar}
Let  $\W$ as in Theorem~\ref{thm:dLMar}.
Then the Landau's critical speed  is given by 
 		\begin{equation}
 			\label{landau}
 			c_L(\W)=\sqrt{2\sigma},
 		\end{equation}	
and, for almost every $c\in (0,c_L(\W))$,
there exists a nontrivial solution $u\in \mathcal E(\R)$ to \eqref{eq:TWc}.
\end{corollary}

\begin{figure}
		\begin{tabular}{cc}
	\resizebox{0.45\textwidth}{!}{
\begin{overpic}
	[trim=60 540 350 50,clip]{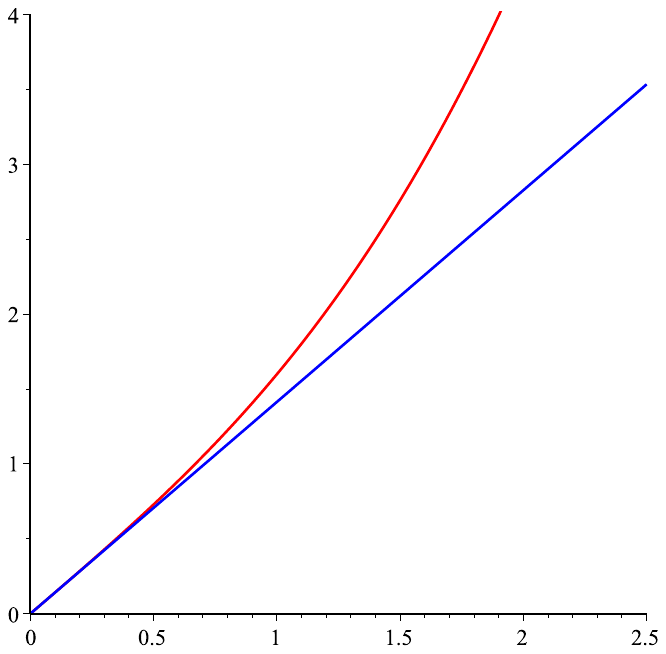}
		\put(50,0){{$\xi$}}
		\put(78,87){{$\color{red}\omega_\lambda$}}
		\put(78,60){{$\color{blue}\sqrt 2 \xi$}}
	\end{overpic}
}
&
\resizebox{0.45\textwidth}{!}{
	\begin{overpic}
		[trim=50 540 350 50,clip]{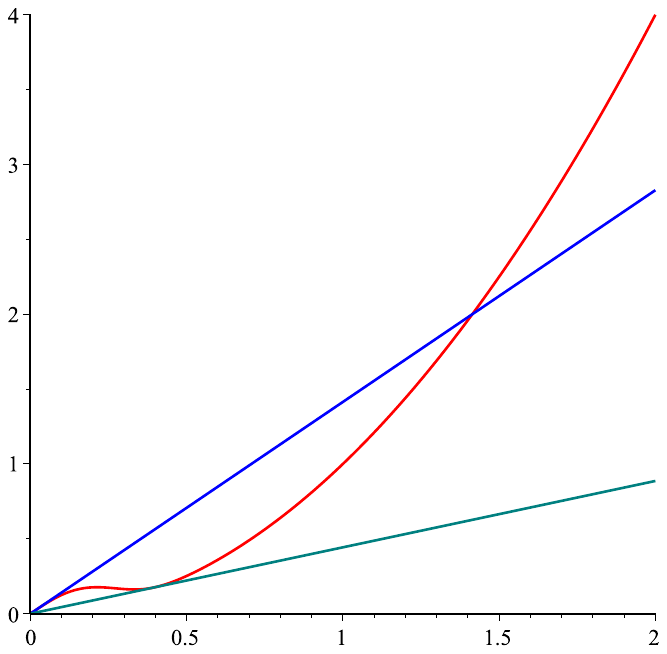}
		\put(50,0){{$\xi$}}
		\put(92,72){{$\color{red}\omega_\lambda$}}
		\put(45,41){{$\color{blue}\sqrt 2 \xi$}}
		\put(90,18){{$\color{teal} c_L  \xi$}}
	\end{overpic}
}
			
		\end{tabular}
	\caption{Dispersion curve $\xi\mapsto \omega_\lambda (\xi)$ (red),
          sound speed curve $\xi\mapsto \sqrt 2\times \xi$ (blue),
          and Landau speed curve $\xi\mapsto c_L \times \xi$ (green),
          for the Gaussian potential $\wh \W(\xi)=e^{-\lambda^2\xi^2}$ with $\lambda=1/2$ (left panel),
          and $\lambda=5$ (right panel).
          In the left panel, the sonic speed $c_s=\sqrt2$ is equal to the Landau speed $c_L$.}
	\label{fig:dispersion}
\end{figure}


As seen in Lemma~\ref{lemma:regularity},  if  $c>0$, then 
solutions to \eqref{eq:TWc} do not vanish, i.e.\  $\abs{u(x)}>0$, for all $x\in\R$. 
Moreover, as shown in \cite{dLMar2022}, for $c>0$, the change of variables, 
$\eta=1-\abs{u}^2$ allows us   to recast  \eqref{eq:TWc} as a single real-valued equation
\begin{equation}
	\label{eq:zeroJ}
	J_c(\eta,\lambda)=0,
\end{equation}
where
\begin{equation}
	\label{eq:defJc}
	J_c(\eta,\lambda) := \eta''-2{\mathcal W}_\lambda\star \eta + c^2\eta + \frac{c^2\eta^2}{2(1-\eta)}
	+\frac{(\eta')^2}{2(1-\eta)} + 2({\mathcal W}_\lambda\star \eta)\eta.
\end{equation}
Notice that in \eqref{eq:defJc}, instead of a single potential $\W$,
we have introduced a family of potentials $\W_\lambda$ indexed by a real parameter $\lambda$.
This will allow us to relate the nonlocal equation \eqref{eq:zeroJ} to the local one.
In fact, in several examples below, we will choose the family $\W_\lambda$ such that $\W_0=\delta_0$.

By Lemma~\ref{lemma:regularity}, if $u$ is a solution to \eqref{eq:TWc} with $c>0$, then $\eta<1$ on $\R$, so that $J_c$ is well-defined.
Also, we can recover $u$ from $\eta$  by setting 
\begin{equation}
	\label{def:u}
	u(x)=\sqrt{1-\eta(x)}e^{i\theta(x)}, \quad \textup{where}\quad  \theta(x)=\frac{c}2\int_0^x \frac{\eta(s)}{1-\eta(s)}{\rm d}s.
\end{equation}
Moreover, if $\eta$ solves \eqref{eq:zeroJ}, then  the momentum \eqref{momentum} and the energy \eqref{energy}
can be computed as follows (see \cite{dLMar2022} for details) 
\begin{align}
	\label{eq:Eeta}
	E(\eta) &= \frac{c^2}{8} \int_{\R}\frac{\eta(x)^2}{1-\eta(x)} \dd x
	+ \frac{1}{8} \int_{\R} \frac{\eta'(x)^2}{1-\eta(x)} \dd x
	+ \frac{1}{4} \int_{\R} \left({\mathcal W}_\lambda\star\eta\right)(x)\eta(x) \dd x,\\
	\label{eq:Peta}
	P(\eta) &= \frac{c}{4} \int_{\R} \frac{\eta(x)^2}{1-\eta(x)} \dd x.
\end{align}

For instance, in these new variables,  the soliton $u_c$ in \eqref{sol:1D}, 
associated with the solution to \eqref{eq:TWc}  when $\mathcal W=\delta_0$, 
reads, for $c\in(0,\sqrt 2)$,
\begin{equation}
	\label{eq:defeta0}
	\eta_c(x) = \frac{2-c^2}{2} \sech^2\left(\frac{\sqrt{2-c^2}}{2}x\right), 
	\quad \text{and}\quad 
\theta_c(x)=\arctan
\Bigg(
\frac{\sqrt{2-c^2}}{c}
\tanh\Bigg( \frac{\sqrt{2-c^2}}{2} x
\Bigg)
\Bigg).
\end{equation}

The main purpose of this article is to provide a numerical scheme to obtain approximate solutions to
\eqref{eq:TWc} for some $c>0$ and to implement this method to provide actual numerical
experiments providing us with families of such dark solitons for several physically relevant
potentials $\mathcal W$.
To do so, we introduce in Section \ref{sec:examples} six families of physically relevant
potentials $\mathcal W_\lambda$.
We also analyze the corresponding dispersion relations, sonic speeds and Landau speeds,
and also  discuss the existence of roton minima.
Section \ref{sec:stability} is devoted to a stability criterion which
ensures orbital stability of dark solitons.
We describe the numerical method we propose in Section \ref{sec:discr} for the computation
of approximation of dark solitons ({\it i.e.} solutions to \eqref{eq:TWc}).
This method uses the equivalent equation \eqref{eq:zeroJ} in $\eta$,
which is first reduced to a bounded interval with homogeneous Dirichlet boundary conditions
thanks to the decay properties (Lemma \ref{lemma:regularity}), and then discretized
by finite differences.
The corresponding discrete analog to \eqref{eq:zeroJ} is then solved numerically
by minimizing a residue by gradient descent.
This allows us to obtain numerical results in Section \ref{sec:num} for the
six families $\W_\lambda$ of nonlocal potentials introduced in Section \ref{sec:examples}.
In particular, this gives insight into the profiles of dark solitons to \eqref{eq:TWc},
and allows for an investigation of the role of the Landau speed,
and the stability of solitons (in connection with Section \ref{sec:stability}).


\section{Examples of interaction potentials}
\label{sec:examples}
We consider six examples of interacting potentials $\W_\lambda$ studied in \cite{dLMar2022}.
They fulfill the hypotheses of Theorem~\ref{thm:dLMar} and Corollary~\ref{cor:dLMar}.
More precisely, they are even tempered distributions, with 
$   \widehat\W\in L^\infty(\R)$ and  
$\wh \W$ 	is of class $\mathcal C^2$
in a neighborhood of the origin, satisfying  the normalization condition 
\begin{equation}
	\label{normalization}
	\widehat{\mathcal W}(0)=1.
\end{equation}
In addition, in all our examples, except the last one,
the potentials are parametrized by $\lambda \in \Lambda \subset \R$, such that $\W_0=\delta_0$.
Thus, for $\lambda=0$, some solutions to \eqref{eq:TWc} are explicitly given
by the soliton \eqref{eq:defeta0}, for speed $c\in (0,\sqrt{2})$,
that we will use later to initiate our numerical method.

We use the dispersion curve $\omega$ defined in \eqref{bogo}, and the Landau speed
in \eqref{def:landau}.
When the potential is parametrized by $\lambda\in\Lambda$,
we denote for simplicity $\omega_\lambda$ the corresponding dispersion curve and
$$c_L(\lambda)=c_L(\W_\lambda).$$

We remark that the sign of $\omega''(\xi)$, for small positive $\xi$
corresponds to the sign of $1+(\wh \W)''(\xi)$.
Therefore, if $(\wh \W)''(0)>-1$, then  the curve $\omega$ is  convex close to the origin,
and lies above the tangent  $\sqrt 2\xi $, for small $\xi>0$.  
If $(\wh \W)''(0)<-1$, then  the curve $\omega$ is  concave close to the origin,
and lies below the tangent  $\sqrt 2\xi$, for small $\xi>0$.
See for instance the potentials in Figure~\ref{fig:dispersion}.



\medskip

\noindent{\bf Example 1}. 
Let $\beta>0$  and $\lambda \in (-\infty,\beta/2)$, we consider 
\begin{equation}
	\label{eq:potSalvador}
	\tag{E1}
	{\mathcal W}_\lambda (x)
	=\frac{\beta}{\beta-2\lambda} \left(\delta_0 - \lambda {\rm e}^{-\beta |x|}\right), \ x\in\R,
	 \quad \text {i.e.\  }\quad
	 	\widehat{\mathcal W}_\lambda(\xi)
	 = \frac{\beta}{\beta-2\lambda}\Big(1-\frac{2\lambda\beta}{\xi^2+\beta^2}\Big), \ \xi\in\R.
\end{equation}
This kind of potential has been used in \cite{nikolov2004}
for the study of dark solitons in a self-defocusing nonlocal Kerr-like medium.
If $\lambda>0$, the potential $\W_\lambda$ represents a strong repulsive interaction between particles
that coincide in space, while the interaction becomes attractive otherwise,
being this attraction more significant at short distances.
In contrast, for $\lambda<0$, the potential $\W_\lambda$ is purely repulsive.
According to Theorem~\ref{thm:dLMar} and Corollary~\ref{cor:dLMar}, for this potential,
we distinguish the following cases:
\begin{enumerate}
	\item If $\beta\geq\sqrt{2}$, then there is a soliton for a.e.\ $c\in (0,\sqrt{2})$ and for every $\lambda\in (-\infty,\beta/2)$.
	\item If $\beta\in (0,\sqrt{2})$, then there is a soliton for a.e.\ $c\in (0,c_L(\lambda))$ and for every $\lambda\in (-\infty,\beta/2)$, where $c_L(\lambda)=\sqrt{2\sigma(\lambda)}$ and 
          \begin{equation}
            \label{eq:defsigma}
	\sigma(\lambda)=\begin{cases}
		1 & \text{ if }-\frac{\beta^3}{2(2-\beta^2)}<\lambda<\frac{\beta}{2},
		\\
		\displaystyle\frac{\beta(1+\sqrt{-\lambda(\beta-2\lambda)})}{\beta-2\lambda}+\beta\sqrt{\frac{-\lambda}{\beta-2\lambda}}-\frac{\beta^2}{2} & \text{ if } \lambda\leq-\frac{\beta^3}{2(2-\beta^2)}.
\end{cases}
\end{equation}
\end{enumerate}
In particular, in this second case, as a function of $\lambda$, $c_L$ is continuous,
positive and  $\lim_{\lambda\to-\infty}c_L(\lambda)=(\beta(2\sqrt{2}-\beta))^{1/2}$.
Notice that the critical value $-{\beta^3}/({2(2-\beta^2)})$, corresponds to the zero of the function $1+(\wh \W_\lambda)''$.

In Figure~\ref{dispersion:E1}, we depict the dispersion curve $\omega_\lambda$
for two values of $(\beta,\lambda)$.
In the left panel, we have  $\beta=0.15$ and $\lambda=0.05$, so that $\wh\W_\lambda''(0)>-1 $,
the function $\omega$ is convex and its graph lies above that of the tangent at the origin
$\xi\mapsto \sqrt 2\xi$. There is no roton minimum and the Landau speed is $c_L=\sqrt 2$.
In the center panel, we have  $\beta=0.5$ and $\lambda=-1$, so that $\wh\W_\lambda''(0)<-1 $,
and $\omega_\lambda$ is concave near the origin, and is below the tangent $\sqrt 2\xi$,
for $\xi\lesssim 1$. In this case, the Landau speed $c_L\sim 1.19$ is less than $\sqrt 2$
and there is a roton for $\omega_\lambda$.
For the sake of clarity, the right panel shows the curve $\omega_\lambda(\xi)/\xi$
for $\beta=0.5$ and $\lambda=-1$.

\begin{figure}[ht!]
\begin{tabular}{ccc}
	\resizebox{0.33\textwidth}{!}{
		\begin{overpic}
       [scale=0.6,trim=60 500 300 50,clip]{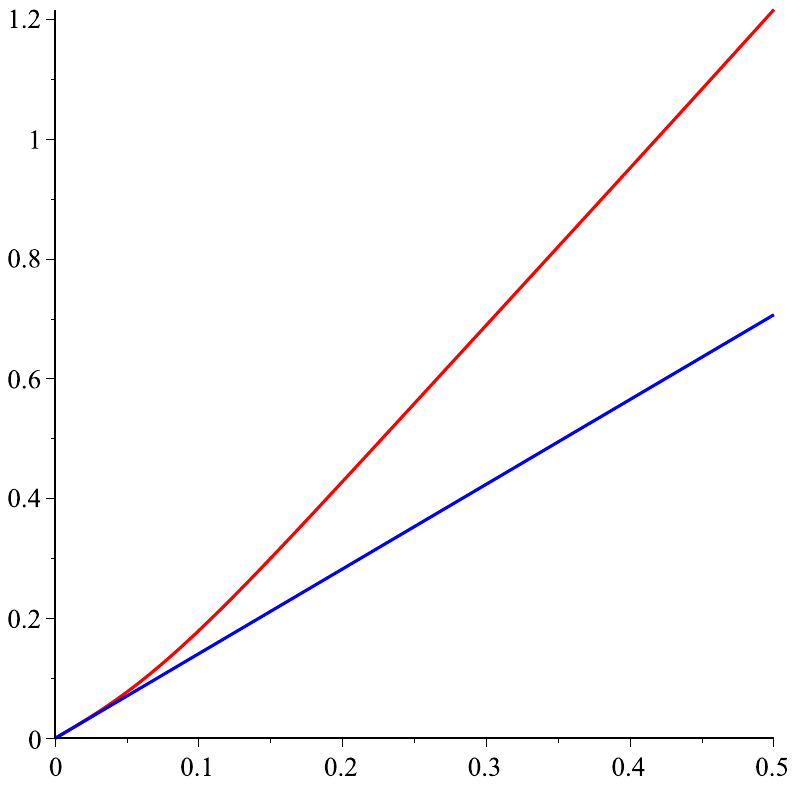}
       	\put(50,-1){{$\xi$}}
       	\put(78,86){{$\color{red}\omega_\lambda$}}
       	 \put(80,60){{$\color{blue}\sqrt 2 \xi$}}
			\end{overpic}
		}
				&
					\resizebox{0.33\textwidth}{!}{
						 	\begin{overpic}
			[scale=0.6,trim=50 500 300 50,clip]{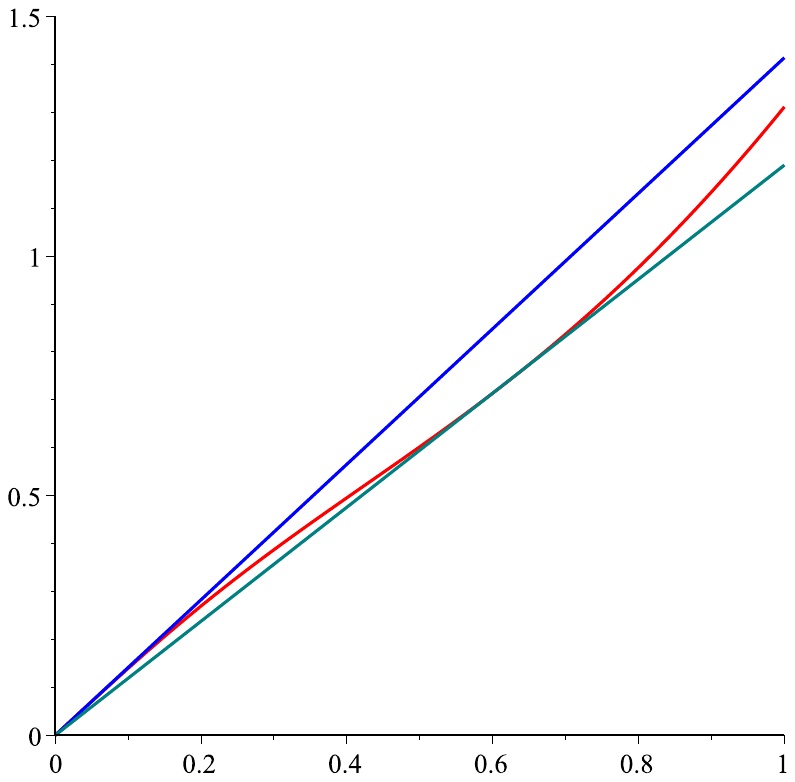}
	       	\put(50,-1){{$\xi$}}
	       	\put(94,72){{$\color{red}\omega_\lambda$}}
			\put(78,80){{$\color{blue}\sqrt 2 \xi$}}
			\put(90,60){{$\color{teal} c_L \xi$}}
				\end{overpic}
			}
			& 
				\resizebox{0.33\textwidth}{!}{
						\begin{overpic}
			[scale=0.6,trim=50 500 300 50,clip]{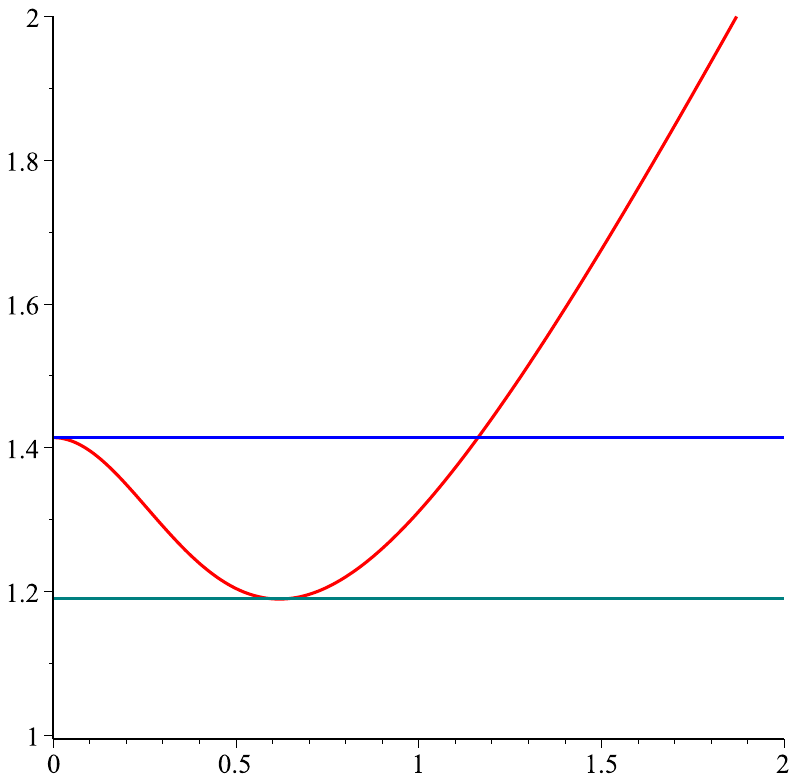}
				       	\put(50,-1){{$\xi$}}
			\put(87,80){{$\color{red}\omega_\lambda(\xi)/\xi$}}
			\put(85,45){{$\color{blue}\sqrt 2 $}}
			\put(85,18){{$\color{teal} c_L$}}
						\end{overpic}
		}
\end{tabular}
\caption{Dispersion curves $\xi\mapsto \omega_\lambda(\xi)$ for potential \eqref{eq:potSalvador}.
  Left: $\beta=0.15$ and $\lambda=0.05$.
  Center and right: $\beta=0.5$ and $\lambda=-1$, so that $c_L \sim 1.19$.}
\label{dispersion:E1}
\end{figure}

\noindent{\bf Example 2}. 
	Another interesting example proposed in \cite{Lopez-Aguayo} is the Gaussian function, for $\lambda\in\R\setminus\{0\}$,
\begin{equation}
  \label{eq:potGaussien}
  \tag{E2}
  {\mathcal W}_\lambda(x) = \frac{1}{2|\lambda|\sqrt{\pi}} {\rm e}^{-\frac{x^2}{4\lambda^2}},\ x\in\R,
	 \quad \text {i.e.\  }\quad
    \widehat{\mathcal W_\lambda}(\xi)
    = {\rm e}^{-\lambda^2\xi^2},\  \xi\in\R,
\end{equation}
with $\mathcal W_0=\delta_0$ and $\widehat{\mathcal W}_0\equiv 1$.
This potential is a classical model for a smooth approximation of the Dirac delta function, as 
$\lambda$ goes to zero.
The function $ 1+  (\widehat{\mathcal W}_\lambda)''$ vanishes at $\lambda^*=\sqrt{2}/2$.
According to Corollary~\ref{cor:dLMar}, for $\lambda \in[0,1/\sqrt{2})$,
we have existence for almost every $c\in(0,\sqrt 2)$.
Also, for $\lambda\geq \sqrt{2}/2$, we have existence of solitons for almost every 
$c\in (0,c_L(\lambda))$,  where 
\begin{equation}
	\label{sigma:lambda}
	c_L(\lambda)=\frac1{\lambda}{\sqrt{1+\ln(2\lambda^2)}}.
\end{equation} 

In Figure~\ref{dispersion:E2}, we depict the dispersion curve $\omega_\lambda$ for two values
of $\lambda$.
In the left and center panels, we have  $\lambda=1$, there is a roton minimum and $c_L\sim 1.3$,
so the Landau speed is very close to the speed of sound.
In the right panel, we have  $\lambda=3$, there is also a roton minimum and   $c_L\sim 0.66$.

\begin{figure}[ht!]
	\begin{tabular}{ccc}
		\resizebox{0.33\textwidth}{!}{
			\begin{overpic}
				[scale=0.6,trim=60 500 300 50,clip]{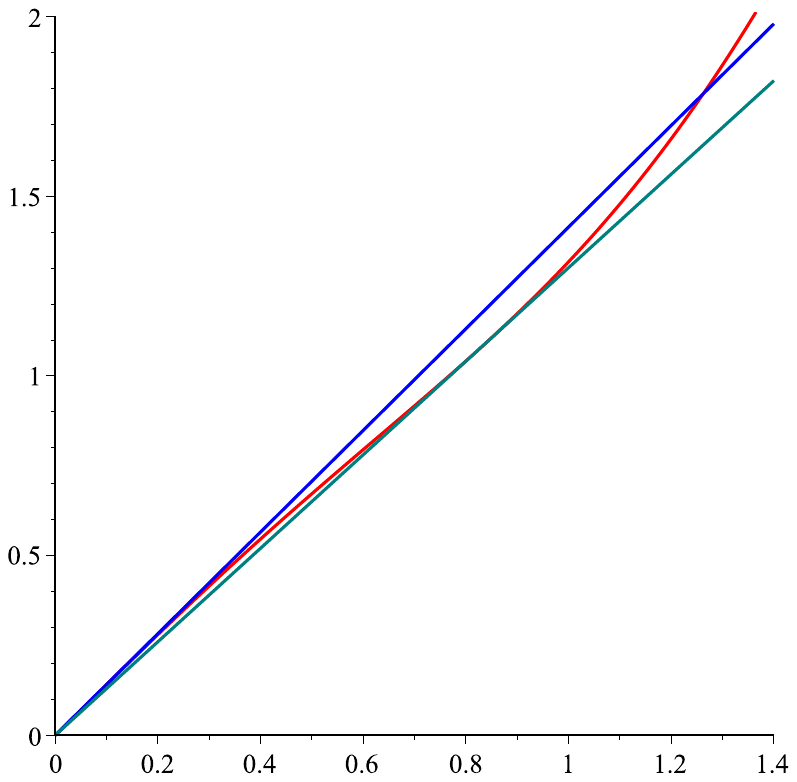}
				\put(50,-2){{$\xi$}}
				\put(85,94){{$\color{red}\omega_\lambda$}}
				\put(58,70){{$\color{blue}\sqrt 2 \xi$}}
				\put(90,75){{$\color{teal} c_L \xi$}}
			\end{overpic}
		}&\resizebox{0.33\textwidth}{!}{
			\begin{overpic}
				[scale=0.6,trim=50 500 300 50,clip]{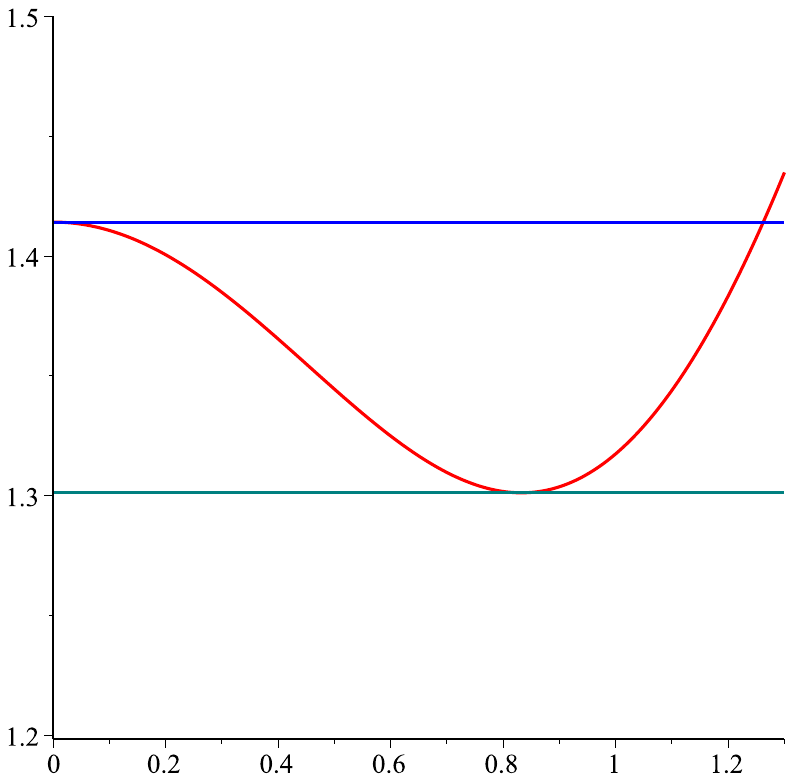}
				\put(50,-2){{$\xi$}}
				\put(88,76){{$\color{red}\omega_\lambda(\xi)/\xi$}}
				\put(25,70){{$\color{blue}\sqrt 2 $}}
				\put(20,38){{$\color{teal} c_L$}}
			\end{overpic}
		}& 	\resizebox{0.33\textwidth}{!}{
			\begin{overpic}
				[scale=0.6,trim=50 500 300 50,clip]{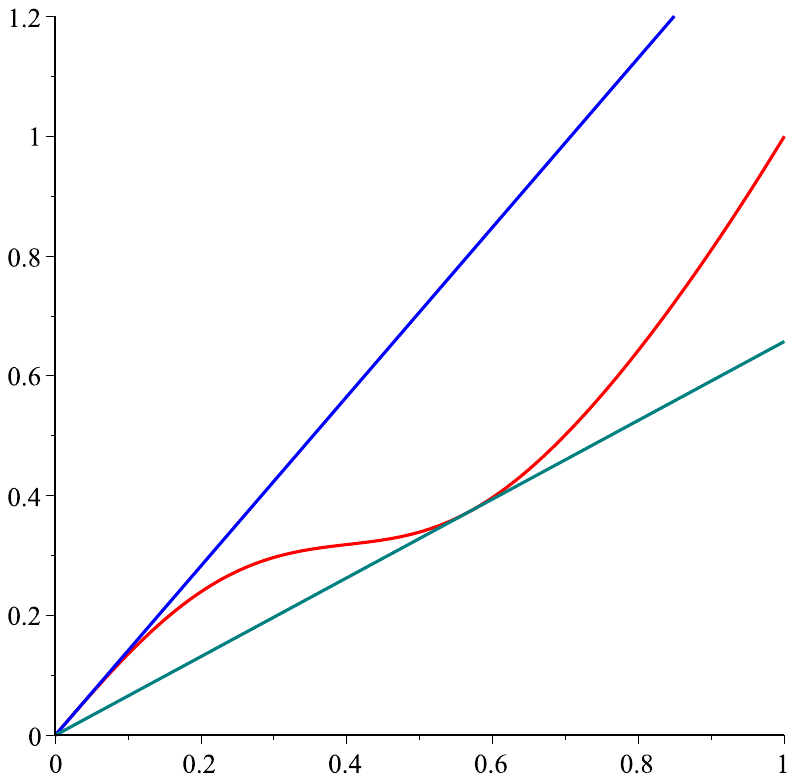}
				\put(50,-2){{$\xi$}}
				\put(88,76){{$\color{red}\omega_\lambda$}}
				\put(62,80){{$\color{blue}\sqrt 2 \xi$}}
				\put(90,55){{$\color{teal} c_L \xi$}}
			\end{overpic}		}
	\end{tabular}
	\caption{Dispersion curves $\xi\mapsto \omega_\lambda(\xi)$ for potential 
	\eqref{eq:potGaussien}.
          Left and center: $\lambda=1$ so that $c_L\sim 1.3$.
          Right: $\lambda=3$ so that $c_L \sim 0.66$.}
	\label{dispersion:E2}
\end{figure}

\noindent{\bf Example 3}.
The next potential, so-called rectangular potential, was used in \cite{aftalion-blanc}
to study supersolids, and also in \cite{Kong-Wang-10} as a model in nonlocal materials.
In addition, it can be seen as a nonsmooth approximation of the Dirac delta when $\lambda>0$ is small,
\begin{equation}
  \label{eq:potunif}
  \tag{E3}
  {\mathcal W_\lambda}(x) = 
  \left\{
    \begin{matrix}
      \frac{1}{2|\lambda|} & \quad \text{if } |x|\leq |\lambda|\\
      0 & \quad \text{otherwise},
    \end{matrix}
  \right.
  \ x\in\R, 
  \quad \text {i.e.\  }\quad
  \widehat{\mathcal W}_\lambda (\xi)= {\rm sinc}( \lambda\xi )=\frac{\sin(\lambda\xi)}{\lambda\xi}, \ \xi\in\R.
\end{equation}
In this case, as a function of $\lambda$, the function $ 1+  (\widehat{\mathcal W}_\lambda)''(0)$
vanishes at $\lambda^*=\sqrt 3$.
Moreover, the Landau speed $c_L(\lambda)$ is less than $\sqrt 2$ for $\lambda>\sqrt 3$.
Note that $\lambda=2$ yields a Landau speed of $c_L(2)\sim 1.374$,
and $\lambda=4.5$, yields a Landau speed of $c_L(4.5) \sim 0.624$.
In Figure~\ref{dispersion:E3}, we show the respective dispersion curves:
In the left and center panels, we see that for $\lambda=2$, there is a roton minimum.
In the right panel, we  see a roton minimum for $\lambda=4.5$.

\begin{figure}[ht!]
	\begin{tabular}{ccc}
		\resizebox{0.30\textwidth}{!}{
			\begin{overpic}
				[scale=0.6,trim=60 400 250 50,clip]{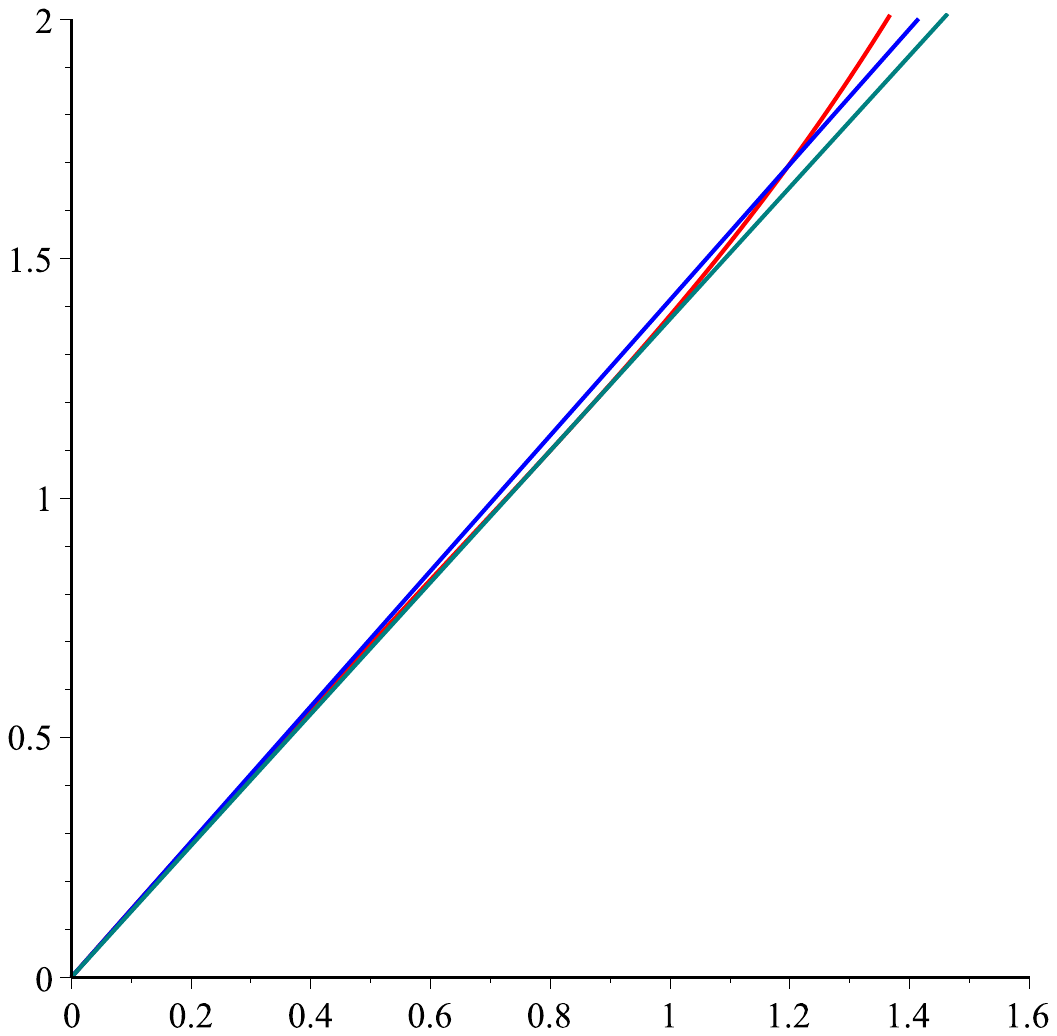}
				\put(50,5){{$\xi$}}
				\put(74,98){{$\color{red}\omega_\lambda$}}
				\put(80,98){{$\color{blue}\sqrt 2 \xi$}}
				\put(80,89){{$\color{teal} c_L\xi $}}
			\end{overpic}
		}&\resizebox{0.36\textwidth}{!}{
			\begin{overpic}
				[scale=0.6,trim=55 486 290 40,clip]{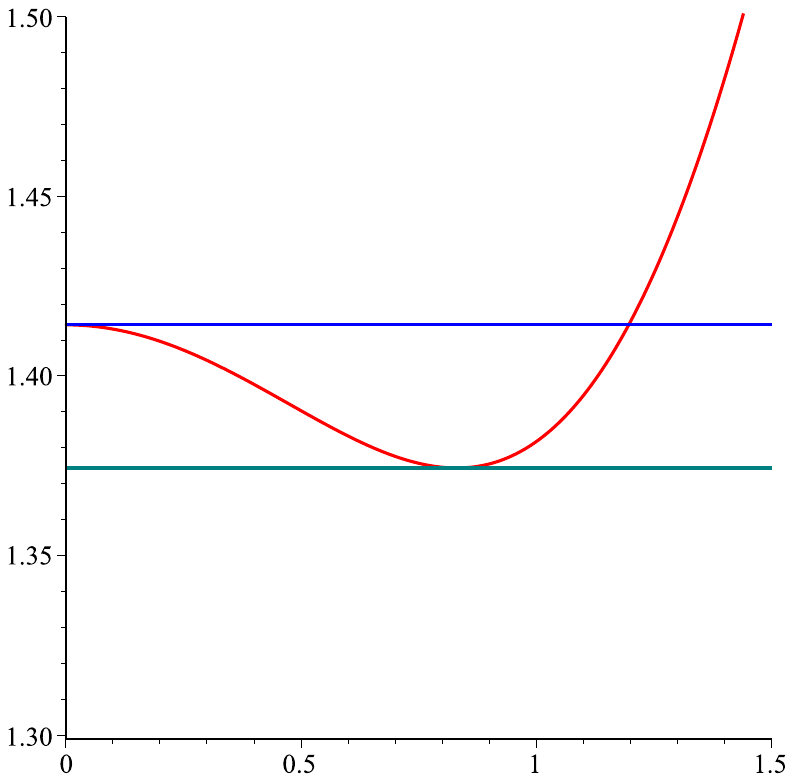}
				\put(50,5){{$\xi$}}
				\put(86,80){{$\color{red}\omega_\lambda(\xi)/\xi$}}
				\put(50,60){{$\color{blue}\sqrt 2 $}}
				\put(20,44){{$\color{teal} c_L$}}
			\end{overpic}
		}& 	\resizebox{0.33\textwidth}{!}{
			\begin{overpic}
				[scale=0.6,trim=54 400 230 100,clip]{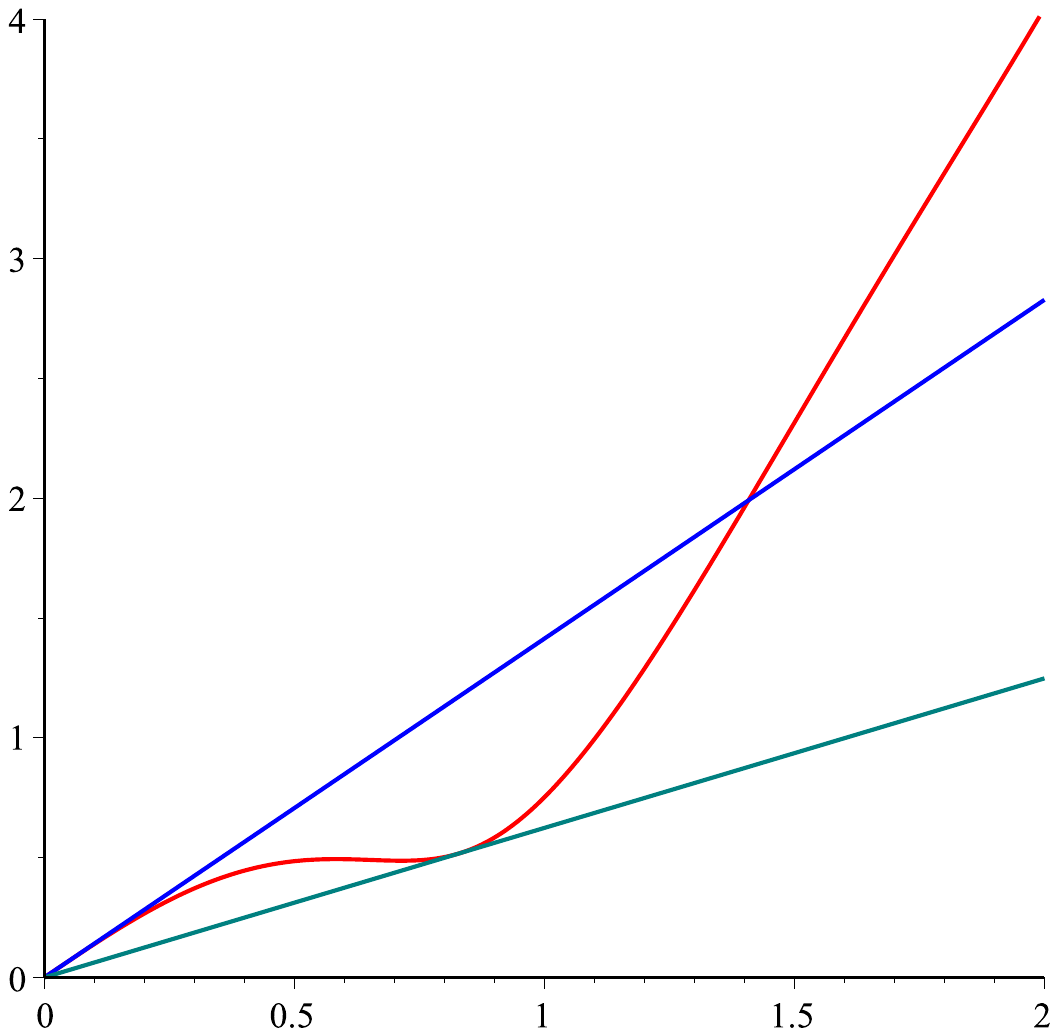}
								\put(50,5){{$\xi$}}
				\put(78,80){{$\color{red}\omega_\lambda$}}
				\put(82,72){{$\color{blue}\sqrt 2 \xi$}}
				\put(80,41){{$\color{teal} c_L \xi$}}
		\end{overpic}		}
	\end{tabular}
	\caption{Dispersion curves $\xi\mapsto \omega_\lambda(\xi)$ for potential \eqref{eq:potunif}.
          Left and center: $\lambda=2$, so that $c_L\sim 1.374$.
          Right: $\lambda=4.5$, so that $c_L \sim 0.624$.}
	\label{dispersion:E3}
\end{figure}

\noindent{\bf Example 4}.
The following potential was proposed in \cite{veskler2014} as a simple model for interactions
in a Bose--Einstein condensate.
It is given  by a contact interaction $\delta_0$ and two Dirac delta functions
centered at $\pm \lambda$, as
\begin{equation}
  \label{eq:pot3Dirac}
  \tag{E4}
  {\mathcal W_\lambda} = 2 \delta_0 - \frac12\left(\delta_{-\lambda} + \delta_{\lambda}\right),
  \quad \text {i.e.\  }\quad
  \widehat{\mathcal W}_\lambda (\xi)=2 - \cos(\lambda\xi), \ \xi \in\R.
\end{equation}
For any $\lambda\geq 0$, we have $c_L(\lambda)=\sqrt 2$ and
Corollary~\ref{cor:dLMar} provides the existence of solitons for almost every speed $c\in (0,\sqrt 2)$. Notice the dispersion curve $\omega_\lambda$
can have inflection points as seen in Figure~\ref{dispersion:E4}. However,
the dispersion curve has no roton minimum, for any $\lambda\geq 0$.
 
\begin{figure}[ht!]
	\begin{tabular}{cc}
		\resizebox{0.36\textwidth}{!}{
			\begin{overpic}
				[scale=0.6,trim=60 400 230 50,clip]{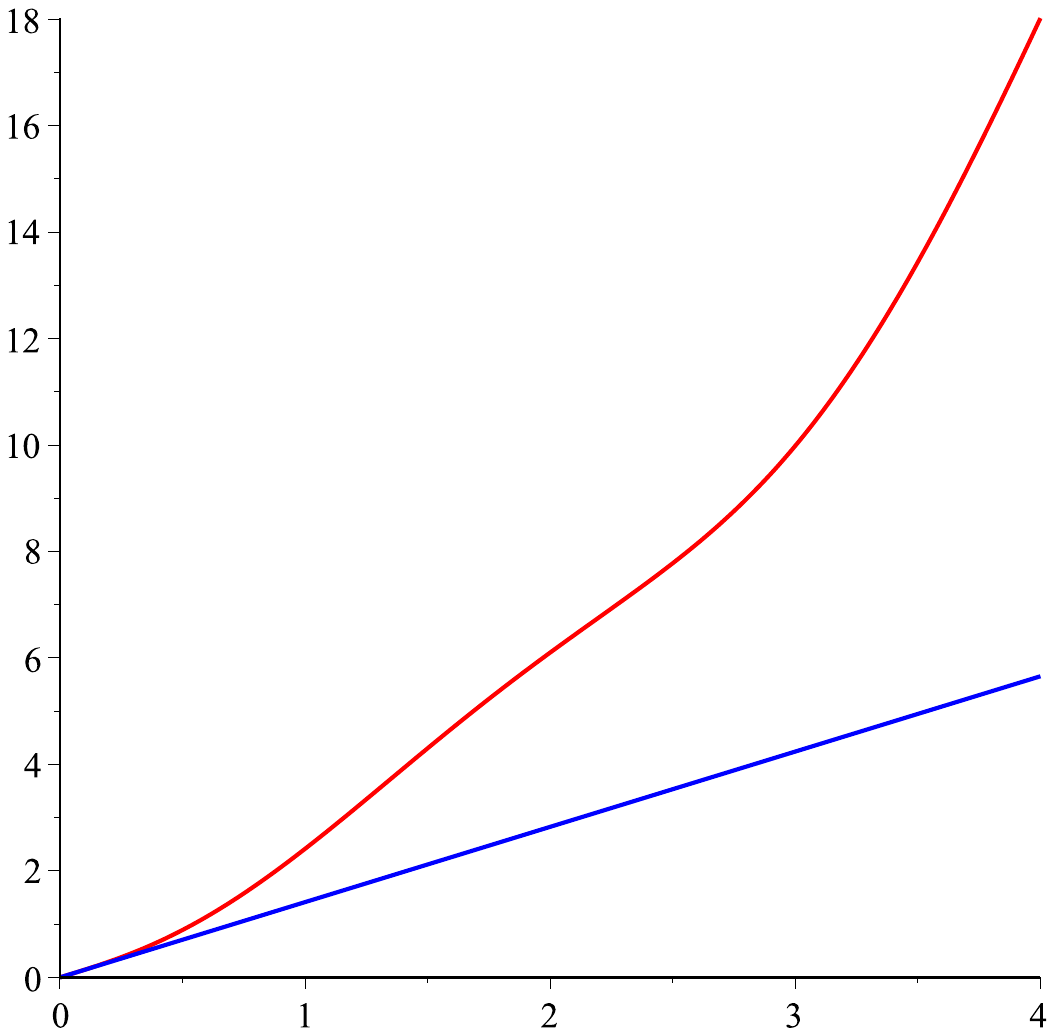}
				\put(50,5){{$\xi$}}
				\put(80,92){{$\color{red}\omega_\lambda$}}
				\put(80,41){{$\color{blue}\sqrt 2 \xi$}}
			\end{overpic}
		}& 	\resizebox{0.36\textwidth}{!}{
			\begin{overpic}
				[scale=0.6,trim=50 400 230 50,clip]{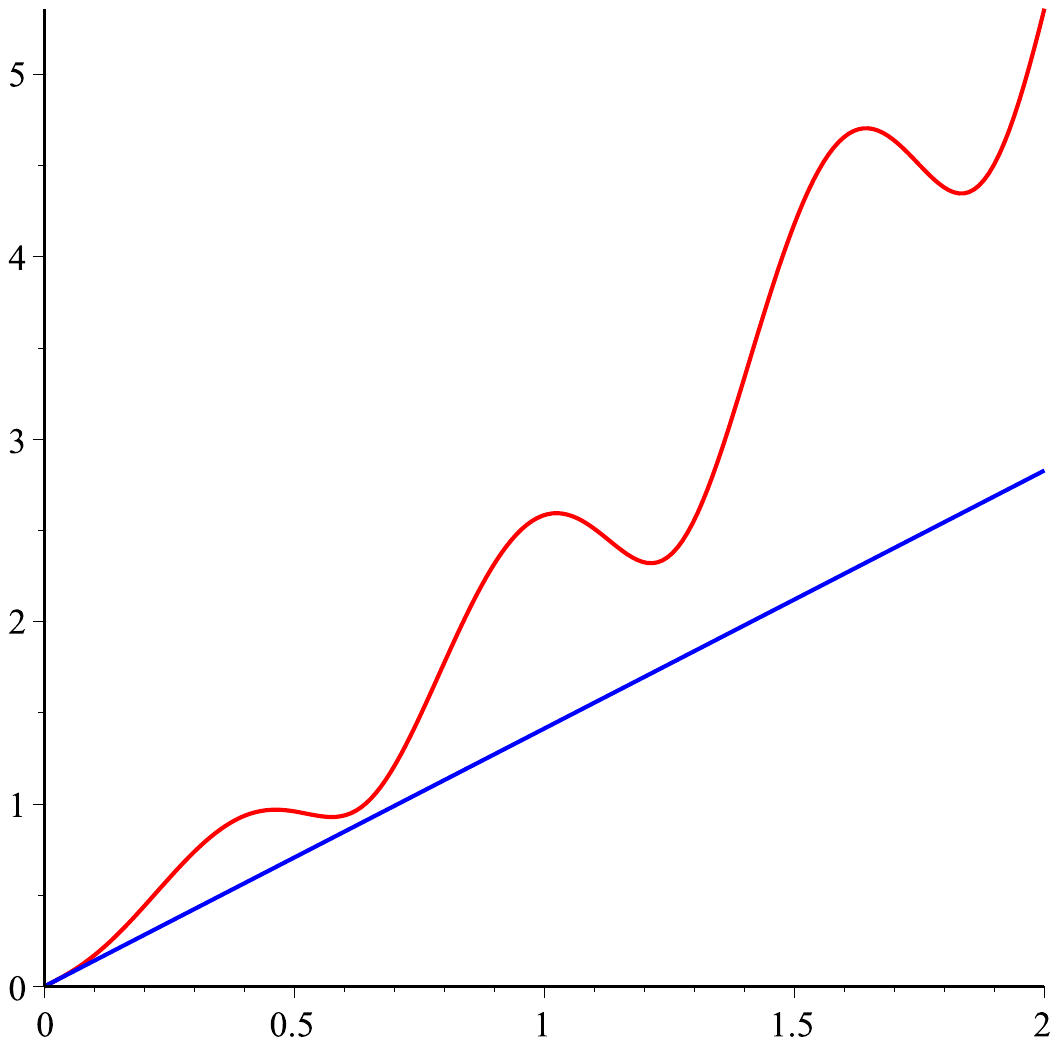}
					\put(50,5){{$\xi$}}
				\put(80,94){{$\color{red}\omega_\lambda$}}
				\put(80,58){{$\color{blue}\sqrt 2 \xi$}}
		\end{overpic}		}
	\end{tabular}
	\caption{Dispersion curves $\xi\mapsto \omega_\lambda(\xi)$ for potential \eqref{eq:pot3Dirac}.
          Left: $\lambda=2$. Right: $\lambda=10$.}
	\label{dispersion:E4}
\end{figure}


\noindent{\bf Example 5}. 
As pointed out in \cite{krolikowski2000} in the context of solitons in nonlocal media, 
when the response function is narrow compared to the extent of the beam, 
it is possible to use the Taylor expansion 
$$\wh \W(\xi)\approx \wh \W(0)+\xi (\wh \W(0))'+\frac{\xi^2}{2}(\wh \W(0))''=1-\lambda \frac{\xi^2}{2},\quad \text{ with  }\lambda=\int_\R x^2\W(x)dx. $$
Here  we used \eqref{normalization}, and the fact that $\W$ is even.
In this case, $\wh \W$ is not bounded and it leads to a quasilinear equation.
In order to have a truly nonlocal model, we modify it by considering the 
the Bochner--Riesz potential  given by
\begin{equation}
	  \label{eq:potBR}
	\tag{E5}
	\widehat{\mathcal W}_\lambda (\xi)=\Big(1-\lambda\frac{\xi^2}{2}\Big)^+, \ 
\  \xi\in\R.
\end{equation}
From a mathematical point of view, the  Bochner--Riesz potential is also interesting because is not smooth in the Fourier variable. Even though we do not need to use the potential in the physical variable for our simulations, 
we give its expression to show the slow decay and oscillations at infinity:  
\begin{equation*}
  {\mathcal W_\lambda}(x) = \frac{\sqrt{\lambda}}{\pi x^2} \Big(\sqrt{\lambda}\frac{\sin\big(x\sqrt{2/\lambda}\big)}{x}-\sqrt{2}\cos\big(x\sqrt{2/\lambda}\big)\Big).
\end{equation*}
In Section \ref{sec:num}, we perform numerical simulation for $\lambda=1$ and  $\lambda=4$.
For $\lambda=1$, we get $c_L(1)=\sqrt2$.
This is a critical value, since the dispersion curve coincides with tangent $\sqrt 2\xi$,
for $\xi\in[0,\sqrt 2]$, as seen in the left panel of Figure~\ref{dispersion:E4}.
For $\lambda=4$, the Landau speed is $c_L=\sqrt{2}/2 \sim 0.7$, 
and there is a roton minimum, as seen in the right panel of Figure~\ref{dispersion:E4}.

\begin{figure}[ht!]
	\begin{tabular}{cc}
		\resizebox{0.35\textwidth}{!}{
			\begin{overpic}
				[scale=0.6,trim=60 400 230 50,clip]{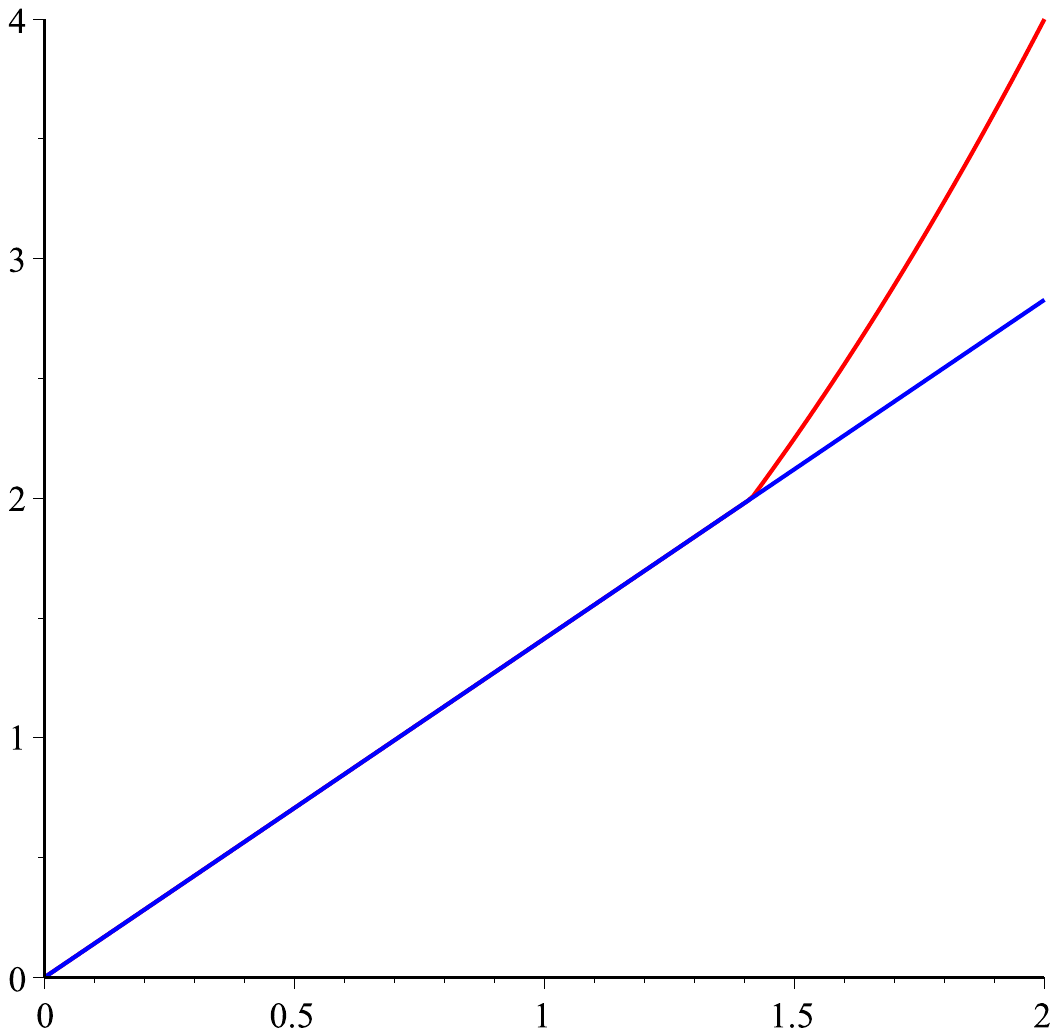}
				\put(50,5){{$\xi$}}
				\put(92,90){{$\color{red}\omega_\lambda$}}
				\put(83,75){{$\color{blue}\sqrt 2 \xi$}}
			\end{overpic}
		}& 	\resizebox{0.35\textwidth}{!}{
			\begin{overpic}
				[scale=0.6,trim=50 400 230 50,clip]{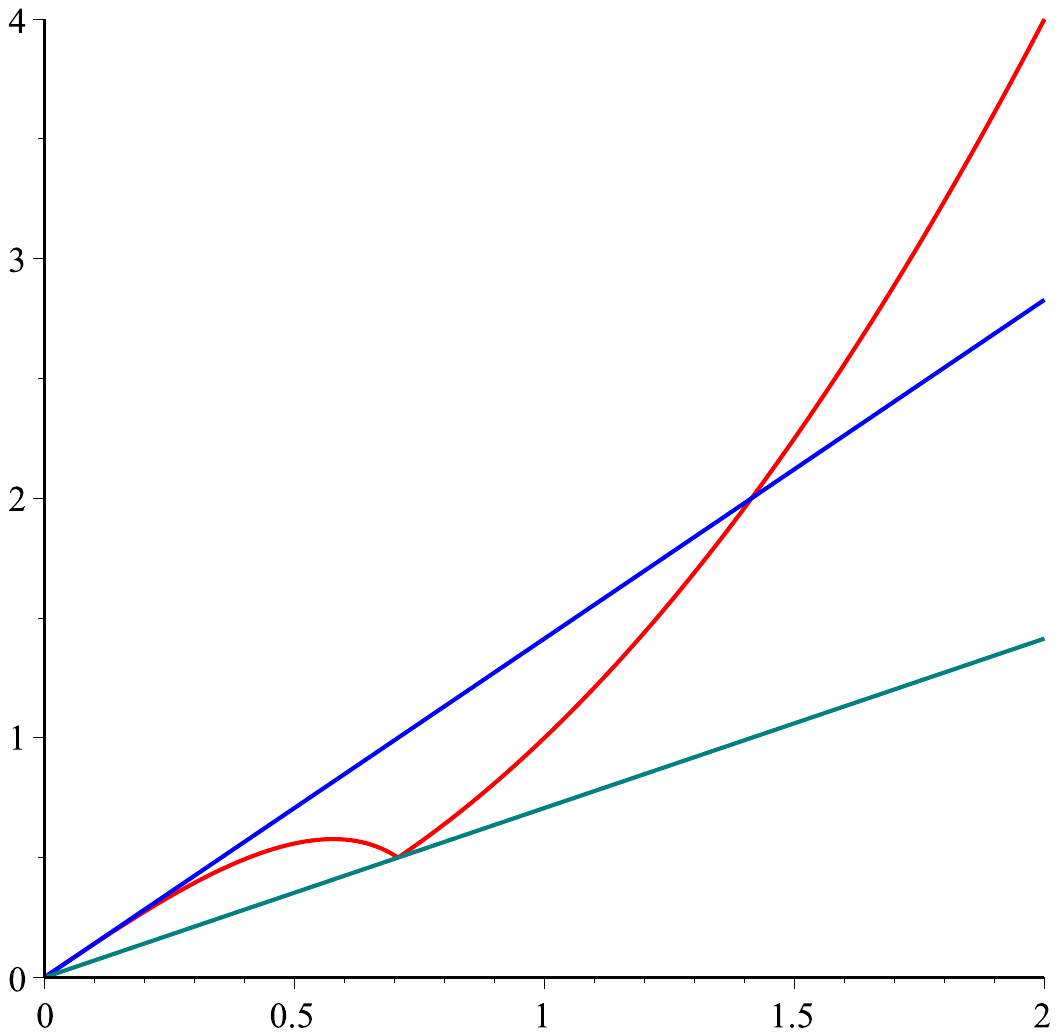}
								\put(50,5){{$\xi$}}
				\put(82,90){{$\color{red}\omega_\lambda$}}
				\put(86,74){{$\color{blue}\sqrt 2 \xi$}}
					\put(90,47){{$\color{teal} c_L \xi$}}
		\end{overpic}		}
	\end{tabular}
	\caption{Dispersion curves $\xi\mapsto \omega_\lambda(\xi)$ for \eqref{eq:potBR}.
          Left: $\lambda=1$. Right: $\lambda=4$, so that $c_L=\sqrt{2}/2$.}
	\label{dispersion:E5}
\end{figure}

\noindent{\bf Example 6}.
For our last example, we consider 
the potential defined by its Fourier transform
\begin{equation}
	\label{eq:potPierre}
	\tag{E6}
	\widehat {\mathcal W}(\xi) =(1+a\xi^2+b\xi^4) {\rm e}^{-c\xi^2}.
\end{equation}
This potential was proposed in \cite{berloff0,reneuve2018}
to describe a quantum fluid exhibiting a roton-maxon spectrum such as Helium~4.
In \cite{delaire-mennuni}, some numerical simulations were done for 
$a=-36$, $b=2687$, $c=30$, and a branch of solitons was found with speeds in $(0,\sqrt 2)$. Therefore, we will consider the same values of $(a,b,c)$, to compare our results with theirs.
Also, the value of $c_L$ is approximately equal to  0.596, so that  Corollary~\ref{cor:dLMar} gives the existence of solitons for almost every speed $c<c_L$.

\begin{figure}[ht!]
		\resizebox{0.35\textwidth}{!}{
				\begin{overpic}
				[scale=0.5,trim=60 400 230 50,clip]{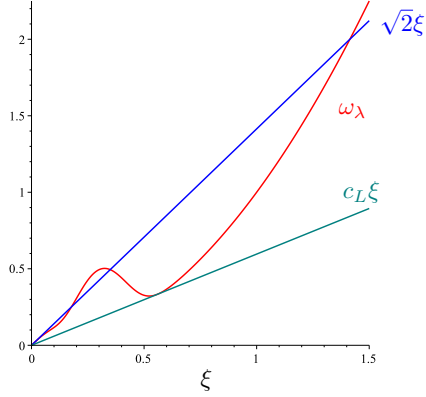}
				\put(50,5){{$\xi$}}
				\put(83,70){{$\color{red}\omega_\lambda$}}
				\put(93,90){{$\color{blue}\sqrt 2 \xi$}}
				\put(84,50){{$\color{teal} c_L \xi$}}
			\end{overpic}  }
		\caption{Dispersion curves for \eqref{eq:potPierre}, with 
$c_L \# 0.596$.}
	\label{dispersion:E6}
\end{figure}


\section{Stability of dark solitons}
\label{sec:stability}

In the case  $\W=\delta_0$, where the dark solitons are explicitly given by the formulas in \eqref{eq:defeta0}, Lin showed in \cite{linbubbles} their orbital stability for $c\in (0,\sqrt 2)$
(see also \cite{chiron-stability}). His proof relied on the general  Grillakis-Shatah-Strauss 
theory \cite{GriShaStra}, that could still be applied 
to the nonlocal evolution equation \eqref{eq:GP},
to determine if the dark solitons are orbitally stable.
We now recall briefly the idea of the method in our context.
Let us fix $\W$ and assume that there exists $\mathfrak{c}>0$
such that for any  $c\in (0,\mathfrak{c})$, there is a solution
$u(c)$ to \eqref{eq:TWc}, and that the local branch 
$c\in (0,\mathfrak{c}) \to u(c)\in \mathcal E(\R) $ is $\mathcal C^1$, so that $c\in (0,\mathfrak{c}) \to \eta(c)\in H^1(\R) $ is $\mathcal C^1$, where $\eta(c)=1-\abs{u(c)}^2.$

Using \eqref{eq:Eeta} and \eqref{eq:Peta}, we can check that the functions $\tilde E(c)=E(\eta(c))$ and  $\tilde  P(c)=P(\eta(c))$ satisfy Hamilton group relation 
\begin{equation}
	\label{group}
\tilde E'(c)=c \tilde P'(c).
\end{equation}

Let us assume also that the  spectral assumptions in the 
Grillakis-Shatah-Strauss theory hold. This is probably  true  because of the mountain-pass theorem used to prove Theorem~\ref{thm:dLMar}. Then,
 the second derivative of the function
$d(c)=\tilde E(c)-c \tilde P(c)$
determines the stability. More precisely, if $d''(c)<0$, the soliton is stable. In view of \eqref{group}, we have 
$d''(c)=-\tilde P(c)$, so that the stability holds if $c\mapsto \tilde P(c)$
 has negative derivative.

In the literature, the energy-momentum $(E,P)$ diagram is preferred
to depict the local branch of solitons,
and the stability is guaranteed by the {\it strict concavity} of the map 
$ P\to  E$, assuming it is of class $\mathcal C^2$, since 
\begin{equation}
	\label{slope}
\frac{d E}{d P}=c,\quad\text{and}\quad 
\frac{d^2 E}{d P^2}=\frac{d }{dP}\left( \frac{dE }{dP} \right)=\frac{d c}{dP}.
\end{equation}}

An alternative method to establish the orbital stability was given in \cite{delaire-mennuni} by using a  minimization approach, assuming that $\W$ is an even tempered distribution with $\wh \W\in L^\infty(\R)\cap\mathcal C^3(\R)$, $\wh \W \geq 0$ on $\R$, with $\wh \W(0)=1$ such that  $\wh \W(\xi)\geq 1-\xi^2/2$, for all  $\abs{\xi}< 1/\sqrt2.$

For $\gq\geq 0$, they consider the  minimization curve
\begin{equation}
	\label{E:min}
	E_\text{min}(\gq):=\inf\{E(v)  :   v\in \boE(\R),\ p(v)=\gq \},
\end{equation}
and  set 
\begin{equation}
	\label{q_*}
	\gq_*=\sup\{\gq>0~|~\forall v \in \mathcal{E}(\mathbb{R}), \  E(v)\leq E_{\min}(\gq)\Rightarrow \underset{\mathbb{R}}\inf|v|>0\}.
\end{equation}

Theorem~4 in \cite{delaire-mennuni} establishes that if
  $\Emin$ is {\it  concave} on $\R^+$, then the set of solutions
of the minimization problem \eqref{E:min} is orbitally stable, for all $\gq\in (0,\gq_*)$. Their proof is based on the Lions-Cazenave argument. When $\W=\delta_0$, using \eqref{sol:1D}, 
it is possible to check directly that $\Emin$ is smooth on $\R^+$, strictly concave on $[0,\pi/2]$, and constant on $(\pi/2,\infty)$, 
as depicted in Figure~\ref{courbeth}.
Moreover, in view of \eqref{slope}, the slope of the tangent to the curve is exactly  $c$, for  $\gq \in (0,\pi/2)$, 
and the curve remains below the critical line $\sqrt{2} \gq.$

Therefore, both criteria explained above can be used to prove that the dark solitons are stable for every $c\in(0,\sqrt 2)$.
The case $c=0$ corresponding to the black soliton is more involved
and is beyond the scope of this work.
One problem is that the momentum \eqref{momentum} is ill-defined for perturbations of the black soliton, thus it is necessary to introduce some renormalized momentum (see \cite{BeGraSaut-stability}). 
   
\begin{figure}
	\begin{center}
		\begin{tabular}{cc}
	\resizebox{0.45\textwidth}{!}{
	\begin{overpic}
	[trim={0.31cm 0.2cm 0.05cm 0cm}, clip]{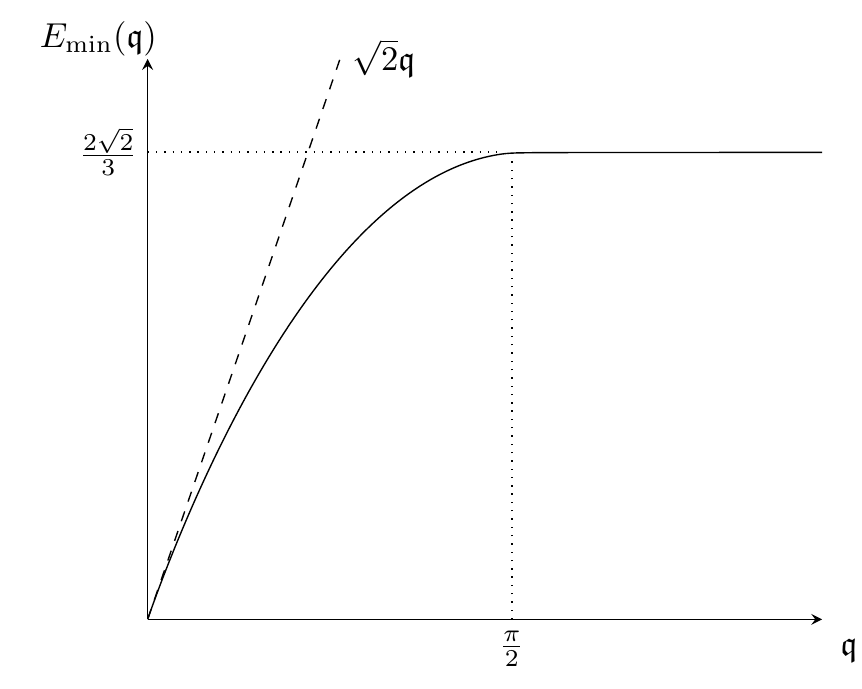}
\end{overpic}
}    
			&
	\resizebox{0.45\textwidth}{!}{
	\begin{overpic}
	[trim={0.7cm 1.1cm 0.3cm 0.2cm}, clip]{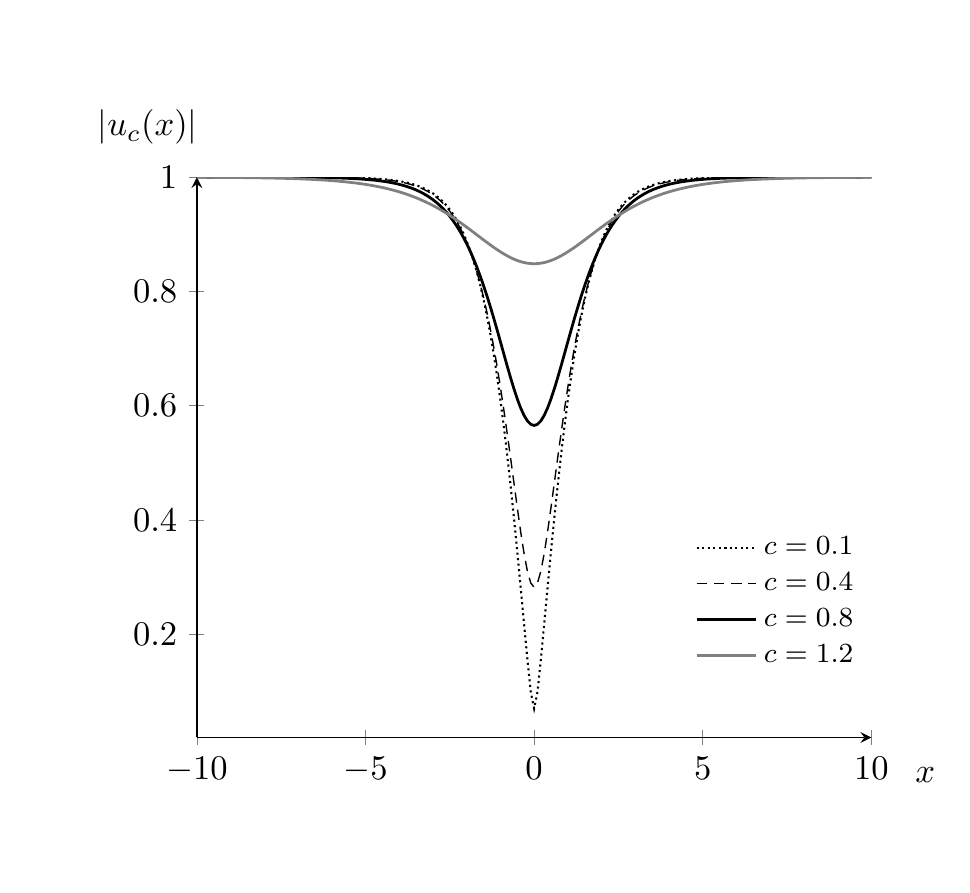}
\end{overpic}
 }
		\end{tabular}
	\end{center}
	\caption{Curve $\Emin$ and solitons in the case $\W=\delta_0$.}
	\label{courbeth}
\end{figure}

Concerning the examples in Section~\ref{sec:examples}, de Laire and Mennuni were able to prove
in \cite{delaire-mennuni} that the minimization curve $\Emin$ is indeed concave
for the potential \eqref{eq:potSalvador}, which implies the stability of the solitons. 
They conjectured that in all these examples we should have  $\gq_*=\pi/2$, 
that $\Emin$ should be strictly concave on $[0,\pi/2]$, and constant on $(\pi/2,\infty)$.
This was corroborated by some numerical simulations in \cite{delaire-mennuni}.

In this paper, we adopt a different approach. 
Indeed, we do not minimize the energy at fixed momentum,
to end up with dark solitons at some speed that need to be computed afterward.
In contrast, we take advantage of the equivalent form \eqref{eq:zeroJ} of \eqref{eq:TWc},
which uses implicitly parts of the results of Lemma \ref{lemma:regularity}, and we solve it
by a simple gradient descent method on its residue, after discretization.
This method is described in Section \ref{sec:discr} and implemented on the six
potentials of Section \ref{sec:examples} in Section \ref{sec:num}.
It has at least two advantages.
First, it allows to us fix the speed of the soliton {\it prior} to the computation.
Second, it allows us to work in a vector space, since the change of unknown from $u$ in 
\eqref{eq:TWc}
to $\eta$ in \eqref{eq:zeroJ} took away the inhomogeneous boundary condition.

\section{The discretized problem}
\label{sec:discr}

\subsection{The discretization of \eqref{eq:zeroJ}}
\label{subsec:discr}

For the discretization of $J_c(\cdot,\lambda)$ defined in \eqref{eq:defJc},
we fix some $L>0$ and we work on the bounded closed interval $[-L/2,L/2]$.
We introduce a number $N\geq 1$ of interior points,
and set $\delta x=L/(N+1)$ for $k\in\{0,\dots,N+1\}$, $x_k=k\delta x-L/2$.
We define the matrix corresponding to the centered first derivative
with homogeneous Dirichlet boundary conditions as
\begin{equation}
  \label{eq:defD}
  D=\frac{1}{2\delta x}
  \begin{pmatrix}
    0      & 1      & 0      & \hdots & 0      \\
    -1     & 0      & 1      & \ddots & \vdots \\
    0      & -1     & 0      & \ddots & 0      \\
    \vdots & \ddots & \ddots & \ddots & 1      \\
    0      & \hdots & 0      & -1     & 0      \\
  \end{pmatrix},
\end{equation}
so that  $D\in\mathcal M_N(\R)$ is skew-symmetric.
We also define the matrix corresponding to the classical second order derivative
with homogeneous Dirichlet boundary conditions as
\begin{equation}
  \label{eq:defA}
 A=\frac{1}{\delta x^2}
  \begin{pmatrix}
    -2     & 1      & 0      & \hdots & 0      \\
    1      & -2     & 1      & \ddots & \vdots \\
    0      &  1     & -2     & \ddots & 0      \\
    \vdots & \ddots & \ddots & \ddots & 1      \\
    0      & \hdots & 0      & 1      & -2     \\
  \end{pmatrix}.
\end{equation}
Of course, $A\in\mathcal M_N(\R)$ is symmetric.
For $\lambda\in\Lambda$, we define the interacting potential vector $V_\lambda\in\R^N$
with the entries $(V_\lambda)_k={\mathcal W_\lambda(x_k)}$ for $k\in\{1,\dots,N\}$,
provided $\W_\lambda$ is a smooth enough function.

Eventually, we define the discretized analogue of $J_c$ defined in \eqref{eq:defJc}
for $\bm\eta \in\R^N$ and $\lambda\in\Lambda$ as
\begin{equation}
  \label{eq:defJnum}
  J^{\delta x}_c(\bm\eta,\lambda) =  A\bm\eta-2V_\lambda\ast \bm\eta + c^2\bm\eta
  + c^2\bm \eta.\bm \eta ./ (2({\mathds 1}-\bm\eta))
  +(D\bm\eta. D\bm \eta)./(2({\mathds 1}-\bm\eta))
  + 2( V_\lambda\ast \bm\eta).\bm\eta,
\end{equation}
where $.$ and $./$ stand for componentwise multiplication and division respectively, and $\ast$
stands for the convolution of vectors.
This discrete convolution corresponds to the restriction to $\{1,\cdots,N\}$ of the usual
convolution of the two infinite sequences consisting in its two arguments completed by zeros.
This is consistent with the continuous context, since we expect the first and last components
of both arguments ($V_\lambda$ and ${\bm\eta}$) to be small, provided $L>0$
has been chosen large enough.
We define the numerical momentum of ${\bm\eta}$ a solution to 
$J_c^{\delta x}(\bm\eta,\lambda)=0_{\R^N}$,
by the discrete analogues to \eqref{eq:Peta} and \eqref{eq:Eeta}
\begin{align}
  \label{eq:Petadeltax}
  P^{\delta x}& = \frac{c}{4} \delta x \sum_{k=1}^N  \frac{\bm \eta_k^2}{1-\bm\eta_k},\\
  \label{eq:Eetadeltax}
  E^{\delta x} &= \frac{c^2}{8} \delta x \sum_{k=1}^N \frac{\bm \eta_k^2}{1-\bm \eta_k}
  + \frac{1}{8} \delta x \sum_{k=1}^N \frac{(D\bm\eta)_k^2}{1-\bm\eta_k}
  + \frac{1}{4} \delta x \sum_{k=1}^N\left(V_\lambda\ast\bm\eta\right)_k\bm\eta_k.
\end{align}

\subsection{The minimization algorithm}
\label{subsec:algo}

Since we are interested in nontrivial zeros of the function $J_c$ at fixed $c\in\R$,
we adopt the following strategy:
we minimize at fixed $\lambda\in\Lambda$ and $c\in\R$ the squared euclidean norm
of the vector-valued function ${\bm\eta}\mapsto J_c^{\delta x}(\bm\eta,\lambda)$, starting
from a discretization of the known nontrivial solution \eqref{eq:defeta0} (for $\lambda=0$ and for the given $c$),
using a gradient method with adaptative step.
This requires the introduction of the function
\begin{equation}
  \label{eq:defF}
  F_\lambda^{\delta x}(\bm \eta) = {\delta x}\|J_c^{\delta x}(\bm\eta,\lambda)\|^2.
\end{equation}

Observe that $\bm \eta\mapsto J_c^{\delta x}(\bm\eta,\lambda)$
is a smooth function on some neighborhood of $0$ in $\R^N$, and so is $F_\lambda^{\delta x}$.
Moreover, its partial differential with respect to $\bm\eta$ at some point $(\bm\eta,\lambda)$
is given by
\begin{multline}
  \label{eq:defdiffJ}
  \forall \bm\omega\in\R^N,\qquad
  \partial_{{\bm\eta}} J_c^{\delta x} (\bm \eta,\lambda) (\bm\omega) =  A\bm \omega
  -2 V_\lambda\ast\bm\omega
  +c^2\bm\omega\\
  + \left[(2 c^2 \bm \eta. ({\mathds 1}-\bm\eta/2) + D\bm\eta.D\bm\eta)./(2({\mathds 1-\bm\eta}).({\mathds 1-\bm\eta}))\right].\bm\omega\\
  +\left(D\bm\eta./({\mathds 1}-\bm\eta)\right).(D\bm\omega)
  +2 (V_\lambda\star\bm\eta).\bm\omega
  +2 \bm\eta.(V_\lambda\ast\bm\omega).
\end{multline}
From this, we infer that
\begin{equation}
  \label{eq:defgradF}
  \nabla F_\lambda^{\delta x} (\bm \eta) = 2 {\delta x}\left[\partial_{\bm \eta}J_c^{\delta x}(\bm\eta,\lambda)\right]^t J_c^{\delta x}(\bm\eta,\lambda),
\end{equation}
where $\left[\partial_{\bm \eta}J_c^{\delta x}(\bm\eta,\lambda)\right]^t$ is the adjoint of the
linear mapping $\bm\omega \mapsto \partial_{\bm \eta}J_c^{\delta x}(\bm\eta,\lambda)(\bm \omega)$
from $\R^N$ to itself defined in \eqref{eq:defdiffJ}.

With the notations above, the gradient iteration starting from a given ${\bm \eta_n}\in\R^N$
reads
\begin{equation}
  \label{eq:methgrad}
  {\bm \eta}_{n+1} = \bm\eta_n - h \nabla F_\lambda^{\delta x}({\bm \eta}_n),
\end{equation}
for some $h>0$.
After this step, the number $F_\lambda^{\delta x}(\bm \eta_{n+1})$ is computed and compared to
$F_\lambda^{\delta x}(\bm \eta_n)$. If it is greater, then the step $h$ is divided by $2$
and step \eqref{eq:methgrad} is performed again with the new $h$ and the same $n$.
If not, then one changes $n$ in $n+1$ and performs \eqref{eq:methgrad} again.
The method stops if (at least) one of the three criteria below is fulfilled:
\begin{enumerate}[label=(\roman*)]
\item \label{bulletpoint1}
  $F_\lambda^{\delta x}(\bm \eta_n)$ is smaller than $\varepsilon^2$ for some given $\varepsilon>0$.
  In this case, we consider that the method has converged.
\item $h$ is below some given $h_{\rm min}>0$. In this case, we consider that the method has {\it not}
  converged.
\item $n$ has become greater than some given integer $n_{\rm max}$. In this case, we consider
  that the method has {\it not} converged.
\end{enumerate}

Of course, if the method \eqref{eq:methgrad} converges, then its limit is a critical point
of the function $F^{\delta x}_{\lambda}$.
This point may not be a zero of the function $F^{\delta x}_\lambda$.
This is a reason for the check of the size of $F^{\delta x}_\lambda$ in point \ref{bulletpoint1} above.
Moreover, the convergence of the method is only {\it local} in $\R^N$.
In order to avoid convergence to the trivial critical point of $F^{\delta x}_\lambda$
(that is $\bm\eta\equiv 0$), we initialize the method \eqref{eq:methgrad} with
a projection of the continuous soliton $\eta_c$
defined in \eqref{eq:defeta0} on the grid $x_1,\cdots,x_N$.
Of course, this soliton $\eta_c$ is {\it not} a soliton for a fixed $\mathcal W_\lambda$ in general
(in particular when $\lambda\neq 0$), but it proves itself far away enough from the trivial
soliton to obtain convergence to a nontrivial soliton in the numerical examples presented
in the next section.

Instead of the explicit gradient method \eqref{eq:methgrad} to minimize \eqref{eq:defF},
one could have chosen several other options, such as implicit gradient methods, or Gauss--Newton
methods, for example. In this paper, we focus on \eqref{eq:methgrad} because it is explicit
(in contrast to the other aforementioned methods,
it does {\it not} require to solve a linear or nonlinear system at each step) and it proves
itself stable enough in this context.
Indeed, one would expect to have a CFL-type condition on $h/\delta x ^3$ to ensure stability
of \eqref{eq:methgrad}.
However, for the values of $\delta x$ used in the numerical
experiments of Section \ref{sec:num}, this condition appears less restrictive than the condition
on $h$ implied by the condition on the sign of the difference of the values of $F^{\delta x}_\lambda$
between two successive iterations
({\it i.e.} $F^{\delta x}_\lambda({\bm \eta}_{n+1})-F^{\delta x}_\lambda({\bm \eta}_{n})<0$).

\section{Numerical results}
\label{sec:num}

In this section, we consider the six potentials described in Section~\ref{sec:examples}
and compute numerically nontrivial minimizers $\bm\eta \in\R^N$ of the function $F_\lambda^{\delta x}$
defined in \eqref{eq:defF} for several values of $\lambda$
and $c\in(0,\sqrt{2})$, by using the algorithm described in Section~\ref{subsec:algo}.
In the sequel, the number of interior points $N$ is odd and $M=(N+1)/2$ is therefore
an integer corresponding to the index $k$ of the middle of the vector $\bm\eta$ (it
satisfies $x_M=0$ with the notations of Section~\ref{subsec:discr}).
Using \eqref{def:u}, we recover 
from $\bm\eta$ an approximate solution $\bm u$
of the solution $u$ to \eqref{eq:TWc} by setting
$$\bm\theta_k=
\left\{
  \begin{array}{ll}
    - \displaystyle \frac{c}{2} \delta x \sum_{p=k}^{M-1} \frac12\left(\frac{\bm\eta_{p}}{1-\bm\eta_{p}}+\frac{\bm\eta_{p+1}}{1-\bm\eta_{p+1}}\right) & \text{if } k\leq M-1,\\
    0 & \text{if } k=M,\\
    \displaystyle \frac{c}{2} \delta x \sum_{p=M+1}^k \frac12\left(\frac{\bm\eta_{p-1}}{1-\bm\eta_{p-1}}+\frac{\bm\eta_p}{1-\bm\eta_{p}}\right) & \text{if } k\geq M+1,
  \end{array}
\right.
$$
and
$$\bm{u}_k=\sqrt{1-\bm\eta_k} {\rm e}^{i\bm\theta_k},$$
for $k\in\{1,\cdots,N\}$.

Recall that we initialize the method with $\eta_c$, which corresponds to the solution  to \eqref{eq:TWc} for $\W=\delta_0$ (see \eqref{eq:defeta0}).
One of the difficulties is that, for values of $c$ close to $\sqrt{2}$,
the size of the numerical support of the soliton $\eta$ tends to infinity.

\subsection{Solitons with potential \eqref{eq:potSalvador}}
We consider the potential $\mathcal W_\lambda$ defined in \eqref{eq:potSalvador}
with $\beta=1.0$ and $\lambda=0.4$, so that $\beta/(\beta-2\lambda)=5.0$.
In particular, according to \eqref{eq:defsigma}, we have $c_L(\lambda)=\sqrt2$, 
$\wh \W''(0)>-1$, and the aspect of $\omega$  is the same as the one in the left panel
of Figure~\ref{dispersion:E1}.
The numerical experiments are carried out with $L=50$, $N=999$, so that $\delta x = 0.05$,
and $\varepsilon^2=\delta x/4$.
The parameter $h_{\rm min}$ is so small and the parameter $n_{\rm max}$ is so big that
the method has always converged in these numerical experiments.
The results displayed in Figure \ref{fig:etadec2} allow us to see the shape
of the numerical minimizers $\bm \eta$ 
as functions of $x$, and the phases $\bm\theta$, for several values of $c$.
These results confirm that the support of the minimizer $\bm \eta$ tends to spread as the speed $c$
increases. Moreover, the maximum of ${\bm \eta}$ (corresponding to the minimum of $\bm u$)
tends to decrease as $c$ increases. This is similar to the case of the Dirac potential $\lambda=0$
in this nonlocal case ($\lambda=0.4$).
The results displayed in Figure \ref{fig:Edec2} show the evolution of the momentum $P^{\delta x}$
and the energy $E^{\delta x}$ (see \eqref{eq:Petadeltax} and \eqref{eq:Eetadeltax}) of the numerical
minimizers as functions of the speed $c$, together with the evolution of $E^{\delta x}$ as
a function of $P^{\delta x}$.

Concerning the stability criteria and conjectures explained in Section~\ref{sec:stability},
we see that everything is fulfilled numerically:
the mapping $c\mapsto  P^{\delta x}$ is strictly decreasing,
the curve $ P^{\delta x}\mapsto E^{\delta x}$ is  below and tangent to the line $\sqrt{2} P^{\delta x}$,
and is strictly concave on $(0,\pi/2)$.
Therefore, these facts confirm numerically the stability of the solitons for $c\in (0,\sqrt 2)$.

%
\begin{figure}[!ht]
	\centering
	\begin{tabular}{cc}
		\resizebox{0.5\textwidth}{!}{
			\includegraphics{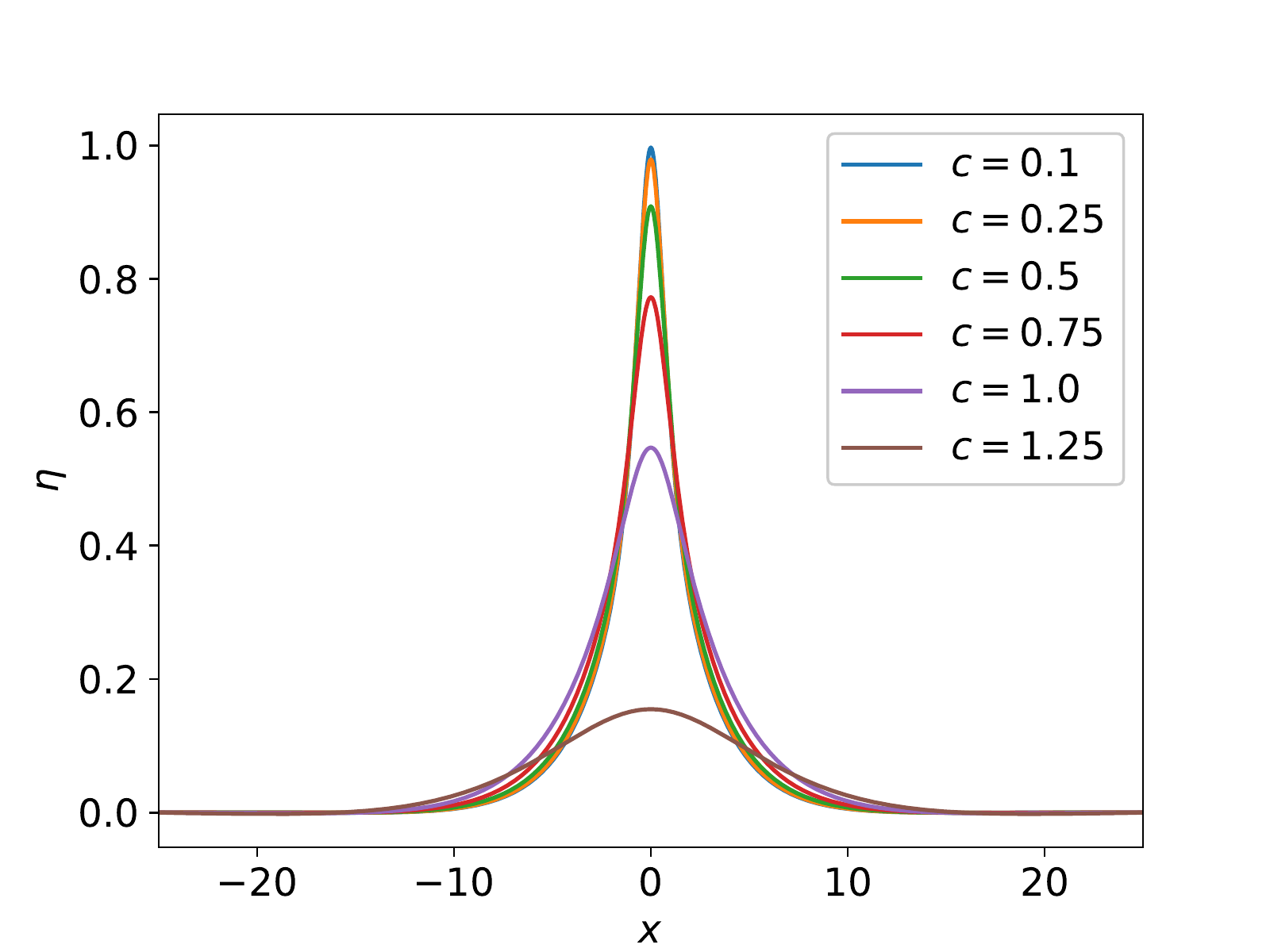}}%
	&
	\resizebox{0.5\textwidth}{!}{\includegraphics{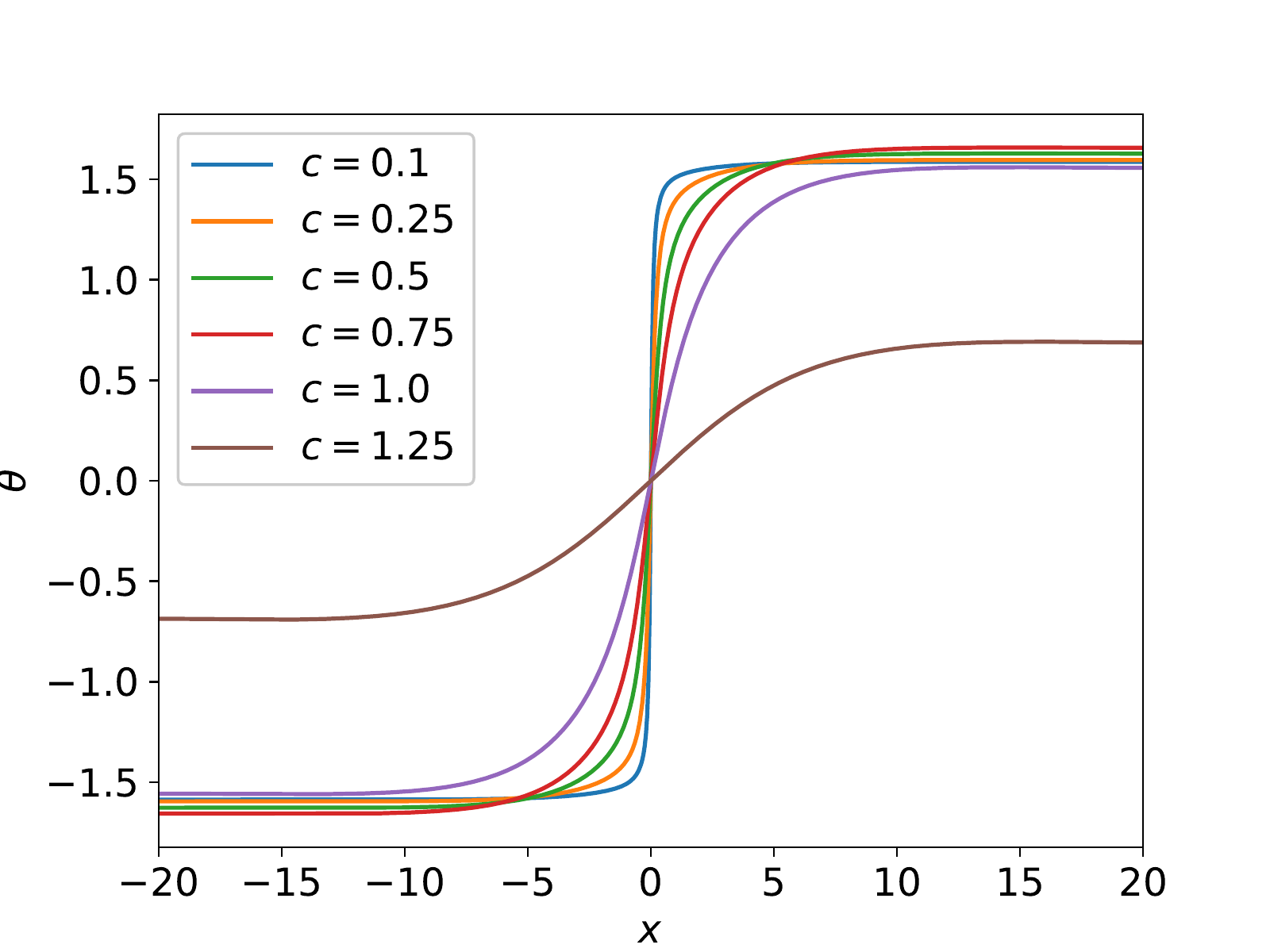}}
\end{tabular}
\caption{Numerically computed solitons for potential \eqref{eq:potSalvador} with $\beta=1.0$ and 
$\lambda=0.4$ showing 
$\bm\eta_k$ (left panel) and ${\bm \theta}_k$ (right panel) as functions of $x_k$.
Numerical parameters are indicated in the text.
}
\label{fig:etadec2}
\end{figure}

\begin{figure}[!ht]
\centering
\begin{tabular}{ccc}
\resizebox{0.33\textwidth}{!}{
\begin{tikzpicture}

\definecolor{color0}{rgb}{0.12156862745098,0.466666666666667,0.705882352941177}

\begin{axis}[
tick align=outside,
tick pos=left,
x grid style={white!69.0196078431373!black},
xlabel={\(\displaystyle c\)},
xmin=0, xmax=1.4142135623731,
xtick style={color=black},
y grid style={white!69.0196078431373!black},
ylabel={\(\displaystyle P^{\delta x}\)},
ymin=0, ymax=1.5,
ytick style={color=black}
]
\addplot [semithick, color0, mark=asterisk, mark size=3, mark options={solid}]
table {%
0.1 1.49010747290778
0.25 1.35421920899306
0.5 1.13442803742674
0.75 0.895799764542702
1 0.57913275076115
1.25 0.0756808096623169
};
\end{axis}

\end{tikzpicture}} &
\resizebox{0.33\textwidth}{!}{
\begin{tikzpicture}

\definecolor{color0}{rgb}{0.12156862745098,0.466666666666667,0.705882352941177}

\begin{axis}[
tick align=outside,
tick pos=left,
x grid style={white!69.0196078431373!black},
xlabel={\(\displaystyle c\)},
xmin=0, xmax=1.4,
xtick style={color=black},
y grid style={white!69.0196078431373!black},
ylabel={\(\displaystyle E^{\delta x}\)},
ymin=0, ymax=1.5,
ytick style={color=black}
]
\addplot [semithick, color0, mark=asterisk, mark size=3, mark options={solid}]
table {%
0.1 1.25293176758973
0.25 1.22895595818063
0.5 1.14208126745016
0.75 0.984816795061364
1 0.696442238303692
1.25 0.104529779338614
};
\end{axis}

\end{tikzpicture}} & \resizebox{0.33\textwidth}{!}{
\begin{tikzpicture}

\begin{axis}[
tick align=outside,
tick pos=left,
x grid style={white!69.0196078431373!black},
xlabel={\(\displaystyle P^{\delta x}\)},
xmin=0, xmax=1.55,
xtick style={color=black},
y grid style={white!69.0196078431373!black},
ylabel={\(\displaystyle E^{\delta x}\)},
ymin=0, ymax=1.55,
ytick style={color=black}
]
\addplot [draw=blue, fill=blue, mark=+, only marks, scatter]
table{%
x  y
1.49010747290778 1.25293176758973
1.35421920899306 1.22895595818063
1.13442803742674 1.14208126745016
0.895799764542702 0.984816795061364
0.57913275076115 0.696442238303692
0.0756808096623169 0.104529779338614
};
\addplot [semithick, green!50.1960784313725!black, dotted]
table {%
0 0
1.49010747290778 2.10733019757969
};
\draw (axis cs:1.42010747290778,1.28293176758973) node[
  scale=0.5,
  anchor=base west,
  text=black,
  rotate=0.0
]{$c=0.1$};
\draw (axis cs:1.28421920899306,1.25895595818063) node[
  scale=0.5,
  anchor=base west,
  text=black,
  rotate=0.0
]{$c=0.25$};
\draw (axis cs:1.06442803742674,1.17208126745016) node[
  scale=0.5,
  anchor=base west,
  text=black,
  rotate=0.0
]{$c=0.5$};
\draw (axis cs:0.825799764542702,1.01481679506136) node[
  scale=0.5,
  anchor=base west,
  text=black,
  rotate=0.0
]{$c=0.75$};
\draw (axis cs:0.50913275076115,0.726442238303692) node[
  scale=0.5,
  anchor=base west,
  text=black,
  rotate=0.0
]{$c=1.0$};
\draw (axis cs:0.00568080966231685,0.134529779338614) node[
  scale=0.5,
  anchor=base west,
  text=black,
  rotate=0.0
]{$c=1.25$};
\end{axis}

\end{tikzpicture}}
\end{tabular}
\caption{
Energy and momentum for numerically computed solitons for potential \eqref{eq:potSalvador}
with $\beta=1.0$ and $\lambda=0.4$.
The first two panels show  $P^{\delta x}$  and $E^{\delta x}$ as functions of $c$,
while the last one depicts  $E^{\delta x}$ as a function of $P^{\delta x}$.
On the last panel, the dashed line corresponds to $P^{\delta x}\mapsto P^{\delta x}\sqrt{2}$.
}
\label{fig:Edec2}
\end{figure}
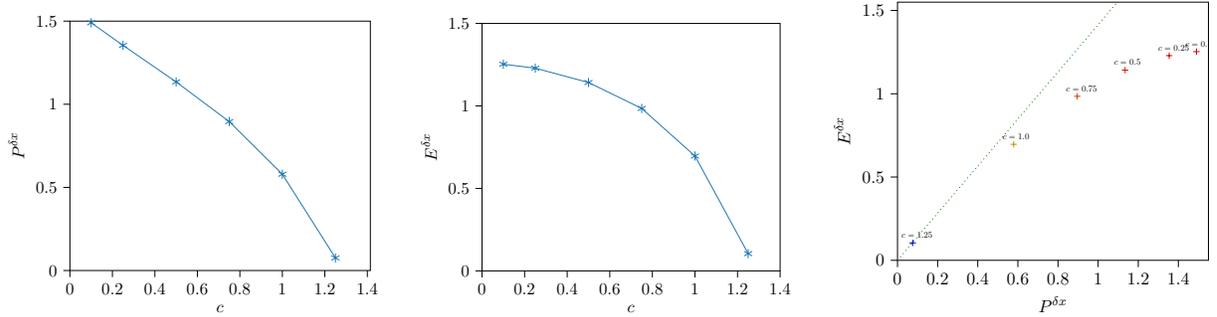

Next, we consider the same potential \eqref{eq:potSalvador} with $\beta=0.15$ and $\lambda=0.05$,
so that $\beta/(\beta-2\lambda)=3.0$ and $c_L=\sqrt2$. 
The numerical experiments are carried out with $L=400$
and $N=6399$, so that $\delta x=6.25{\rm e}-2$ and $\varepsilon^2=0.01$.
The parameter $h_{\rm min}$ is so small and the parameter $n_{\rm max}$ is so big that
the method has always converged in these numerical experiments.
The numerical results are displayed in Figure \ref{fig:PotSalvadorNew2}.
Note that, even though  $c_L=\sqrt 2$ in this case, the numerical
method converged to $0$ when starting from the soliton with speed $c=1.25$, which is {\it not}
displayed in the numerical results. For the other speeds, the numerical method converged to nonzero solitons, which are displayed in Figure \ref{fig:PotSalvadorNew2}.
The energy $P^{\delta x}$ as a function of the momentum $P^{\delta x}$ is displayed
on the left panel of Figure \ref{fig:severalEdeP}.
The dispersion curve 
is depicted in the left panel in Figure~\ref{dispersion:E1}, where there is no roton minimum and the  computed solitons are strictly decreasing on $\R^+$, as well as the curves of
momentum $P^{\delta x}$ and energy $E^{\delta x}$.
\begin{figure}[!ht]
	\centering
	\begin{tabular}{cc}
		\resizebox{0.5\textwidth}{!}{
			\includegraphics{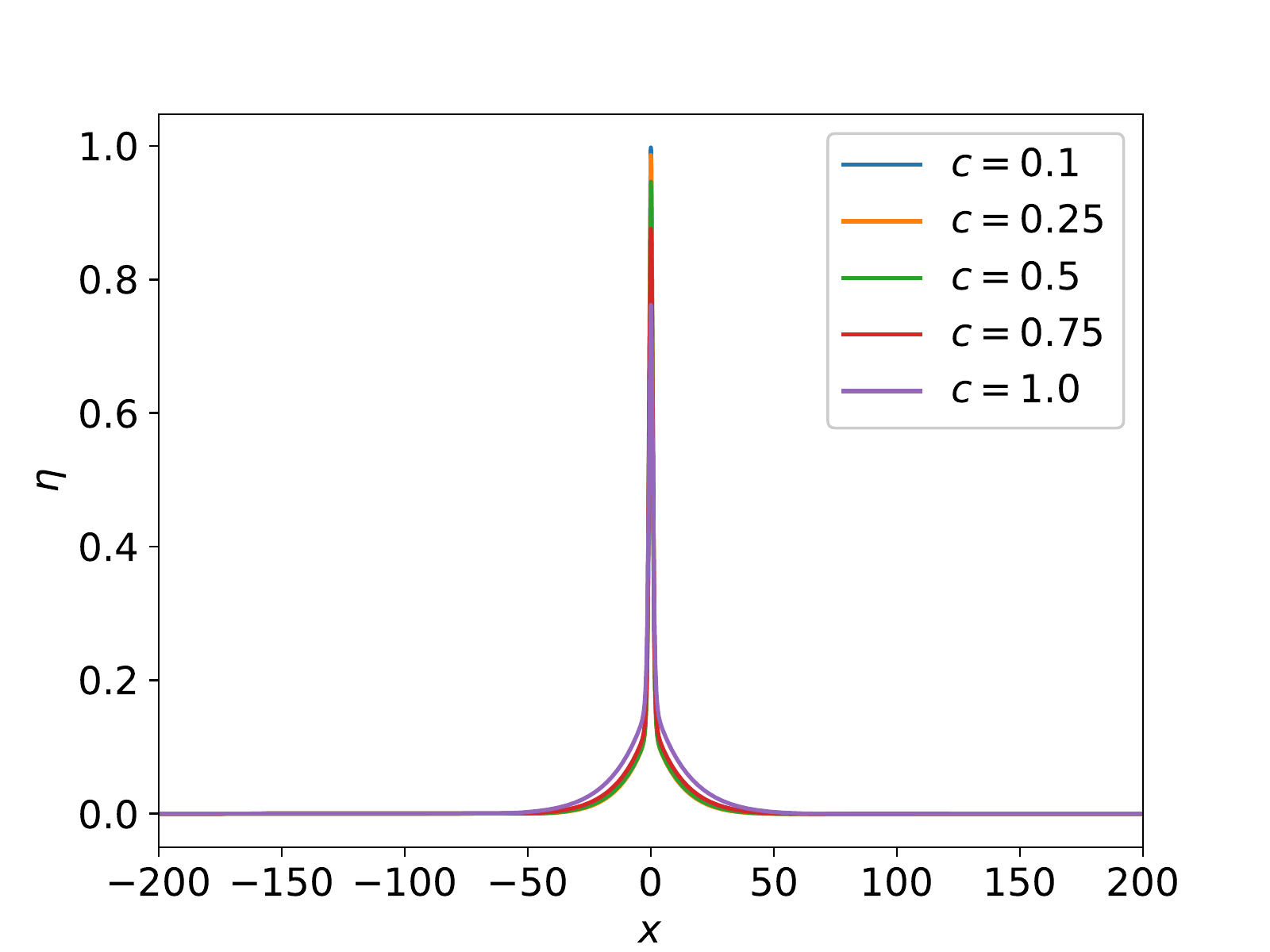}}
	&
   \resizebox{0.5\textwidth}{!}{\includegraphics{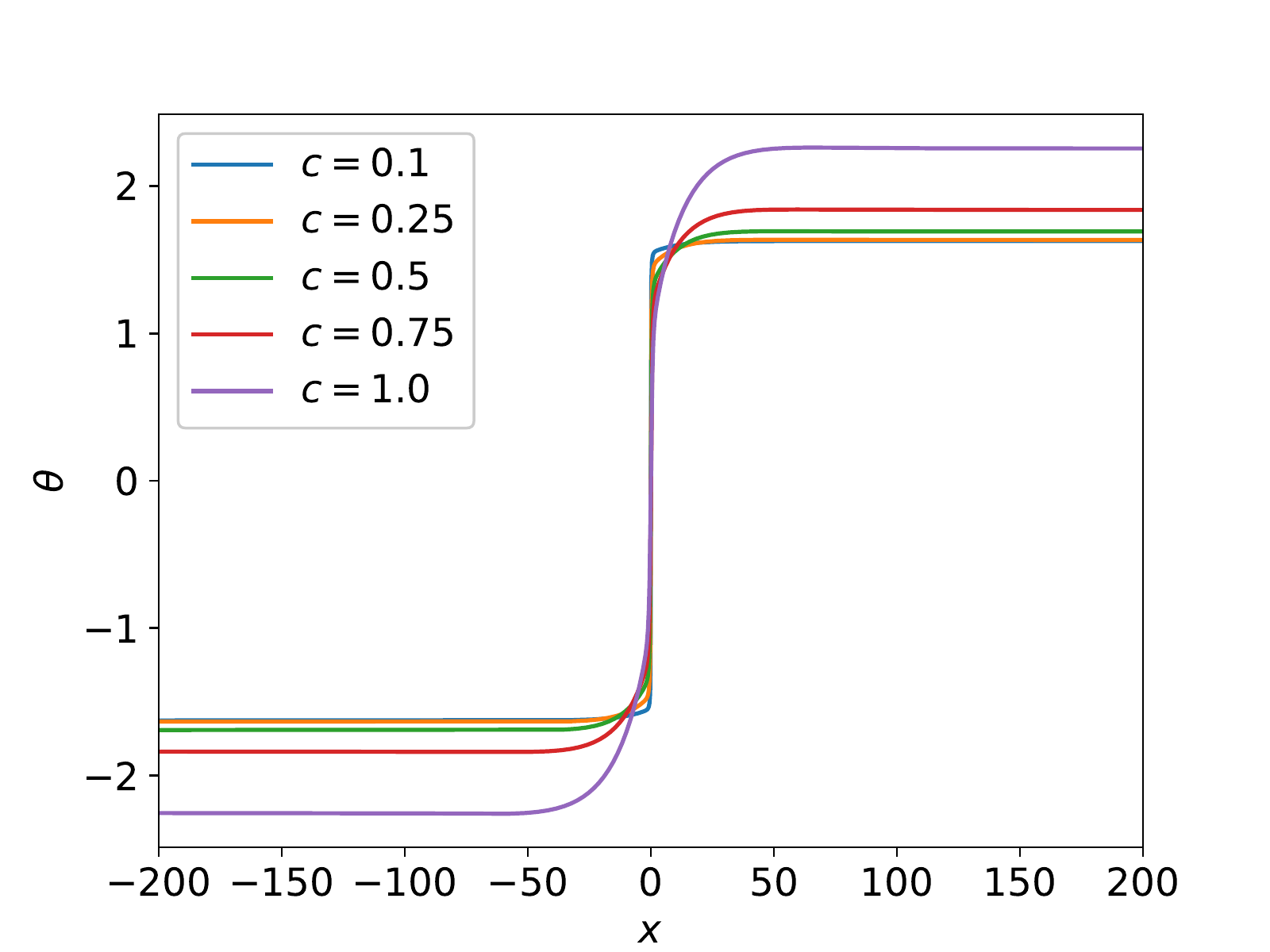}}\\
          \resizebox{0.5\textwidth}{!}{
			\includegraphics{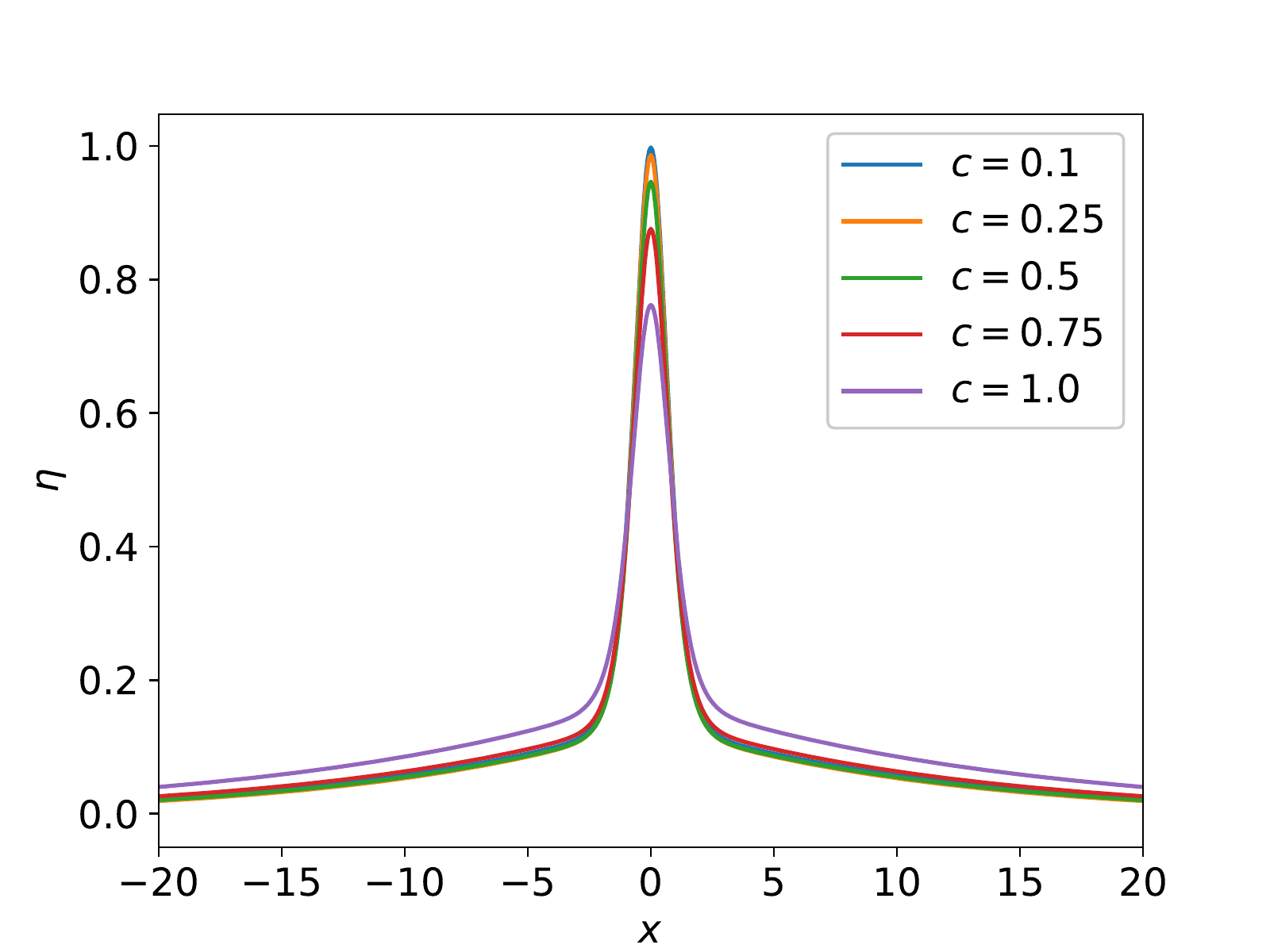}}
	&
   \resizebox{0.5\textwidth}{!}{\includegraphics{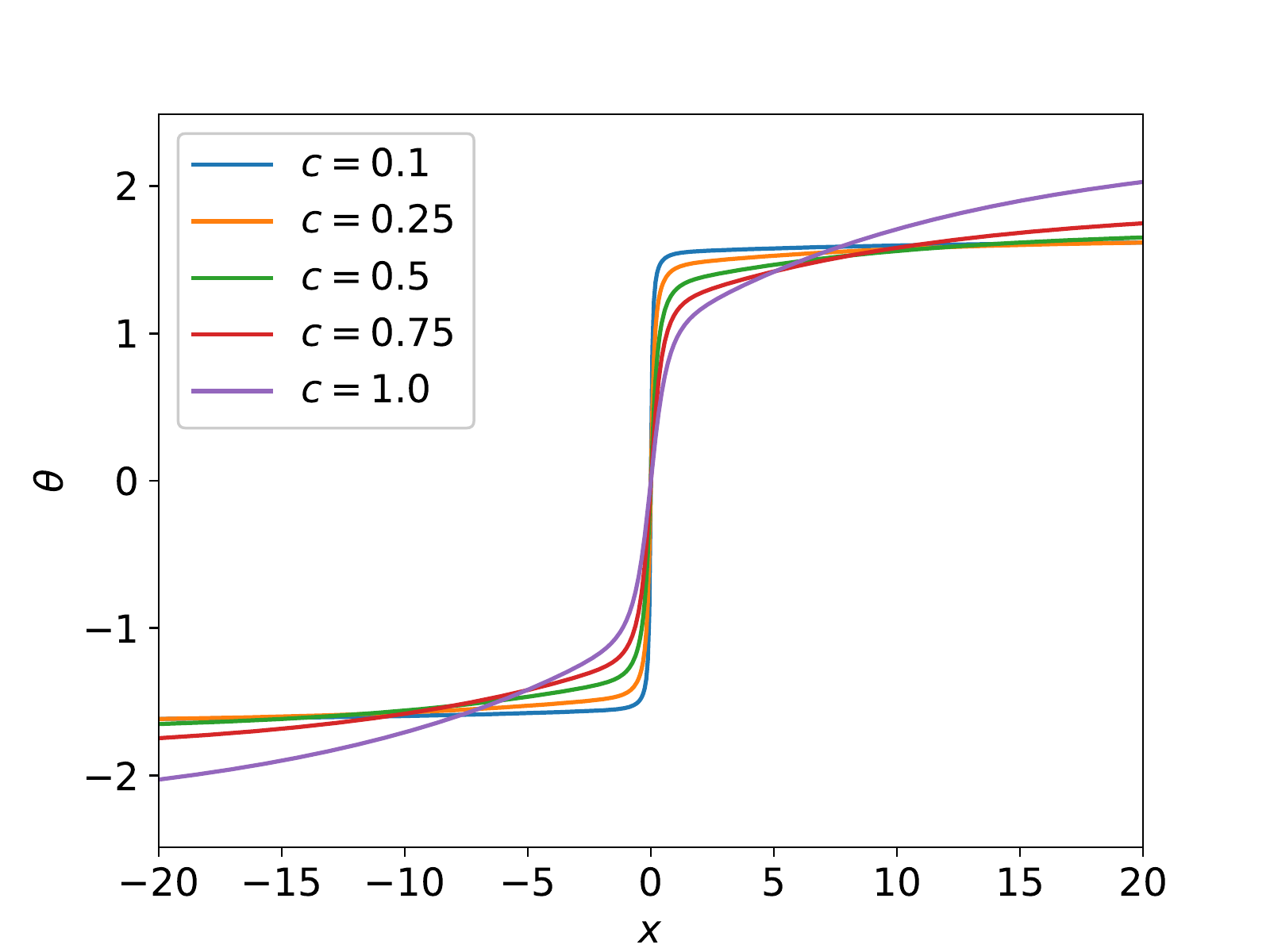}}
\end{tabular}
\caption{Numerically computed solitons for potential \eqref{eq:potSalvador}
with $\beta=0.15$ and $\lambda=0.05$, 
showing  $\bm\eta_k$ (left panel) and ${\bm\theta}_k$ (right panel) as function of $x_k$.
First line: from $-L/2$ to $L/2$. Second line: from $-20$ to $+20$.}
\label{fig:PotSalvadorNew2}
\end{figure}

\begin{figure}[!ht]
	\centering
	\begin{tabular}{cc}
		\resizebox{0.5\textwidth}{!}{
				\includegraphics{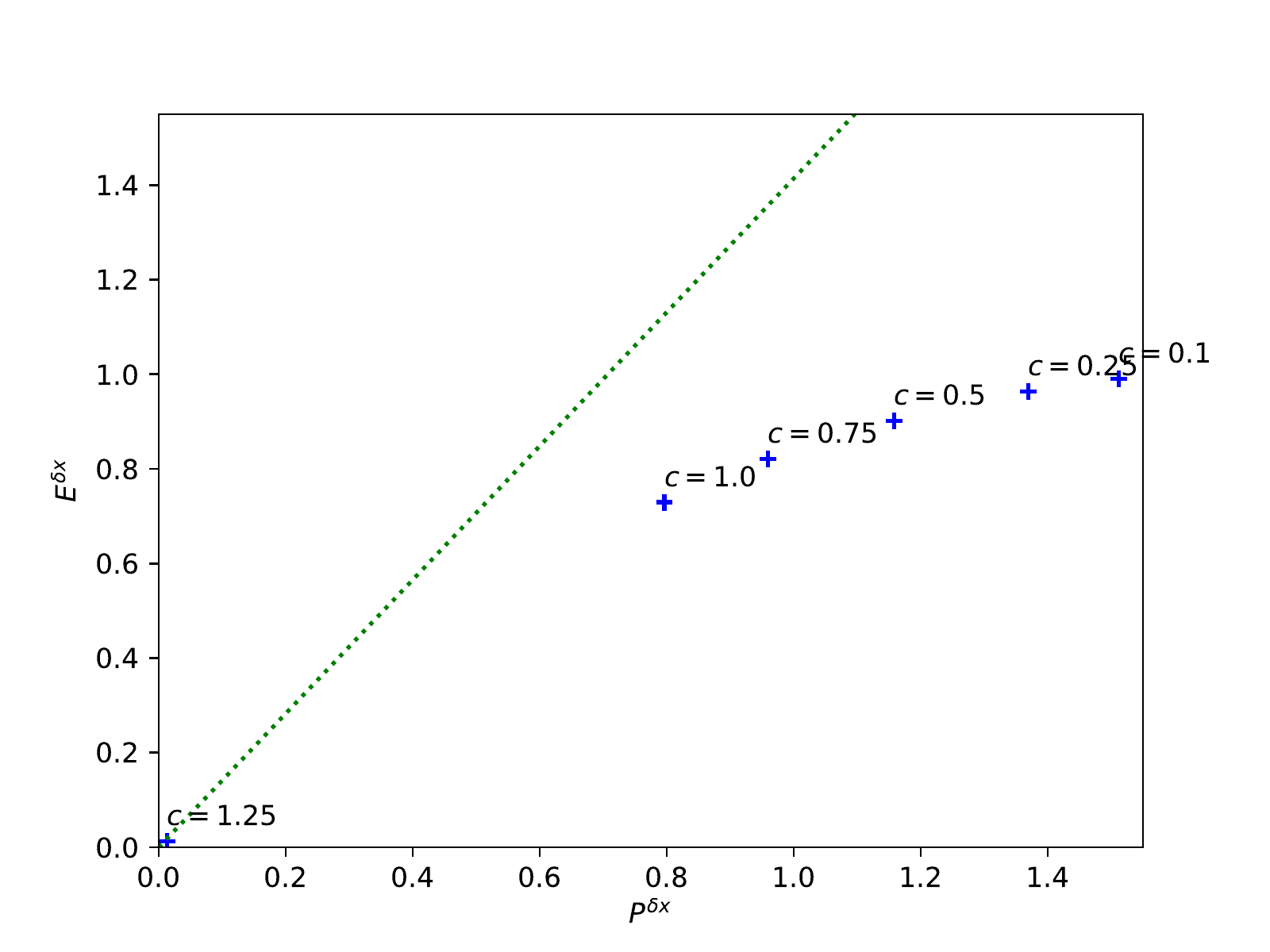}}
	&
   \resizebox{0.5\textwidth}{!}{\includegraphics{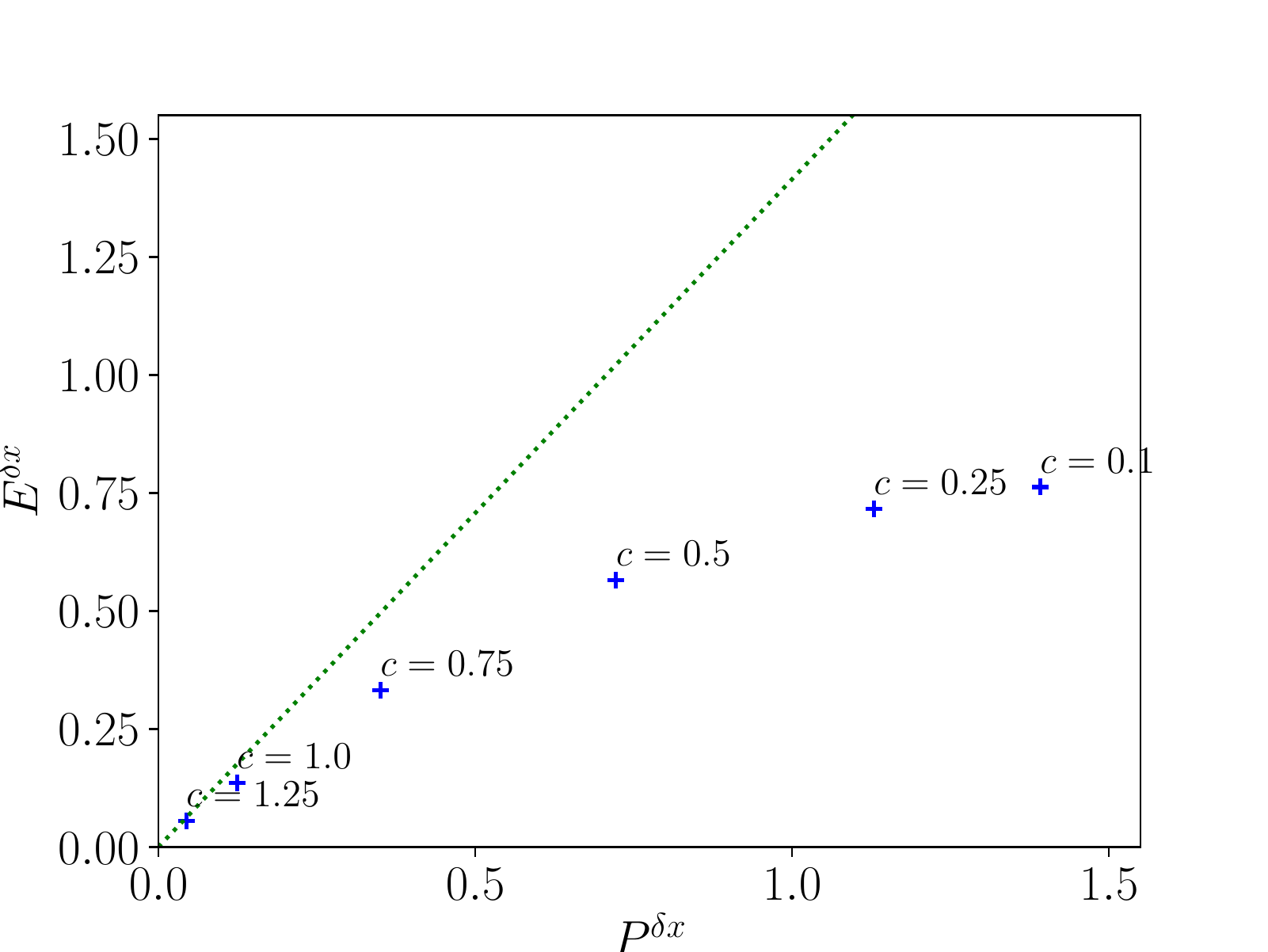}}
 %
\end{tabular}
\caption{$E^{\delta x}$ as a function of $P^{\delta x}$ for several speeds.
  Left:  $\beta=0.15$ and $\lambda=0.05$.
  Right: $\beta=0.5$ and $\lambda=-1.0$.
  On both panels, the dashed line corresponds to $P^{\delta x}\mapsto P^{\delta x}\sqrt{2}$.}
\label{fig:severalEdeP}
\end{figure}

We   examine the potential \eqref{eq:potSalvador} with $\beta=0.5$ and $\lambda=-1.0$, so that $\beta/(\beta-2\lambda)=0.2$, whose  dispersion curve is depicted in the right panel in Figure~\ref{dispersion:E1}. Notice that $\omega_\lambda$ has a roton minimum and that  $c_L(\lambda) \sim 1.19$.
The numerical experiments are carried out with $L=60$
and $N=2399$, so that $\delta x=2.5{\rm e}-2$ and $\varepsilon^2=6.25{\rm e}-3$.
The parameter $h_{\rm min}$ is so small and the parameter $n_{\rm max}$ is so big that
the method has always converged in these numerical experiments.
The numerical results are displayed in Figure \ref{fig:PotSalvadorNew3}.
Note that, despite the fact that the Landau speed is $c_L \# 1.19$, the numerical method converged to a nonzero soliton for $c=1.25$, which is displayed in the numerical results.
Observe that, in contrast to the two previous experiments for potential \eqref{eq:potSalvador},
the minimizers ${\bm \eta}$  are {\em not} strictly decreasing on $\R^+$except for $c=1.25$, and take  negative values, which means that $\abs{\bm u} =\sqrt{1-{\bm \eta}}$ has values above $1.0$.
This leads thinking that the existence of a roton minimum of the dispersion curve is related to 
the   oscillation of solitons. 
Remarkably, the only minimizer without oscillations corresponds to $c=1.25$, which is above the Landau speed.

The energy $E^{\delta x}$ as a function of the momentum $P^{\delta x}$ is displayed in the right panel of Figure \ref{fig:severalEdeP} 
for $(\beta,\lambda)=(0.15,0.05)$ and $(\beta,\lambda)=(0.5,-1.0)$. We see that the existence of a roton minimum does not change the aspect of the curve $(E^{\delta x},P^{\delta x})$, that seems to be strictly concave, so that the computed solitons should be stable.

\begin{figure}[!ht]
	\centering
	\begin{tabular}{cc}
		\resizebox{0.5\textwidth}{!}{
			\includegraphics{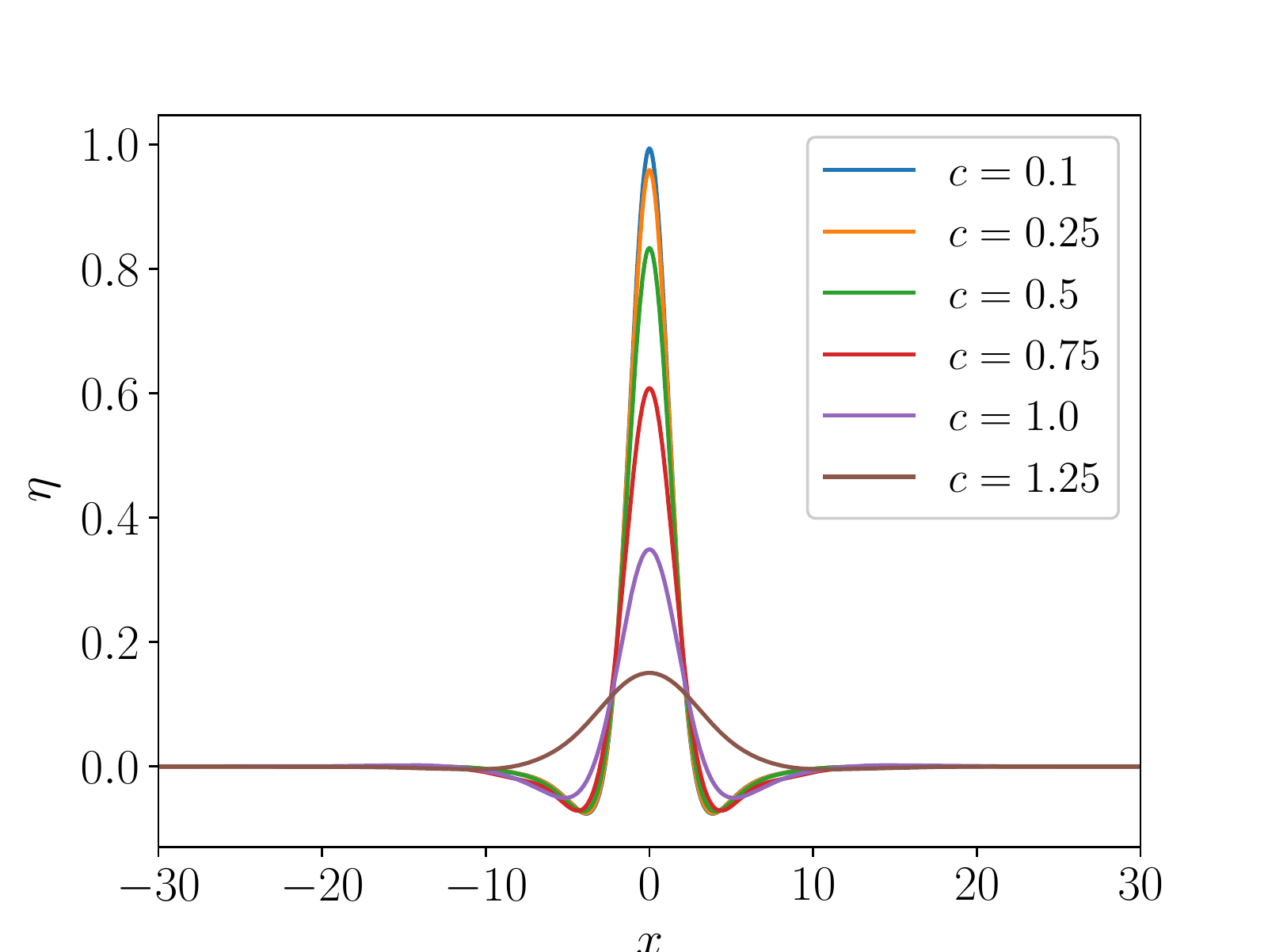}}
	&
   \resizebox{0.5\textwidth}{!}{\includegraphics{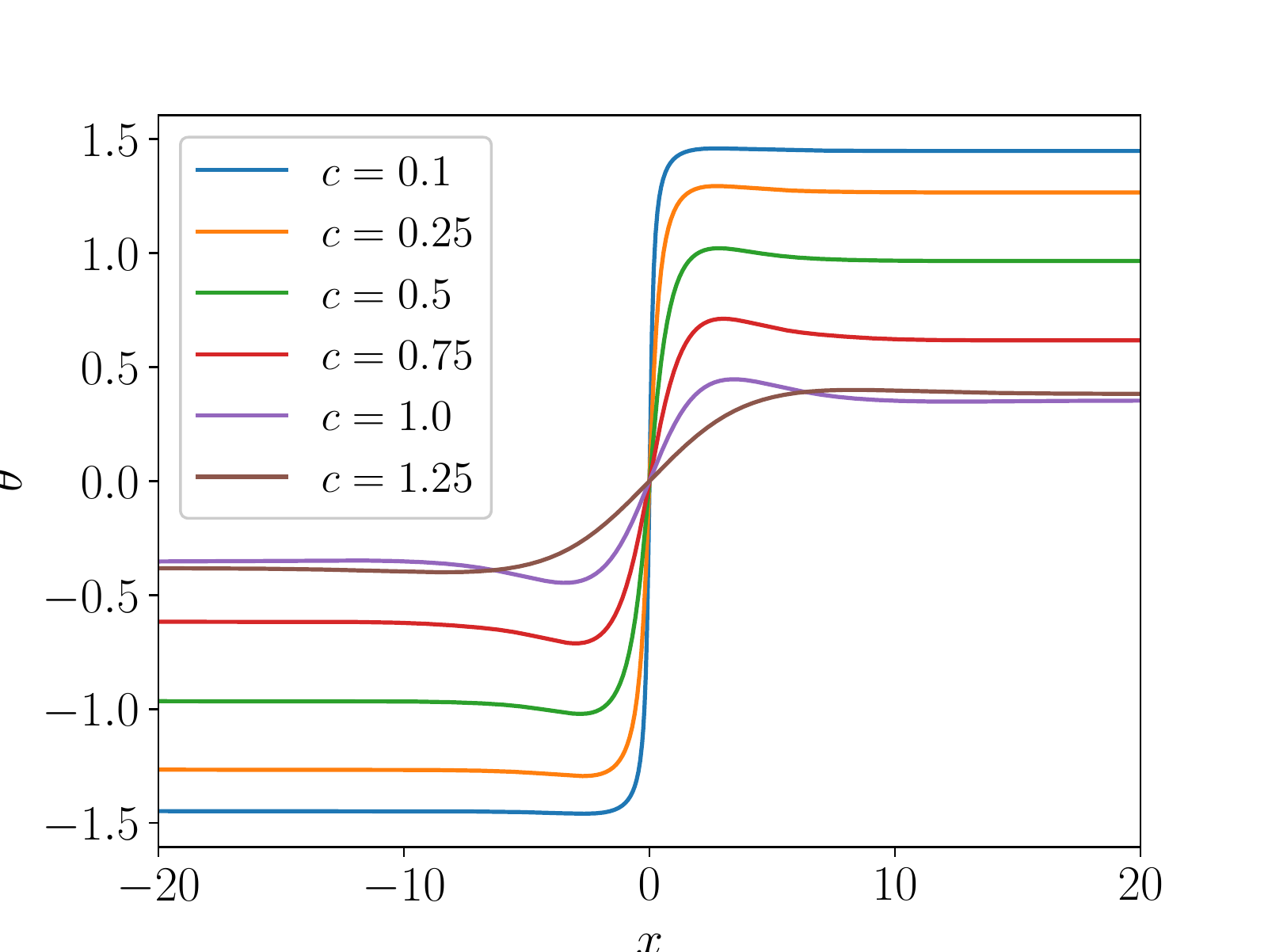}}
\end{tabular}
\caption{Numerically computed solitons for potential \eqref{eq:potSalvador}
with $\beta=0.5$ and $\lambda=-1.0$,
showing  $\bm\eta_k$ (left panel) and ${\bm\theta}_k$ (right panel) as function of $x_k$.
}
\label{fig:PotSalvadorNew3}
\end{figure}

We end this subsection by studying the effect of the parameter $\lambda$, for the same potential \eqref{eq:potSalvador}, on the minimum value of 
$\abs{\bm u}$, and the $L^2$ and $H^1$-norms of the solitons.
For this purpose, we fix $\beta=0.5$ and $c=0.1$, and we let $\lambda$ vary in the interval $[-10, 0.2]$. For the numerical computations, we set $L=60$ and $N=1201$, so that $\delta x \# 4.99{\rm e}-2$
and $\varepsilon^2 \# 1.25{\rm e}-2$.
Numerical solitons ${\bm \eta}$ and their phases ${\bm \theta}$
are displayed in Figure~\ref{fig:PotSalvadorNew-lambdabouge}.
The behavior of the minimum of ${\bm u}$
the soliton width, and the $L^2$ and $H^1$
norms of ${\bm \eta}$ as functions of $\lambda$ are displayed in Figure~\ref{fig:PotSalvadorNew-lambdabouge-delambda}.
Let us recall that the critical value for the presence of a roton minimum in $\omega$
is $\lambda^*=-{\beta^3}/({2(2-\beta^2)})$, that is $\lambda^*\#-0.0357$ when $\beta=0.5$.
In Figure~\ref{fig:PotSalvadorNew-lambdabouge}, we see that the solitons exhibit a negative bump 
for $\lambda\leq \lambda^*$, and that they are strictly  decreasing for $\lambda> \lambda^*$, 
supporting the link between a roton minimum and the oscillations.
In addition, we remark a monotonic behavior on the minimum value of 
$\abs{\bm u}$, on the width of $\bm \eta$,  and on the $L^2$ and $H^1$-norms of the solitons, 
independently of the presence of a roton minimum.

\begin{figure}[!ht]
	\centering
	\begin{tabular}{cc}
		\resizebox{0.5\textwidth}{!}{
			\includegraphics{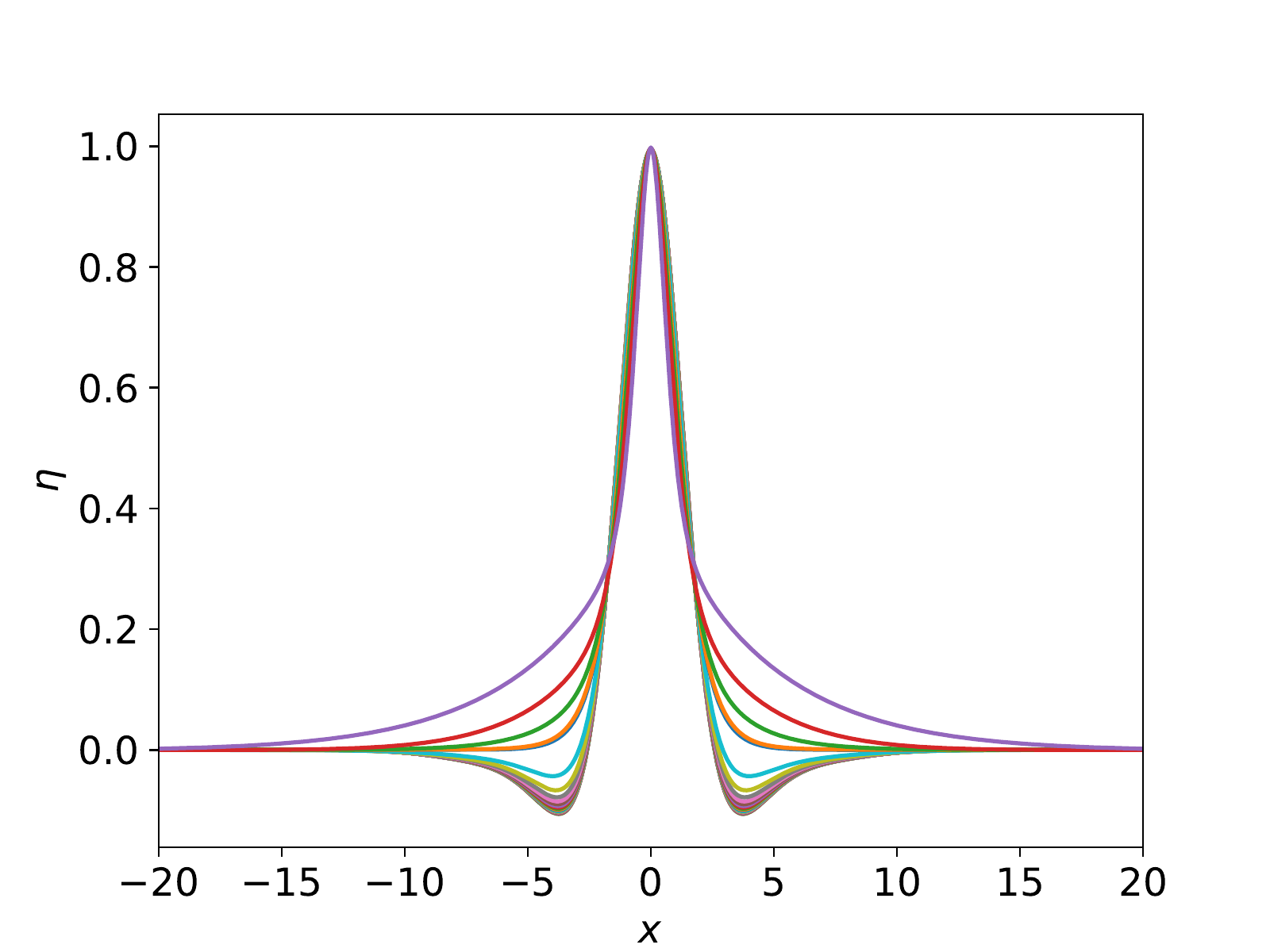}}
	&
   \resizebox{0.5\textwidth}{!}{\includegraphics{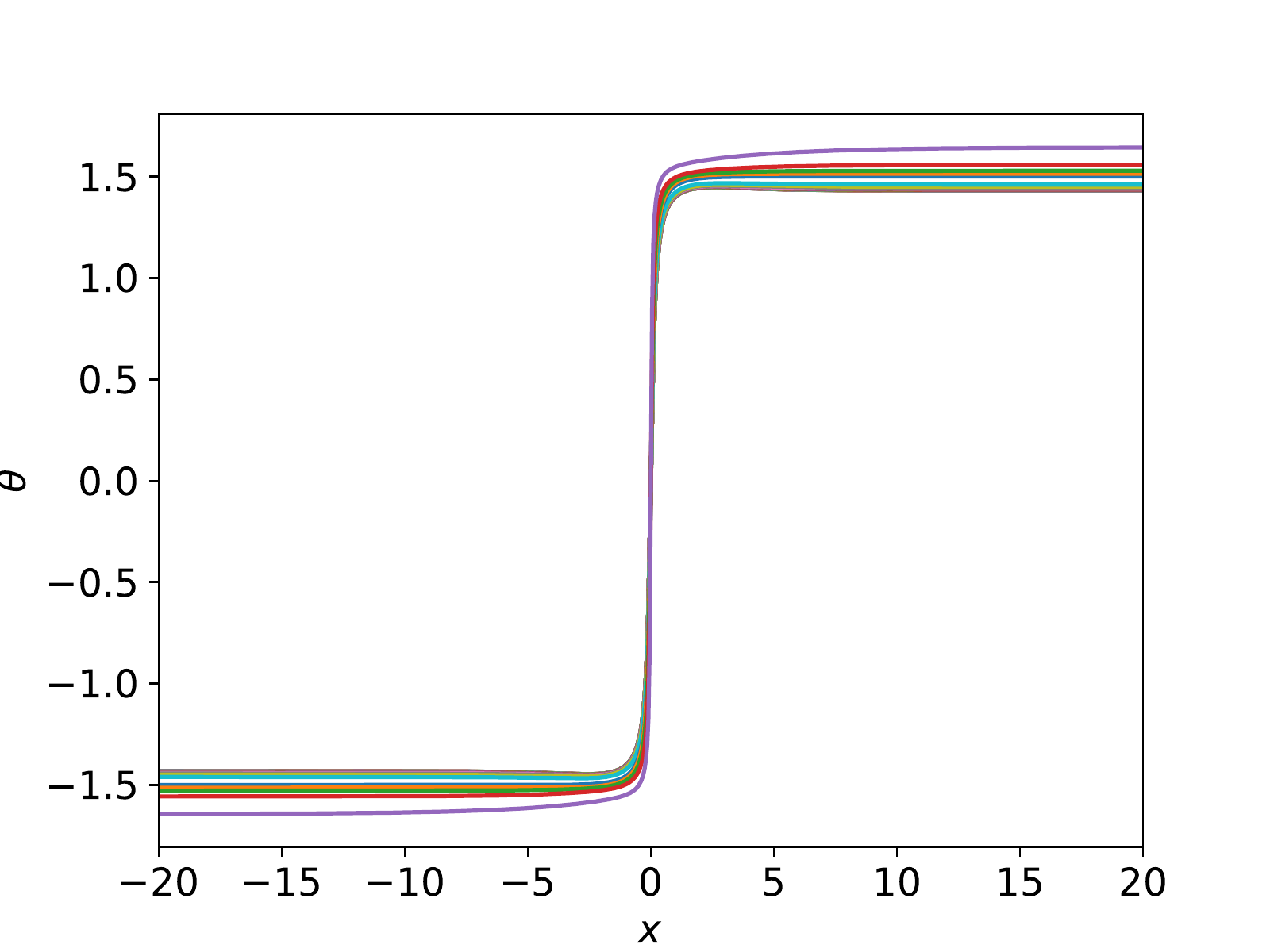}}
        \end{tabular}
        \begin{tabular}{c}
          \resizebox{0.50\textwidth}{!}{\includegraphics{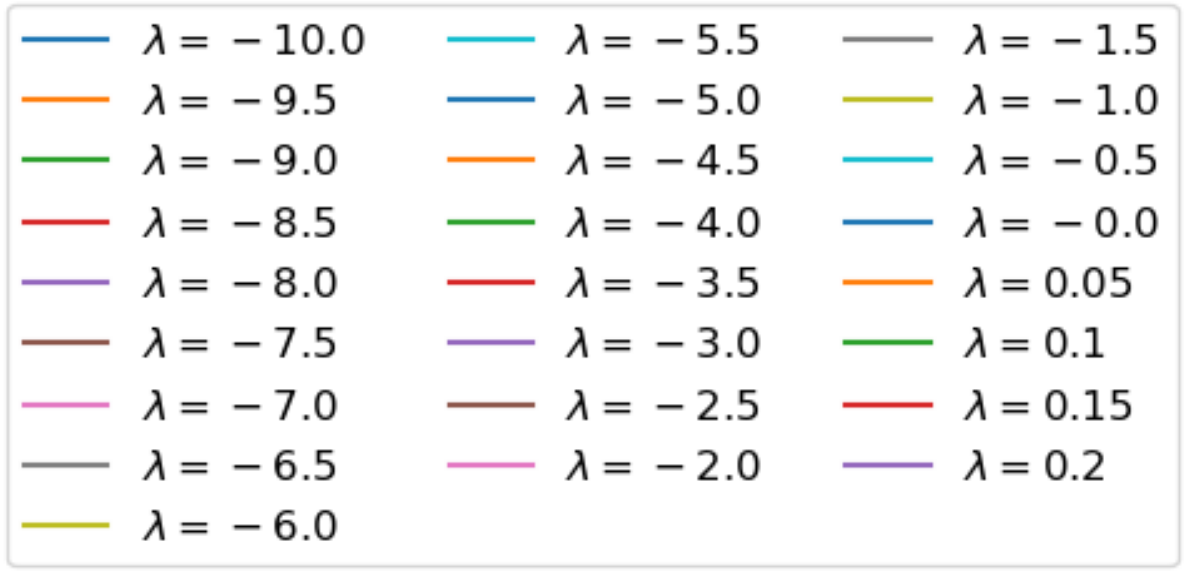}}
        \end{tabular}
\caption{Numerically computed solitons for potential \eqref{eq:potSalvador}
  with $\beta=0.5$ and $c=0.1$ for several values of $\lambda$.
  Top panel:
   $\bm\eta_k$ (left) and ${\bm\theta}_k$ (right) as function of $x_k$.
%
  Bottom panel is the legend of the top panel.}
\label{fig:PotSalvadorNew-lambdabouge}
\end{figure}

\begin{figure}[!ht]
	\centering
	\begin{tabular}{ccc}
		\resizebox{0.3\textwidth}{!}{
			\includegraphics{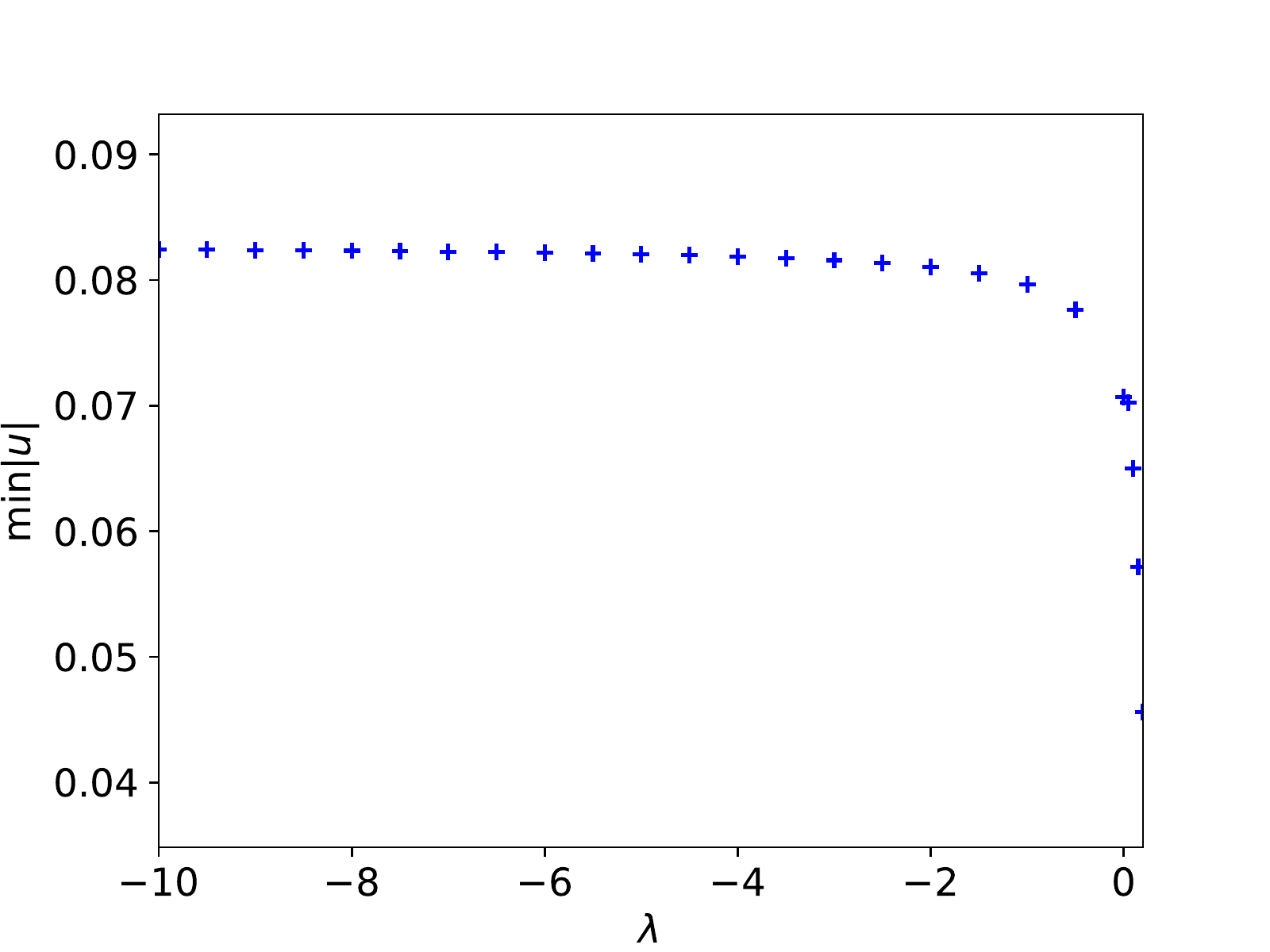}}
	&
   \resizebox{0.3\textwidth}{!}{\includegraphics{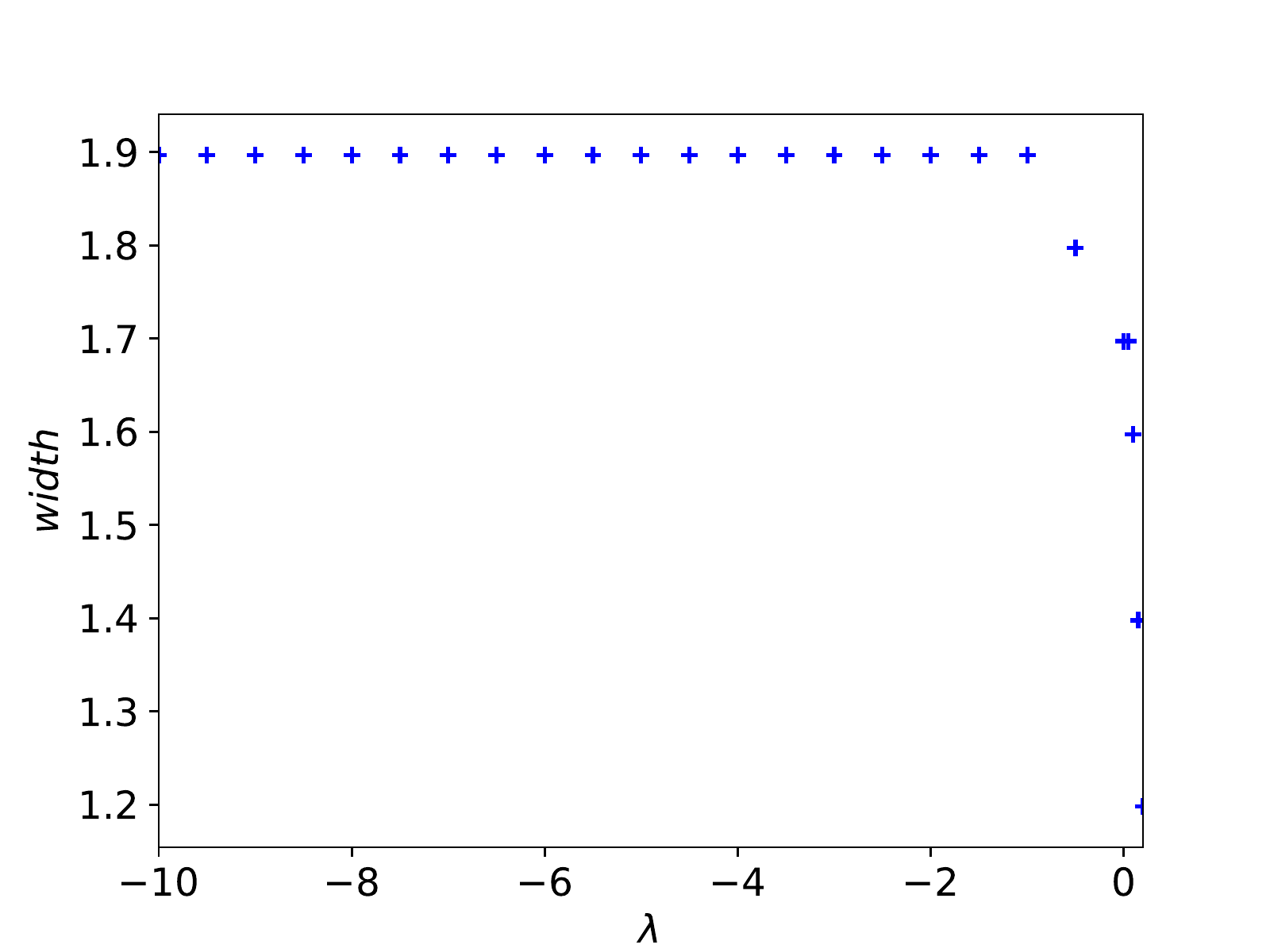}}
   &
   \resizebox{0.3\textwidth}{!}{\includegraphics{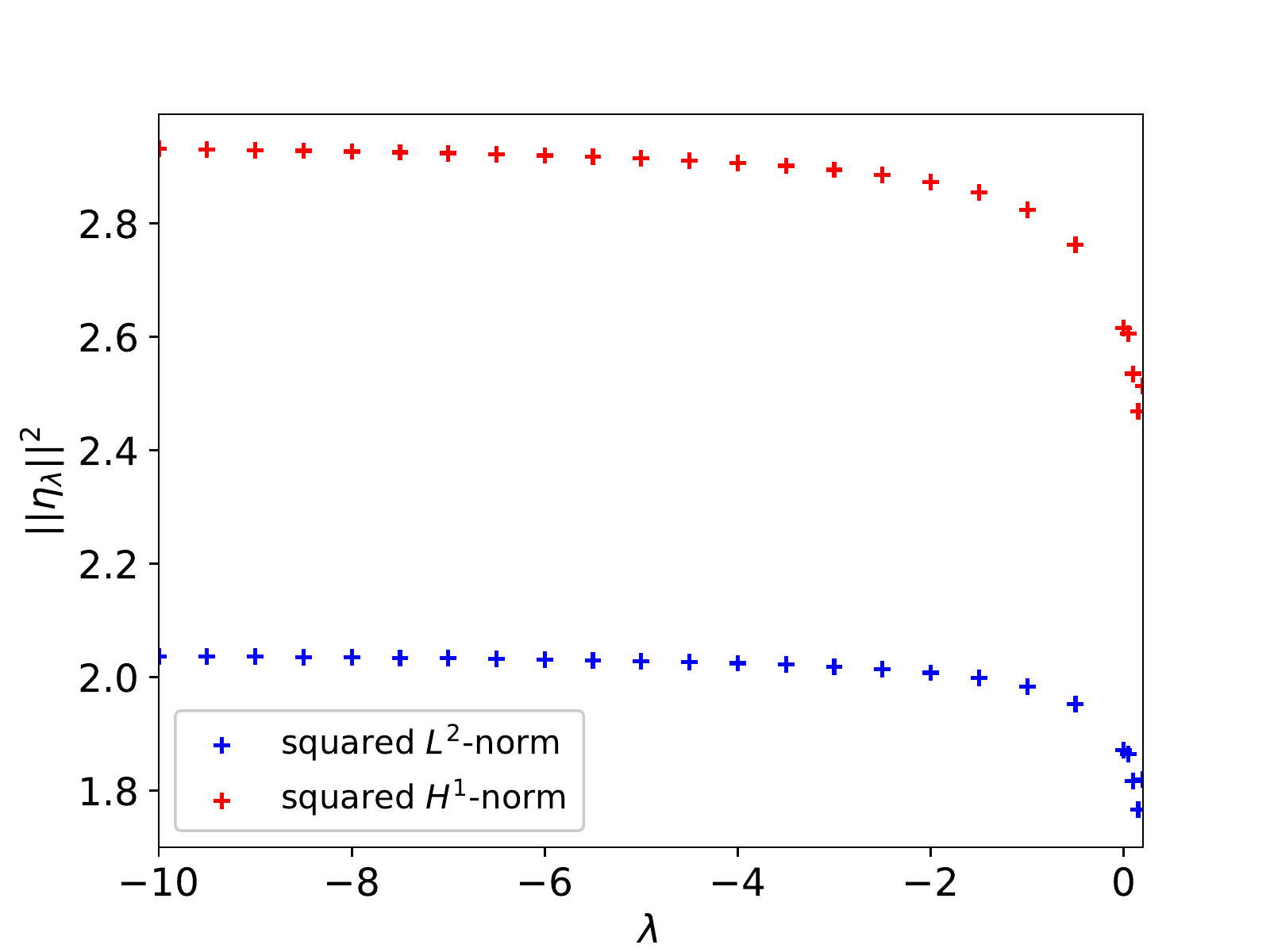}}
\end{tabular}
\caption{Minimum of ${\bm u}$ (left panel), width of the soliton (middle panel),
  and $L^2$ and $H^1$ norms (right panel) as a function of $\lambda$
  for potential \eqref{eq:potSalvador} with $\beta=0.5$ and $c=0.1$.
  Numerical parameters are indicated in the text.}
\label{fig:PotSalvadorNew-lambdabouge-delambda}
\end{figure}

\subsection{Solitons  with potential \eqref{eq:potGaussien}}
\label{sebsec:potGaussien}
We consider now the Gaussian potential $\mathcal W_\lambda$
defined in \eqref{eq:potGaussien} in Example~2.
We take first $\lambda=1.0$ and compute numerically nontrivial minimizers of the function $F_\lambda^{\delta x}$
defined in \eqref{eq:defF} $\lambda=1.0$ and several values of $c$
using the algorithm described in Subsection \ref{subsec:algo}.
The numerical experiments are carried out with $L=50$, $N=999$, so that $\delta x = 0.05$,
and $\varepsilon^2=\delta x/4$. In this case, $\delta \xi=2\pi/L\# 0.13$.
The results are displayed in Figures~\ref{fig:etadec3} and \ref{fig:Edec3}.
As explained in Example~2, we expect solutions for every  $c\in(0,c_L(1))$,
where $c_L(1)=\# 1.30$, in agreement with Figures~\ref{fig:etadec3} and \ref{fig:Edec3}.  Moreover, 
the presence of a roton minimum (see middle panel in Figure~\ref{dispersion:E2}) should imply the 
presence of oscillations on $\R^+$, which is indeed seen in these figures. Notice that for 
$c=1.25$, $\bm \eta_k$ shows no oscillation on $\R^+$, but the there is a bump in its phase $\bm 
\theta_k$, so there is an oscillation of $\bm u_k$. 
\begin{figure}[!ht]
	\centering
	\begin{tabular}{cc}
		\resizebox{0.5\textwidth}{!}{
			\includegraphics{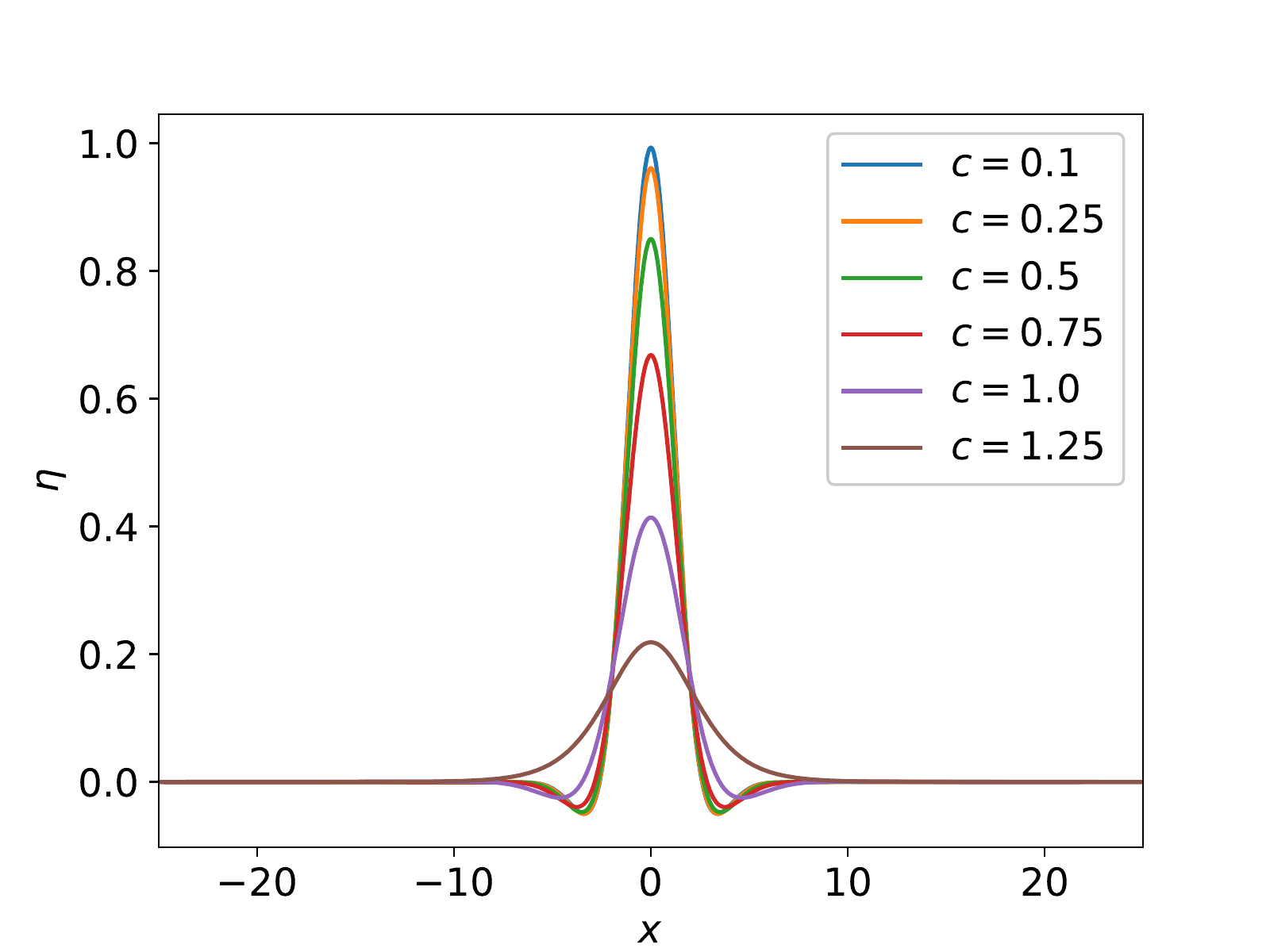}}%
	&                                                                         
	\resizebox{0.5\textwidth}{!}{\includegraphics{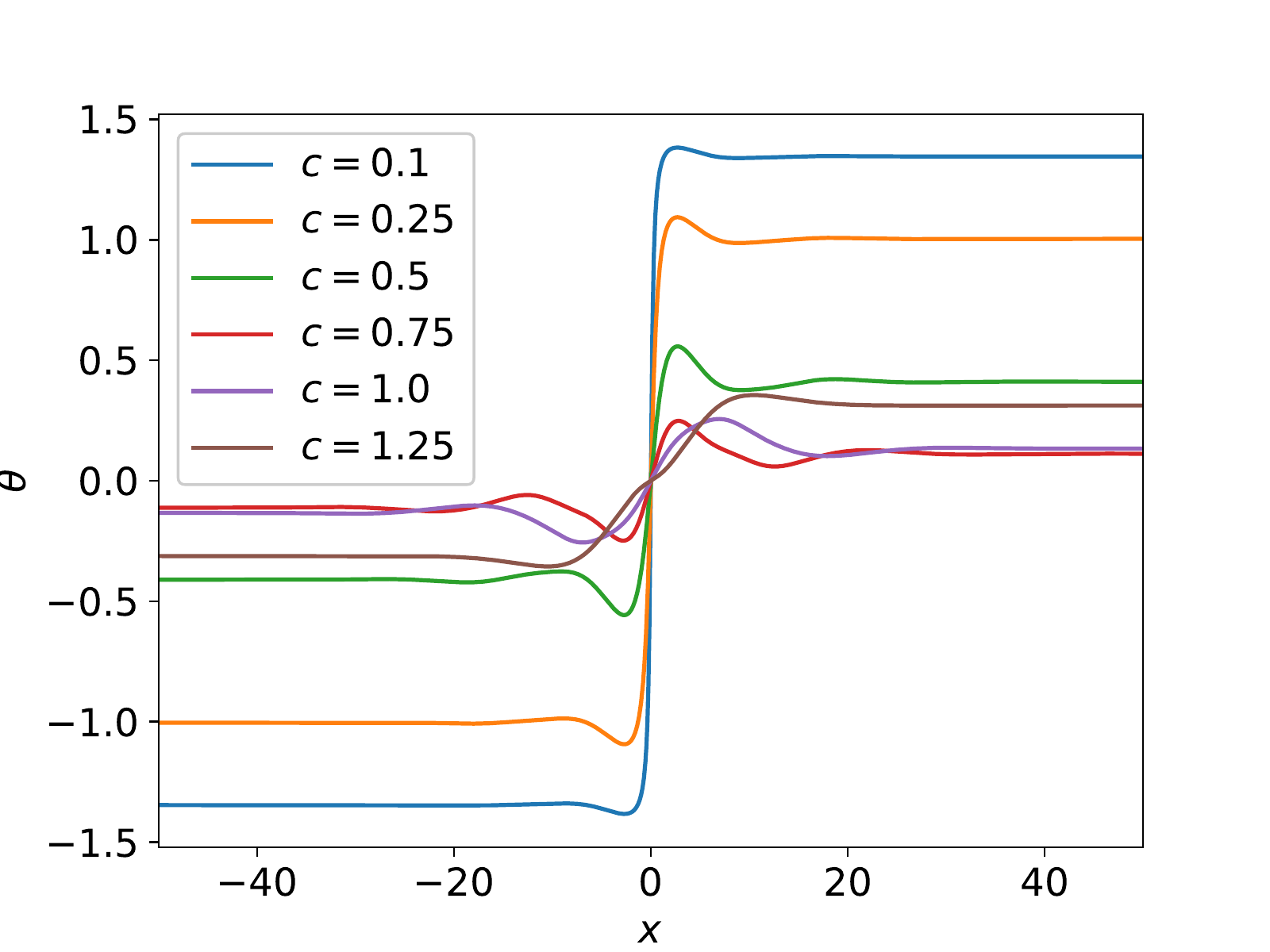}}
\end{tabular}
\caption{Numerically computed solitons for potential \eqref{eq:potGaussien}
with $\lambda=1.0$,
showing  $\bm\eta_k$ (left panel) and ${\bm\theta}_k$ (right panel) as function of $x_k$.
}
\label{fig:etadec3}
\end{figure}

\begin{figure}[!ht]
\centering
\begin{tabular}{ccc}
\resizebox{0.33\textwidth}{!}{
\begin{tikzpicture}

\definecolor{color0}{rgb}{0.12156862745098,0.466666666666667,0.705882352941177}

\begin{axis}[
tick align=outside,
tick pos=left,
x grid style={white!69.0196078431373!black},
xlabel={\(\displaystyle c\)},
xmin=0, xmax=1.4142135623731,
xtick style={color=black},
y grid style={white!69.0196078431373!black},
ylabel={\(\displaystyle P^{\delta x}\)},
ymin=0, ymax=1.5,
ytick style={color=black}
]
\addplot [semithick, color0, mark=asterisk, mark size=3, mark options={solid}]
table {%
0.1 1.40588786085359
0.25 1.1594214823102
0.5 0.769104584644954
0.75 0.429044091098425
1 0.175170133218195
1.25 0.0732963127209293
};
\end{axis}

\end{tikzpicture}} &                                                                           \resizebox{0.33\textwidth}{!}{
\begin{tikzpicture}

\definecolor{color0}{rgb}{0.12156862745098,0.466666666666667,0.705882352941177}

\begin{axis}[
tick align=outside,
tick pos=left,
x grid style={white!69.0196078431373!black},
xlabel={\(\displaystyle c\)},
xmin=0, xmax=1.4,
xtick style={color=black},
y grid style={white!69.0196078431373!black},
ylabel={\(\displaystyle E^{\delta x}\)},
ymin=0, ymax=1.5,
ytick style={color=black}
]
\addplot [semithick, color0, mark=asterisk, mark size=3, mark options={solid}]
table {%
0.1 0.805454409068738
0.25 0.76344942676453
0.5 0.620828259218095
0.75 0.413213838720355
1 0.197944501960267
1.25 0.0927897576154562
};
\end{axis}

\end{tikzpicture}} & \resizebox{0.33\textwidth}{!}{
\begin{tikzpicture}

\begin{axis}[
tick align=outside,
tick pos=left,
x grid style={white!69.0196078431373!black},
xlabel={\(\displaystyle P^{\delta x}\)},
xmin=0, xmax=1.55,
xtick style={color=black},
y grid style={white!69.0196078431373!black},
ylabel={\(\displaystyle E^{\delta x}\)},
ymin=0, ymax=1.55,
ytick style={color=black}
]
\addplot [draw=blue, fill=blue, mark=+, only marks, scatter]
table{%
x  y
1.40588786085359 0.805454409068738
1.1594214823102 0.76344942676453
0.769104584644954 0.620828259218095
0.429044091098425 0.413213838720355
0.175170133218195 0.197944501960267
0.0732963127209293 0.0927897576154562
};
\addplot [semithick, green!50.1960784313725!black, dotted]
table {%
0 0
1.40588786085359 1.98822567999485
};
\draw (axis cs:1.33588786085359,0.835454409068738) node[
  scale=0.5,
  anchor=base west,
  text=black,
  rotate=0.0
]{$c=0.1$};
\draw (axis cs:1.0894214823102,0.79344942676453) node[
  scale=0.5,
  anchor=base west,
  text=black,
  rotate=0.0
]{$c=0.25$};
\draw (axis cs:0.699104584644954,0.650828259218095) node[
  scale=0.5,
  anchor=base west,
  text=black,
  rotate=0.0
]{$c=0.5$};
\draw (axis cs:0.359044091098425,0.443213838720355) node[
  scale=0.5,
  anchor=base west,
  text=black,
  rotate=0.0
]{$c=0.75$};
\draw (axis cs:0.105170133218195,0.227944501960267) node[
  scale=0.5,
  anchor=base west,
  text=black,
  rotate=0.0
]{$c=1.0$};
\draw (axis cs:0.00329631272092928,0.122789757615456) node[
  scale=0.5,
  anchor=base west,
  text=black,
  rotate=0.0
]{$c=1.25$};
\end{axis}

\end{tikzpicture}}
\end{tabular}
\caption{Numerically computed solitons for potential \eqref{eq:potGaussien}
with $\lambda=1.0$.
On the left: $P^{\delta x}$ as a function of $c$. In the middle: $E^{\delta x}$ as a function
of $c$. On the right: $E^{\delta x}$ as a function of $P^{\delta x}$.
}
\label{fig:Edec3}
\end{figure}
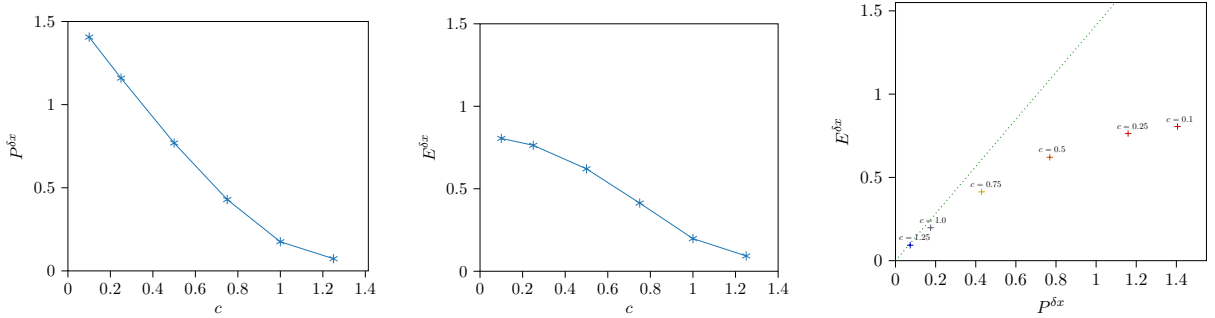

Next, we run another experiment with a larger $\lambda$ to make the potential more long range. We chose $\lambda=3$, and run another series of computation of numerical minimizers for several values of $c$ as before.
As seen in Example 2, there is a roton minimum and the Landau speed is 
$c_L(3)\# 0.66$.
To take into account the larger supports of the minimizers, we take $L=100$ and $N=1999$ so that
$\delta x=0.05$ and $\varepsilon^2=1.25{\rm e}-2$.
The results are displayed in Figures~\ref{fig:PotGaussienNew2-etaettheta}
and \ref{fig:PotGaussienNew2-energies}. We have thus obtained soliton for speeds above the Landau 
speed, which seem to be stable, due to the aspect of the  curves in 
Figure~\ref{fig:PotGaussienNew2-energies}.
This is an unexpected result, that shows that the physical conjecture of the role of the Landau 
speed is not valid in this case. However, the fact that there is bigger  gap between $c_L$  and 
$\sqrt2 $ (than for $\lambda=1$), could be an indication
of the more  oscillating behavior of the solitons in Figure~\ref{fig:PotGaussienNew2-etaettheta}.


\begin{figure}[!ht]
	\centering
	\begin{tabular}{cc}
		\resizebox{0.5\textwidth}{!}{
			\includegraphics{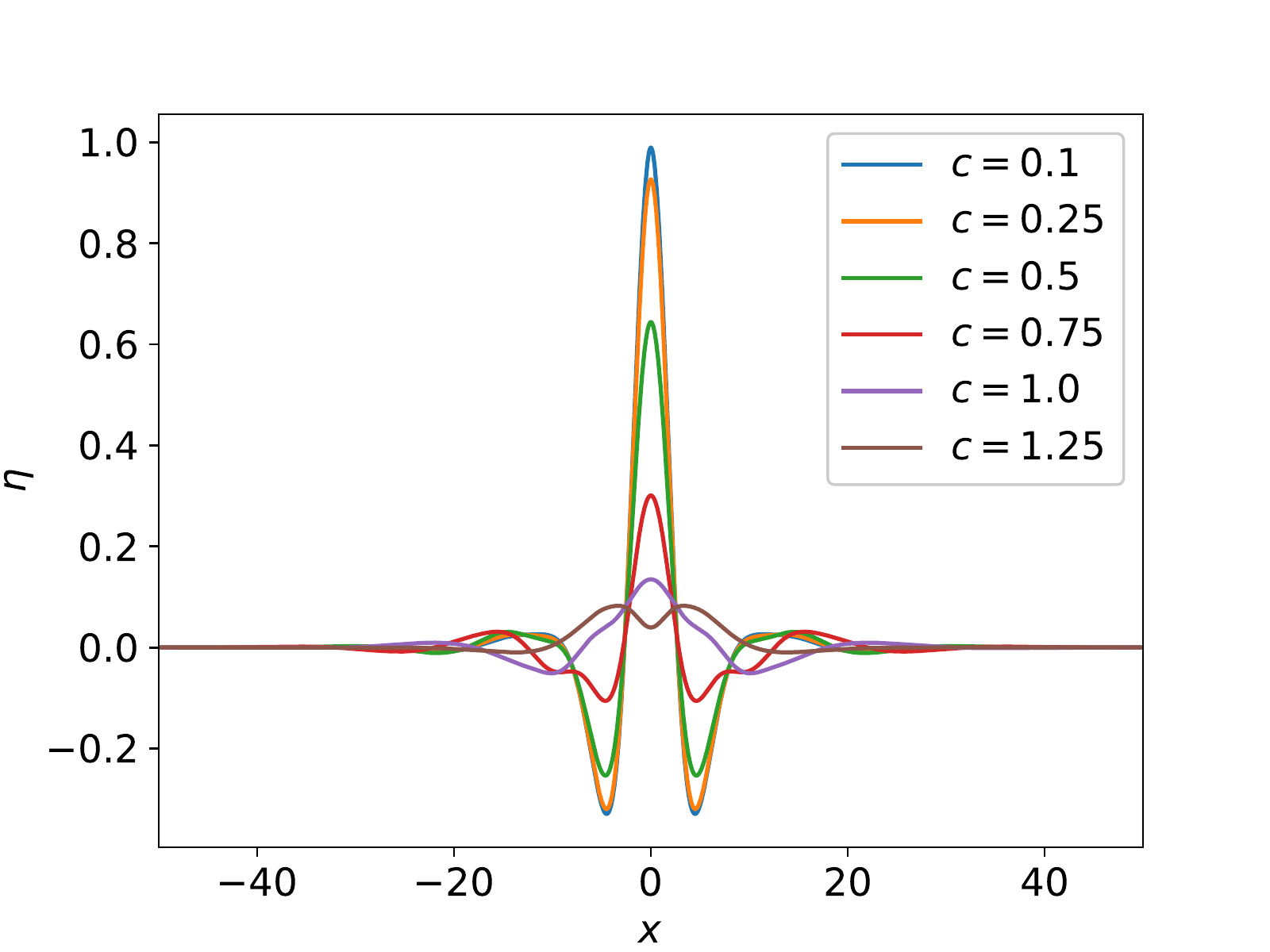}}%
	&                                                                            
	\resizebox{0.5\textwidth}{!}{\includegraphics{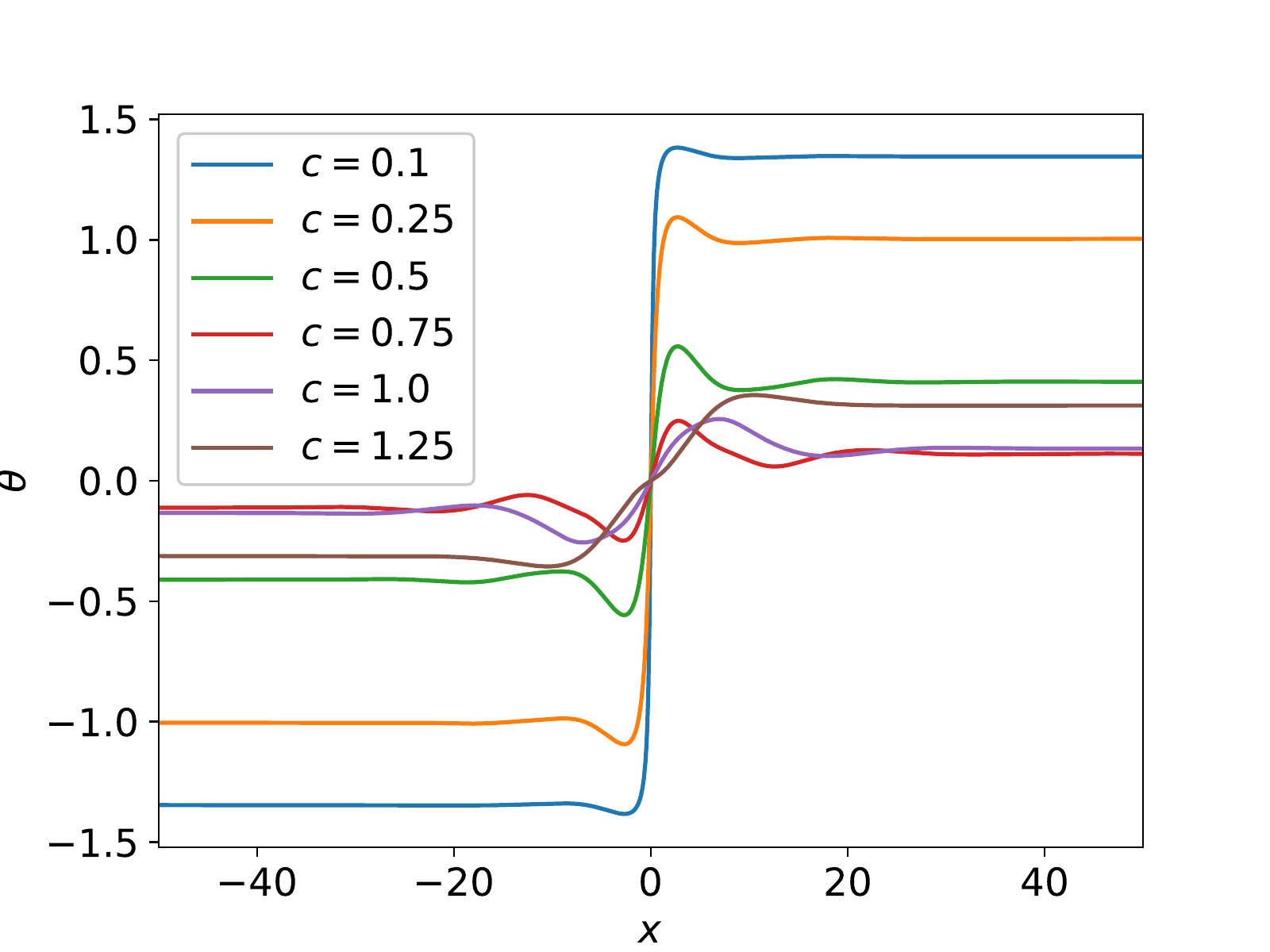}}
\end{tabular}
\caption{Numerically computed solitons for potential \eqref{eq:potGaussien}
with $\lambda=3.0$, showing  $\bm\eta_k$ (left panel) and ${\bm\theta}_k$ (right panel) as function of $x_k$.
}
\label{fig:PotGaussienNew2-etaettheta}
\end{figure}

\begin{figure}[!ht]
\centering
\begin{tabular}{ccc}
\resizebox{0.33\textwidth}{!}{
\begin{tikzpicture}

\definecolor{color0}{rgb}{0.12156862745098,0.466666666666667,0.705882352941177}

\begin{axis}[
tick align=outside,
tick pos=left,
x grid style={white!69.0196078431373!black},
xlabel={\(\displaystyle c\)},
xmin=0, xmax=1.4142135623731,
xtick style={color=black},
y grid style={white!69.0196078431373!black},
ylabel={\(\displaystyle P^{\delta x}\)},
ymin=0, ymax=1.7,
ytick style={color=black}
]
\addplot [semithick, color0, mark=asterisk, mark size=3, mark options={solid}]
table {%
0.1 1.31457263054845
0.25 0.933801378363381
0.5 0.331075869709029
0.75 0.0737322008989271
1.25 0.0232522787582929
};
\end{axis}

\end{tikzpicture}} &                                                                           \resizebox{0.33\textwidth}{!}{
\begin{tikzpicture}

\definecolor{color0}{rgb}{0.12156862745098,0.466666666666667,0.705882352941177}

\begin{axis}[
tick align=outside,
tick pos=left,
x grid style={white!69.0196078431373!black},
xlabel={\(\displaystyle c\)},
xmin=0, xmax=1.4,
xtick style={color=black},
y grid style={white!69.0196078431373!black},
ylabel={\(\displaystyle E^{\delta x}\)},
ymin=0, ymax=1.5,
ytick style={color=black}
]
\addplot [semithick, color0, mark=asterisk, mark size=3, mark options={solid}]
table {%
0.1 0.516781039134106
0.25 0.447554402852548
0.5 0.210426258739309
0.75 0.0563032508788538
1.25 0.0286067202682061
};
\end{axis}

\end{tikzpicture}} & \resizebox{0.33\textwidth}{!}{
\begin{tikzpicture}

\begin{axis}[
tick align=outside,
tick pos=left,
x grid style={white!69.0196078431373!black},
xlabel={\(\displaystyle P^{\delta x}\)},
xmin=0, xmax=1.55,
xtick style={color=black},
y grid style={white!69.0196078431373!black},
ylabel={\(\displaystyle E^{\delta x}\)},
ymin=0, ymax=1.55,
ytick style={color=black}
]
\addplot [draw=blue, fill=blue, mark=+, only marks, scatter]
table{%
x  y
1.31457263054845 0.516781039134106
0.933801378363381 0.447554402852548
0.331075869709029 0.210426258739309
0.0737322008989271 0.0563032508788538
0.0232522787582929 0.0286067202682061
};
\addplot [semithick, green!50.1960784313725!black, dotted]
table {%
0 0
1.31457263054845 1.85908644284609
};
\draw (axis cs:1.31457263054845,0.546781039134106) node[
  scale=0.5,
  anchor=base west,
  text=black,
  rotate=0.0
]{$c=0.1$};
\draw (axis cs:0.933801378363381,0.477554402852548) node[
  scale=0.5,
  anchor=base west,
  text=black,
  rotate=0.0
]{$c=0.25$};
\draw (axis cs:0.331075869709029,0.240426258739309) node[
  scale=0.5,
  anchor=base west,
  text=black,
  rotate=0.0
]{$c=0.5$};
\draw (axis cs:0.0737322008989271,0.0863032508788538) node[
  scale=0.5,
  anchor=base west,
  text=black,
  rotate=0.0
]{$c=0.75$};
\draw (axis cs:0.0232522787582929,0.0586067202682061) node[
  scale=0.5,
  anchor=base west,
  text=black,
  rotate=0.0
]{$c=1.25$};
\end{axis}

\end{tikzpicture}}
\end{tabular}
\caption{Numerically computed solitons for potential \eqref{eq:potGaussien}
with $\lambda=3.0$, showing
 $P^{\delta x}$ (left) and  $E^{\delta x}$ (center) as a function
of $c$. On the right, $E^{\delta x}$ as a function of $P^{\delta x}$.
}
\label{fig:PotGaussienNew2-energies}
\end{figure}
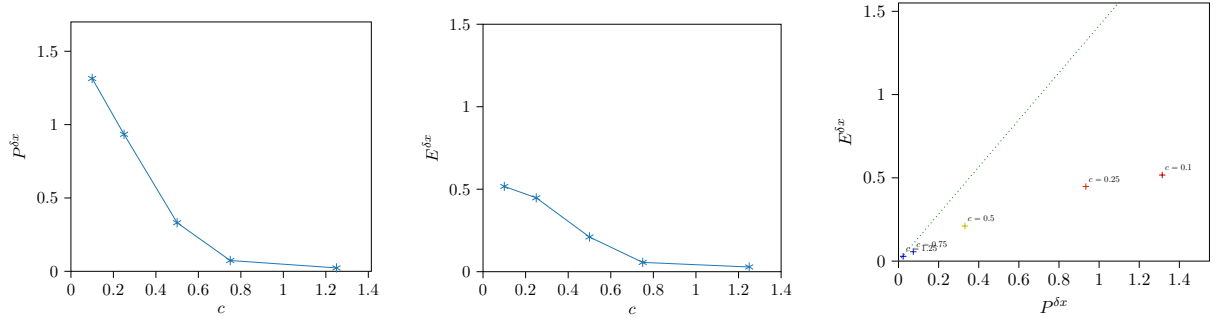

\subsection{Solitons with potential \eqref{eq:potunif}}
\label{subsec:E3}
We study now  the rectangular potential $\mathcal W_\lambda$ defined in \eqref{eq:potunif} in Example~3. 
We start by computing numerically nontrivial minimizers of the function $F_\lambda^{\delta x}$ in \eqref{eq:defF} for $\lambda=2.0$ and several values of $c$ as described in Subsection \ref{subsec:algo}.
The numerical experiments are carried out with $L=50$, $N=999$, so that $\delta x = 0.05$, and $\varepsilon^2=\delta x/4$.
In this case, $\delta \xi=2\pi/L\sim 0.13$.
The results are displayed in Figures \ref{fig:etadec5} and \ref{fig:Edec5}.
As seen in Figure~\ref{dispersion:E3}, there is a roton minimum 
and the Landau speed  $c_L \# 1.374$ is very close to the speed of sound.
This agrees of the small bump on $\eta_k$, except for the soliton  with speed $c=1.25$, whose 
profile $\bm\eta_k$ and phase $\bm \eta_k$
 seem to be strictly monotone on $\R^+$. As in all previous examples, 
 the monotonicity of  $c\mapsto P^{\delta}$ indicates that these are stable solitons.

We perform another experiment for $\lambda=4.5$, with $L=80$, $N=1599$ and a tolerance of $\varepsilon^1={\delta x}/4=1.25{\rm e}-2$.
The results are displayed in Figures \ref{fig:PotUniflambda4.5-etatheta}
and \ref{fig:PotUnifLambda4.5-NRJ}.

In this case, $c_L \# .624$ and the dispersion curve is depicted in  Figure~\ref{dispersion:E3}. We 
found a change on the behavior of soliton similar to one detected in 
Subsection~\ref{sebsec:potGaussien}, i.e.\ a more oscillating behavior of the solitons.
 Concerning the stability, the curve  $P^{\delta x}$ is strictly decreasing for $c\leq 0
 75$, but it is not clear the exact behavior of the curve for $c>0.75$.



\begin{figure}[!ht]
\centering
\begin{tabular}{cc}
\resizebox{0.5\textwidth}{!}{
\includegraphics{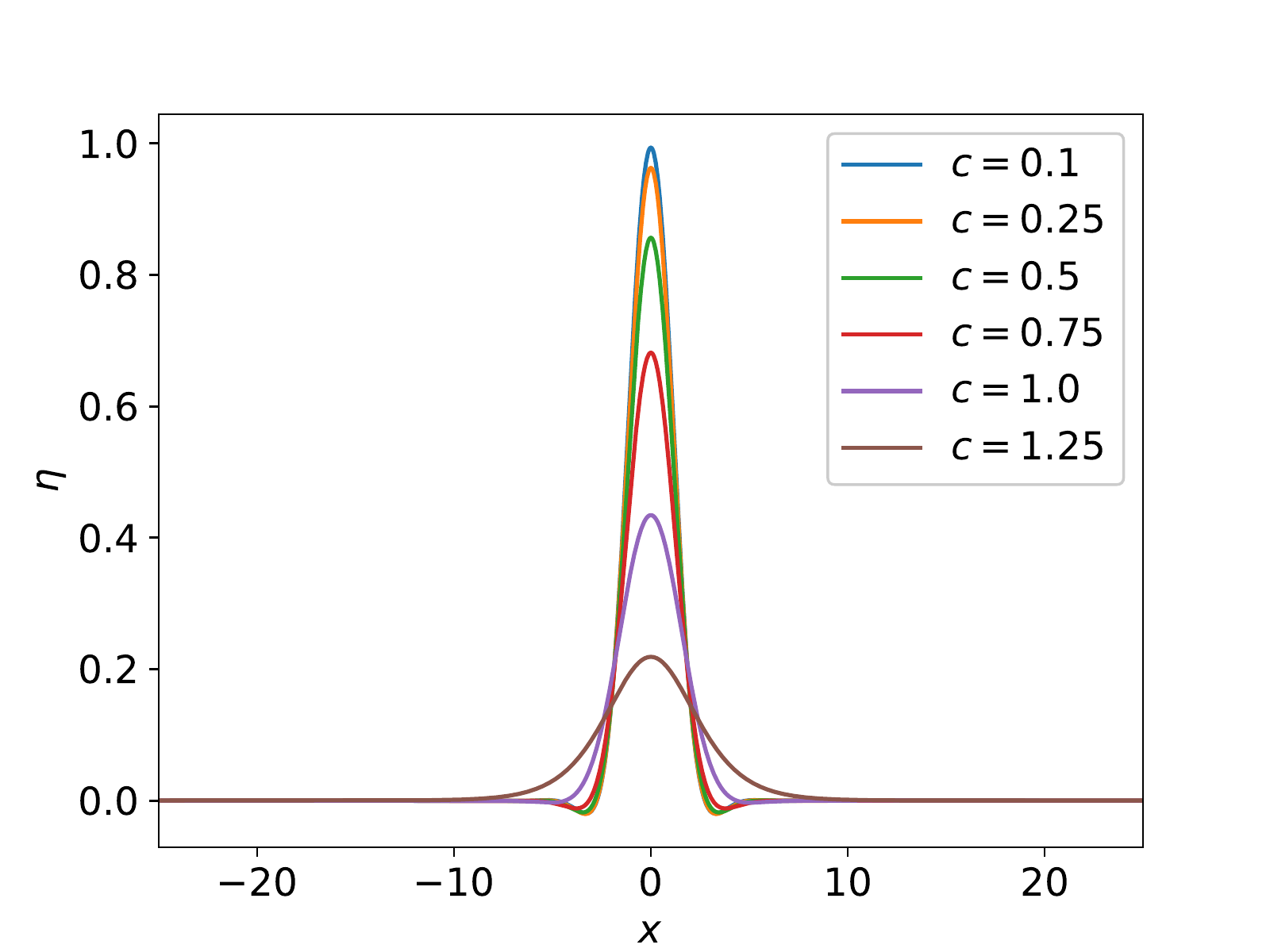}}%
&
\resizebox{0.5\textwidth}{!}{\includegraphics{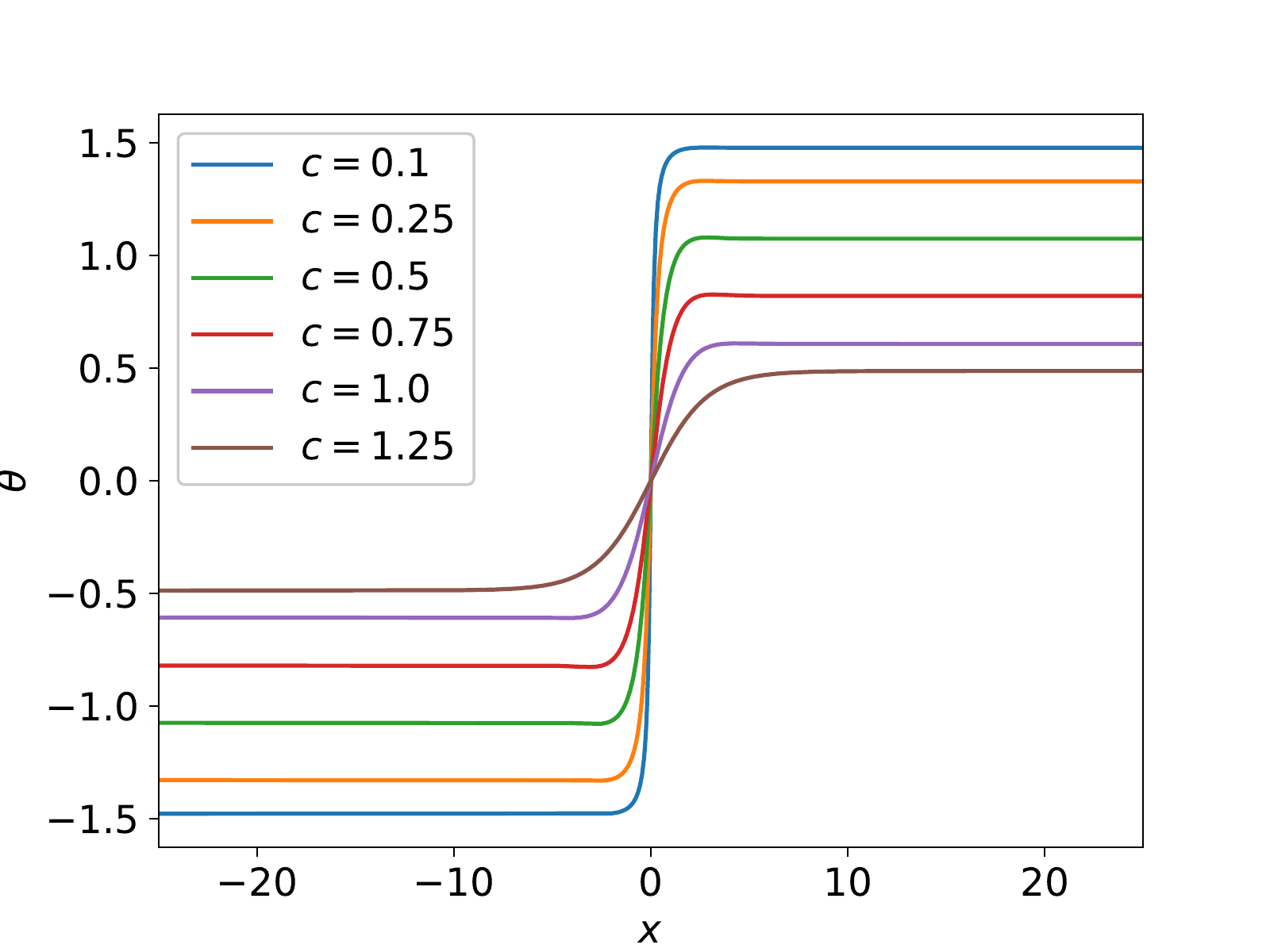}}
\end{tabular}
\caption{Numerically computed solitons for potential \eqref{eq:potunif}
with $\lambda=2.0$, showing 
 $\bm\eta_k$ (left panel) and  ${\bm\theta}_k$ (right panel)
 as a function of $x_k$.}
\label{fig:etadec5}
\end{figure}

\begin{figure}[!ht]
\centering
\begin{tabular}{ccc}
\resizebox{0.33\textwidth}{!}{
\begin{tikzpicture}

\definecolor{color0}{rgb}{0.12156862745098,0.466666666666667,0.705882352941177}

\begin{axis}[
tick align=outside,
tick pos=left,
x grid style={white!69.0196078431373!black},
xlabel={\(\displaystyle c\)},
xmin=0, xmax=1.4142135623731,
xtick style={color=black},
y grid style={white!69.0196078431373!black},
ylabel={\(\displaystyle P^{\delta x}\)},
ymin=0, ymax=1.7,
ytick style={color=black}
]
\addplot [semithick, color0, mark=asterisk, mark size=3, mark options={solid}]
table {%
0.1 1.41192270518982
0.25 1.1712033931801
0.5 0.788212949383917
0.75 0.452775887848796
1 0.201671327405296
1.25 0.0732963127209293
};
\end{axis}

\end{tikzpicture}} &
\resizebox{0.33\textwidth}{!}{
\begin{tikzpicture}

\definecolor{color0}{rgb}{0.12156862745098,0.466666666666667,0.705882352941177}

\begin{axis}[
tick align=outside,
tick pos=left,
x grid style={white!69.0196078431373!black},
xlabel={\(\displaystyle c\)},
xmin=0, xmax=1.4,
xtick style={color=black},
y grid style={white!69.0196078431373!black},
ylabel={\(\displaystyle E^{\delta x}\)},
ymin=0, ymax=1.5,
ytick style={color=black}
]
\addplot [semithick, color0, mark=asterisk, mark size=3, mark options={solid}]
table {%
0.1 0.825810352921703
0.25 0.784897720147318
0.5 0.645475857429304
0.75 0.441319103199709
1 0.22962453180464
1.25 0.0933034899039156
};
\end{axis}

\end{tikzpicture}} & \resizebox{0.33\textwidth}{!}{
\begin{tikzpicture}

\begin{axis}[
tick align=outside,
tick pos=left,
x grid style={white!69.0196078431373!black},
xlabel={\(\displaystyle P^{\delta x}\)},
xmin=0, xmax=1.55,
xtick style={color=black},
y grid style={white!69.0196078431373!black},
ylabel={\(\displaystyle E^{\delta x}\)},
ymin=0, ymax=1.55,
ytick style={color=black}
]
\addplot [draw=blue, fill=blue, mark=+, only marks, scatter]
table{%
x  y
1.41192270518982 0.825810352921703
1.1712033931801 0.784897720147318
0.788212949383917 0.645475857429304
0.452775887848796 0.441319103199709
0.201671327405296 0.22962453180464
0.0732963127209293 0.0933034899039156
};
\addplot [semithick, green!50.1960784313725!black, dotted]
table {%
0 0
1.41192270518982 1.99676023870195
};
\draw (axis cs:1.41192270518982,0.855810352921703) node[
  scale=0.5,
  anchor=base west,
  text=black,
  rotate=0.0
]{$c=0.1$};
\draw (axis cs:1.1712033931801,0.814897720147318) node[
  scale=0.5,
  anchor=base west,
  text=black,
  rotate=0.0
]{$c=0.25$};
\draw (axis cs:0.788212949383917,0.675475857429304) node[
  scale=0.5,
  anchor=base west,
  text=black,
  rotate=0.0
]{$c=0.5$};
\draw (axis cs:0.452775887848796,0.471319103199709) node[
  scale=0.5,
  anchor=base west,
  text=black,
  rotate=0.0
]{$c=0.75$};
\draw (axis cs:0.201671327405296,0.25962453180464) node[
  scale=0.5,
  anchor=base west,
  text=black,
  rotate=0.0
]{$c=1.0$};
\draw (axis cs:0.0732963127209293,0.123303489903916) node[
  scale=0.5,
  anchor=base west,
  text=black,
  rotate=0.0
]{$c=1.25$};
\end{axis}

\end{tikzpicture}}
\end{tabular}
\caption{Numerically computed solitons for potential \eqref{eq:potunif}
with $\lambda=2.0$,
showing  $P^{\delta x}$ (left) and $E^{\delta x}$ (center) as a function
of $c$. On the right, $E^{\delta x}$ as a function of $P^{\delta x}$.
}
\label{fig:Edec5}
\end{figure}
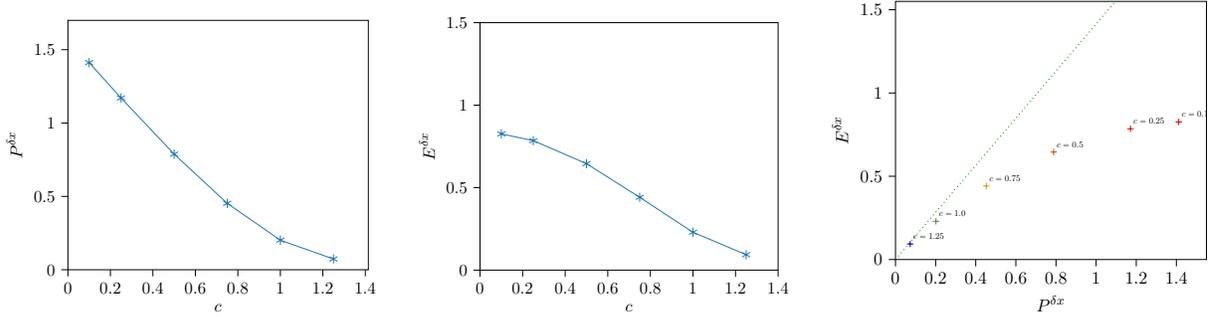


\begin{figure}[!ht]
\centering
\begin{tabular}{cc}
\resizebox{0.5\textwidth}{!}{
\includegraphics{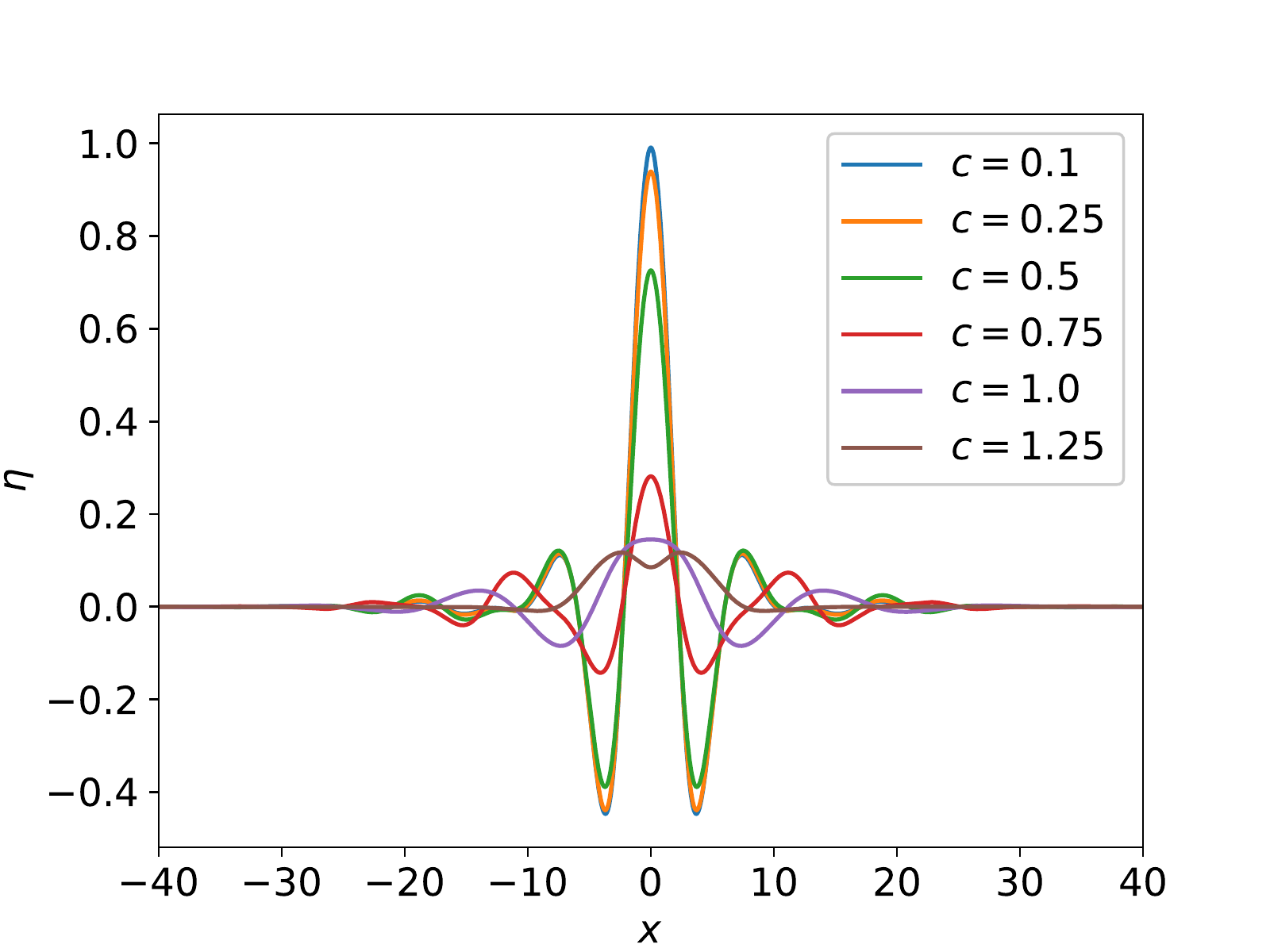}}%
&
\resizebox{0.5\textwidth}{!}{\includegraphics{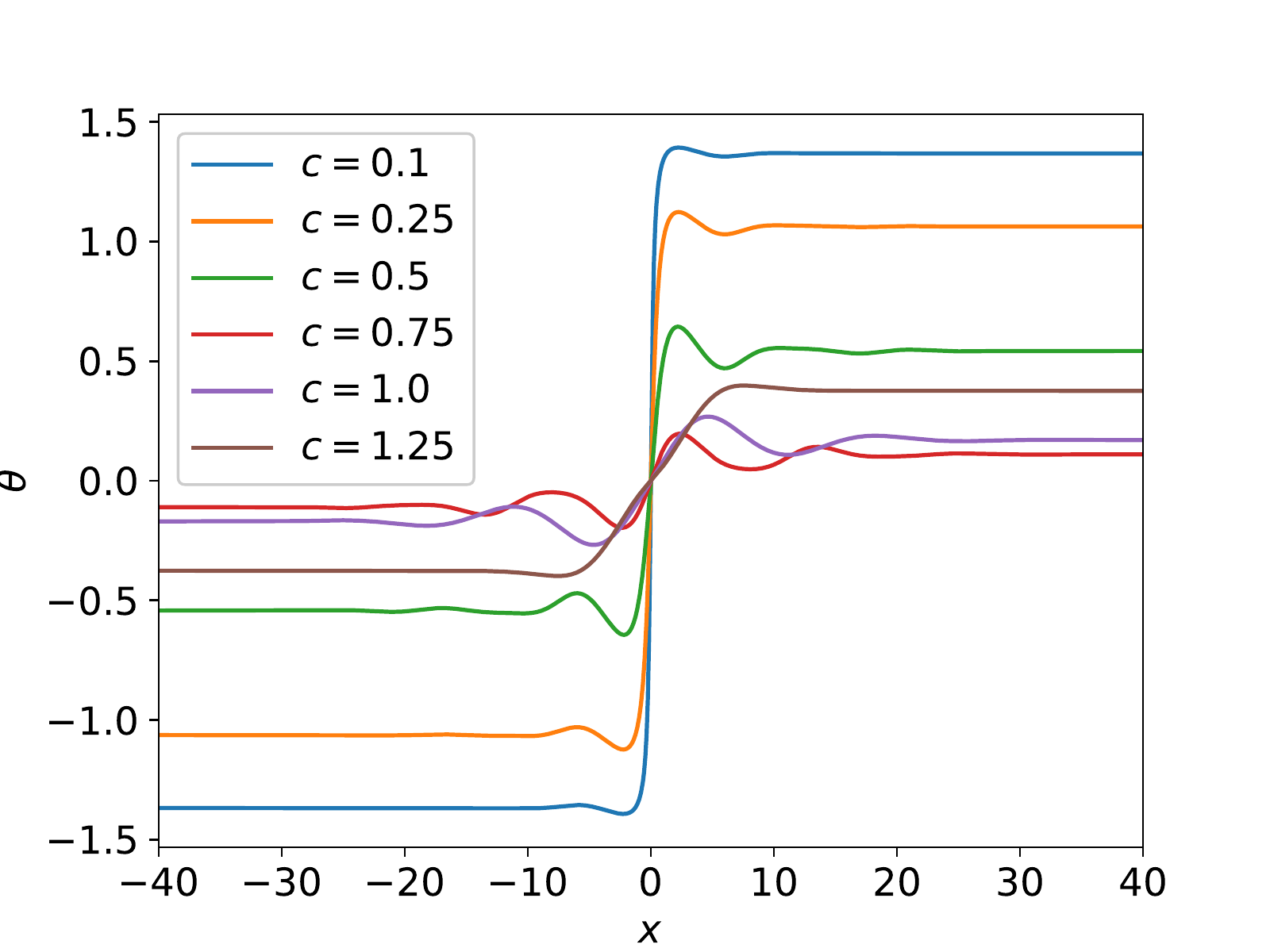}}
\end{tabular}
\caption{Numerically computed solitons for potential \eqref{eq:potunif}
with $\lambda=4.5$, showing 
 $\bm\eta_k$ (left) and  ${\bm\theta}_k$ (right) as
a function of $x_k$. 
}
\label{fig:PotUniflambda4.5-etatheta}
\end{figure}

\begin{figure}[!ht]
\centering
\begin{tabular}{ccc}
\resizebox{0.33\textwidth}{!}{
\begin{tikzpicture}

\definecolor{color0}{rgb}{0.12156862745098,0.466666666666667,0.705882352941177}

\begin{axis}[
tick align=outside,
tick pos=left,
x grid style={white!69.0196078431373!black},
xlabel={\(\displaystyle c\)},
xmin=0, xmax=1.4142135623731,
xtick style={color=black},
y grid style={white!69.0196078431373!black},
ylabel={\(\displaystyle P^{\delta x}\)},
ymin=0, ymax=1.7,
ytick style={color=black}
]
\addplot [semithick, color0, mark=asterisk, mark size=3, mark options={solid}]
table {%
0.1 1.33696404391742
0.25 0.993489048314586
0.5 0.455418614490227
0.75 0.0660329349060465
1 0.0455363453283852
1.25 0.0379953756083166
};
\end{axis}

\end{tikzpicture}} &
\resizebox{0.33\textwidth}{!}{
\begin{tikzpicture}

\definecolor{color0}{rgb}{0.12156862745098,0.466666666666667,0.705882352941177}

\begin{axis}[
tick align=outside,
tick pos=left,
x grid style={white!69.0196078431373!black},
xlabel={\(\displaystyle c\)},
xmin=0, xmax=1.4,
xtick style={color=black},
y grid style={white!69.0196078431373!black},
ylabel={\(\displaystyle E^{\delta x}\)},
ymin=0, ymax=1.5,
ytick style={color=black}
]
\addplot [semithick, color0, mark=asterisk, mark size=3, mark options={solid}]
table {%
0.1 0.552144855200508
0.25 0.490926108406892
0.5 0.285589172773583
0.75 0.0537423010420959
1 0.0501188245117822
1.25 0.0476246124318716
};
\end{axis}

\end{tikzpicture}} & \resizebox{0.33\textwidth}{!}{
\begin{tikzpicture}

\begin{axis}[
tick align=outside,
tick pos=left,
x grid style={white!69.0196078431373!black},
xlabel={\(\displaystyle P^{\delta x}\)},
xmin=0, xmax=1.55,
xtick style={color=black},
y grid style={white!69.0196078431373!black},
ylabel={\(\displaystyle E^{\delta x}\)},
ymin=0, ymax=1.55,
ytick style={color=black}
]
\addplot [draw=blue, fill=blue, mark=+, only marks, scatter]
table{%
x  y
1.33696404391742 0.552144855200508
0.993489048314586 0.490926108406892
0.455418614490227 0.285589172773583
0.0660329349060465 0.0537423010420959
0.0455363453283852 0.0501188245117822
0.0379953756083166 0.0476246124318716
};
\addplot [semithick, green!50.1960784313725!black, dotted]
table {%
0 0
1.33696404391742 1.89075268331319
};
\draw (axis cs:1.33696404391742,0.582144855200508) node[
  scale=0.5,
  anchor=base west,
  text=black,
  rotate=0.0
]{$c=0.1$};
\draw (axis cs:0.993489048314586,0.520926108406892) node[
  scale=0.5,
  anchor=base west,
  text=black,
  rotate=0.0
]{$c=0.25$};
\draw (axis cs:0.455418614490227,0.315589172773583) node[
  scale=0.5,
  anchor=base west,
  text=black,
  rotate=0.0
]{$c=0.5$};
\draw (axis cs:0.0660329349060465,0.0837423010420959) node[
  scale=0.5,
  anchor=base west,
  text=black,
  rotate=0.0
]{$c=0.75$};
\draw (axis cs:0.0455363453283852,0.0801188245117822) node[
  scale=0.5,
  anchor=base west,
  text=black,
  rotate=0.0
  ]{
  };
\draw (axis cs:0.0379953756083166,0.0776246124318716) node[
  scale=0.5,
  anchor=base west,
  text=black,
  rotate=0.0
  ]{
  };
\end{axis}

\end{tikzpicture}}
\end{tabular}
\caption{Numerically computed solitons for potential \eqref{eq:potunif}
with $\lambda=4.5$, showing 
$P^{\delta x}$ (left) and 
$E^{\delta x}$ (center) as a function
of $c$. On the right,  $E^{\delta x}$ as a function of $P^{\delta x}$.
}
\label{fig:PotUnifLambda4.5-NRJ}
\end{figure}
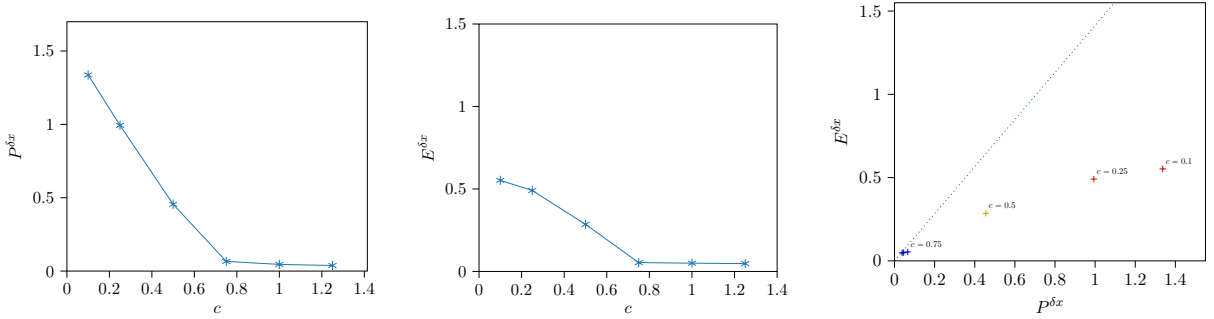


 \subsection{Solitons with potential \eqref{eq:pot3Dirac}}
In this subsection, we consider the  potential $\mathcal W_\lambda$ in \eqref{eq:pot3Dirac} in Example 4, given by three Dirac delta functions,
so that the Landau speed is $c_L(\lambda)=\sqrt 2$, for any $\lambda\geq 0$, 
and there is no roton minimum. We will investigate numerical soliton for $\lambda=2$ and $\lambda=10$, whose dispersion curve are given in Figure~\ref{dispersion:E4}.

We first compute numerically nontrivial minimizers of  $F_\lambda^{\delta x}$
in \eqref{eq:defF} for $\lambda=2.0$ using
the algorithm described in Subsection \ref{subsec:algo}.
The numerical experiments are carried out with $L=50$, $N=999$, so that $\delta x = 0.05$, and $\varepsilon^2=\delta x/4$.
In this case, $\delta \xi=2\pi/L\sim 0.13$.
The results are displayed in Figures~\ref{fig:etadec8} and \ref{fig:Edec8}.
As expected, the aspect of the dispersion curve in the left panel of Figure~\ref{dispersion:E4} 
agrees with the monotonicity of the numerical solitons  in  Figure~\ref{fig:etadec8}, which should 
be stable in view of the results in Figure~\ref{fig:Edec8}.
\begin{figure}[!ht]
	\centering
	\begin{tabular}{cc}
		\resizebox{0.5\textwidth}{!}{
			\includegraphics{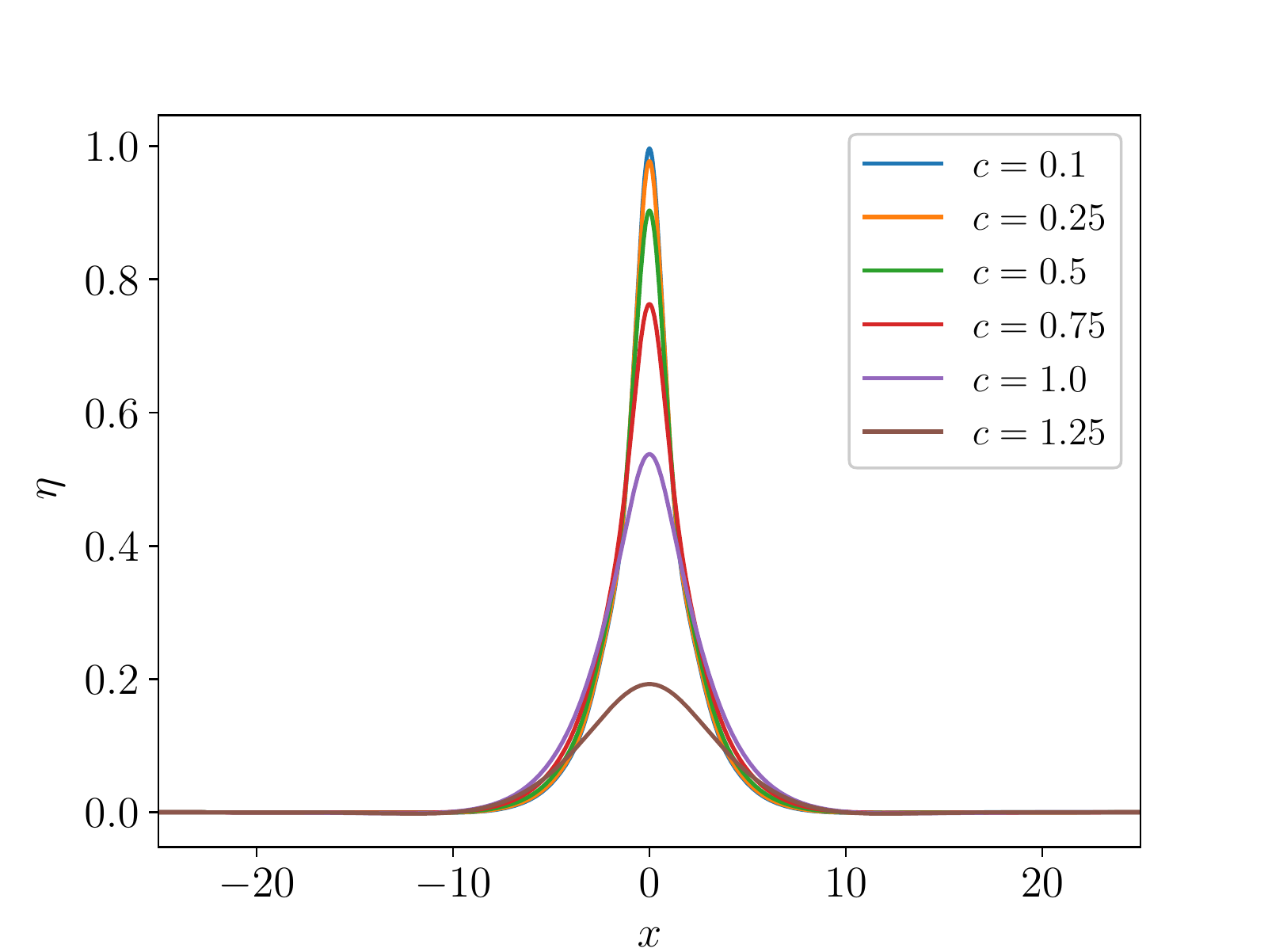}}%
	&
	\resizebox{0.5\textwidth}{!}{\includegraphics{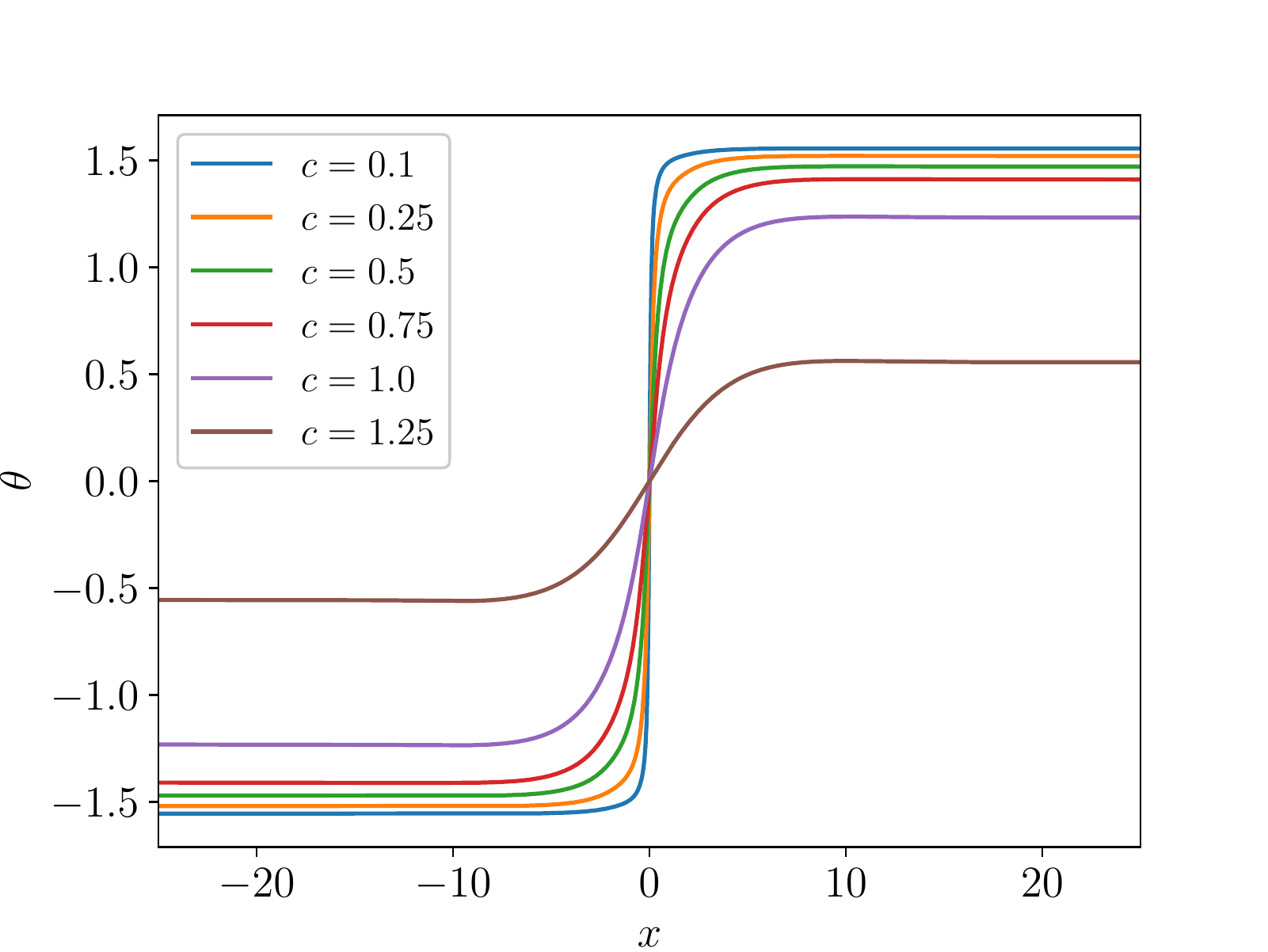}}
\end{tabular}
\caption{Numerically computed solitons for potential \eqref{eq:pot3Dirac}
with $\lambda=2.0$, showing 
$\bm\eta_k$ (left) and $\bm\theta_k$ (right)
as a function of $x_k$.}
\label{fig:etadec8}
\end{figure}
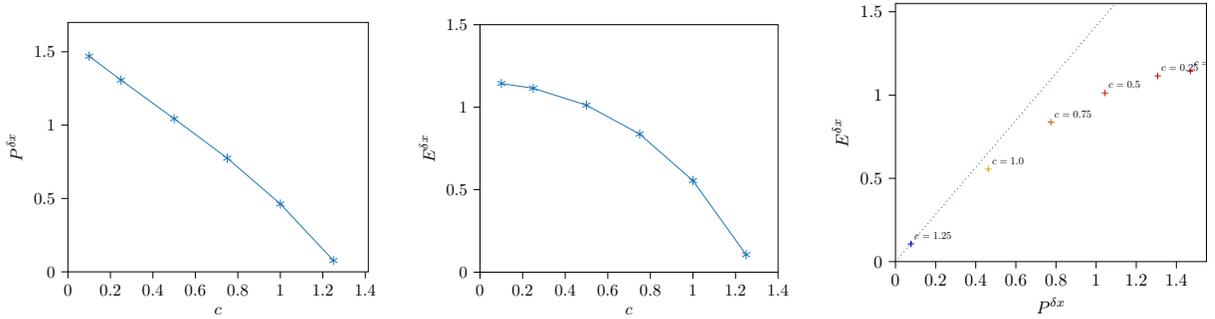
\begin{figure}[!ht]
	\centering
	\begin{tabular}{ccc}
		\resizebox{0.33\textwidth}{!}{
\begin{tikzpicture}

\definecolor{color0}{rgb}{0.12156862745098,0.466666666666667,0.705882352941177}

\begin{axis}[
tick align=outside,
tick pos=left,
x grid style={white!69.0196078431373!black},
xlabel={\(\displaystyle c\)},
xmin=0, xmax=1.4142135623731,
xtick style={color=black},
y grid style={white!69.0196078431373!black},
ylabel={\(\displaystyle P^{\delta x}\)},
ymin=0, ymax=1.7,
ytick style={color=black}
]
\addplot [semithick, color0, mark=asterisk, mark size=3, mark options={solid}]
table {%
0.1 1.46978294073712
0.25 1.30715695016549
0.5 1.04398863825895
0.75 0.774826269018571
1 0.462772657589714
1.25 0.0770393882024286
};
\end{axis}

\end{tikzpicture}} &                                                                            \resizebox{0.33\textwidth}{!}{
\begin{tikzpicture}

\definecolor{color0}{rgb}{0.12156862745098,0.466666666666667,0.705882352941177}

\begin{axis}[
tick align=outside,
tick pos=left,
x grid style={white!69.0196078431373!black},
xlabel={\(\displaystyle c\)},
xmin=0, xmax=1.4,
xtick style={color=black},
y grid style={white!69.0196078431373!black},
ylabel={\(\displaystyle E^{\delta x}\)},
ymin=0, ymax=1.5,
ytick style={color=black}
]
\addplot [semithick, color0, mark=asterisk, mark size=3, mark options={solid}]
table {%
0.1 1.1439131454511
0.25 1.11504087070845
0.5 1.01227527737312
0.75 0.837108745229425
1 0.555025669715542
1.25 0.106161765461107
};
\end{axis}

\end{tikzpicture}} & \resizebox{0.33\textwidth}{!}{
\begin{tikzpicture}

\begin{axis}[
tick align=outside,
tick pos=left,
x grid style={white!69.0196078431373!black},
xlabel={\(\displaystyle P^{\delta x}\)},
xmin=0, xmax=1.55,
xtick style={color=black},
y grid style={white!69.0196078431373!black},
ylabel={\(\displaystyle E^{\delta x}\)},
ymin=0, ymax=1.55,
ytick style={color=black}
]
\addplot [draw=blue, fill=blue, mark=+, only marks, scatter]
table{%
x  y
1.46978294073712 1.1439131454511
1.30715695016549 1.11504087070845
1.04398863825895 1.01227527737312
0.774826269018571 0.837108745229425
0.462772657589714 0.555025669715542
0.0770393882024286 0.106161765461107
};
\addplot [semithick, green!50.1960784313725!black, dotted]
table {%
0 0
1.46978294073712 2.07858696853505
};
\draw (axis cs:1.46978294073712,1.1739131454511) node[
  scale=0.6,
  anchor=base west,
  text=black,
  rotate=0.0
]{$c=0.1$};
\draw (axis cs:1.30715695016549,1.14504087070845) node[
  scale=0.6,
  anchor=base west,
  text=black,
  rotate=0.0
]{$c=0.25$};
\draw (axis cs:1.04398863825895,1.04227527737312) node[
  scale=0.6,
  anchor=base west,
  text=black,
  rotate=0.0
]{$c=0.5$};
\draw (axis cs:0.774826269018571,0.867108745229425) node[
  scale=0.6,
  anchor=base west,
  text=black,
  rotate=0.0
]{$c=0.75$};
\draw (axis cs:0.462772657589714,0.585025669715542) node[
  scale=0.6,
  anchor=base west,
  text=black,
  rotate=0.0
]{$c=1.0$};
\draw (axis cs:0.0770393882024286,0.136161765461107) node[
  scale=0.6,
  anchor=base west,
  text=black,
  rotate=0.0
]{$c=1.25$};
\end{axis}

\end{tikzpicture}}
	\end{tabular}
	\caption{Numerically computed solitons for potential \eqref{eq:pot3Dirac}
		with $\lambda=2.0$, showing
$P^{\delta x}$ (left) and $E^{\delta x}$ (center) as a function
		of $c$. On the right, $E^{\delta x}$ as a function of $P^{\delta x}$.
	}
	\label{fig:Edec8}
\end{figure}

We perform another experiment for $\lambda=10.0$ with $L=300$, $N=5999$, so that $\delta x = 0.05$
and $\varepsilon^2/\delta x/4$. In this case, $\delta\xi=2\pi/L\sim0.021$.
The numerical results are displayed in Figures~\ref{fig:Pot3Dirac-lambda=10.0-etatheta}
and \ref{fig:Pot3Dirac-lambda=10.0-NRJ}.
Even though there is no roton minimum and the Landau speed coincides with the speed of sound, we 
see an oscillating behavior of the solutions. 
In view of the dispersion curve in Figure~\ref{dispersion:E4}, we conjecture that the oscillation are related to the existence of inflection points of the dispersion curve. Once again, these solitons
should be stable in view of the curves in Figure~\ref{fig:Pot3Dirac-lambda=10.0-NRJ}.

Another point is that, despite the oscillations in ${\bm \eta}$ in Figure
\ref{fig:Pot3Dirac-lambda=10.0-etatheta}, the minimum values are of order {\it minus} a few $10^{-3}$
for all speed.
This may indicate that the corresponding continuous solitons $\eta$ are indeed non-negative, so that
the corresponding function $u$ has modulus bounded by $1$.
\begin{figure}[!ht]
\centering
\begin{tabular}{cc}
\resizebox{0.5\textwidth}{!}{
	\includegraphics{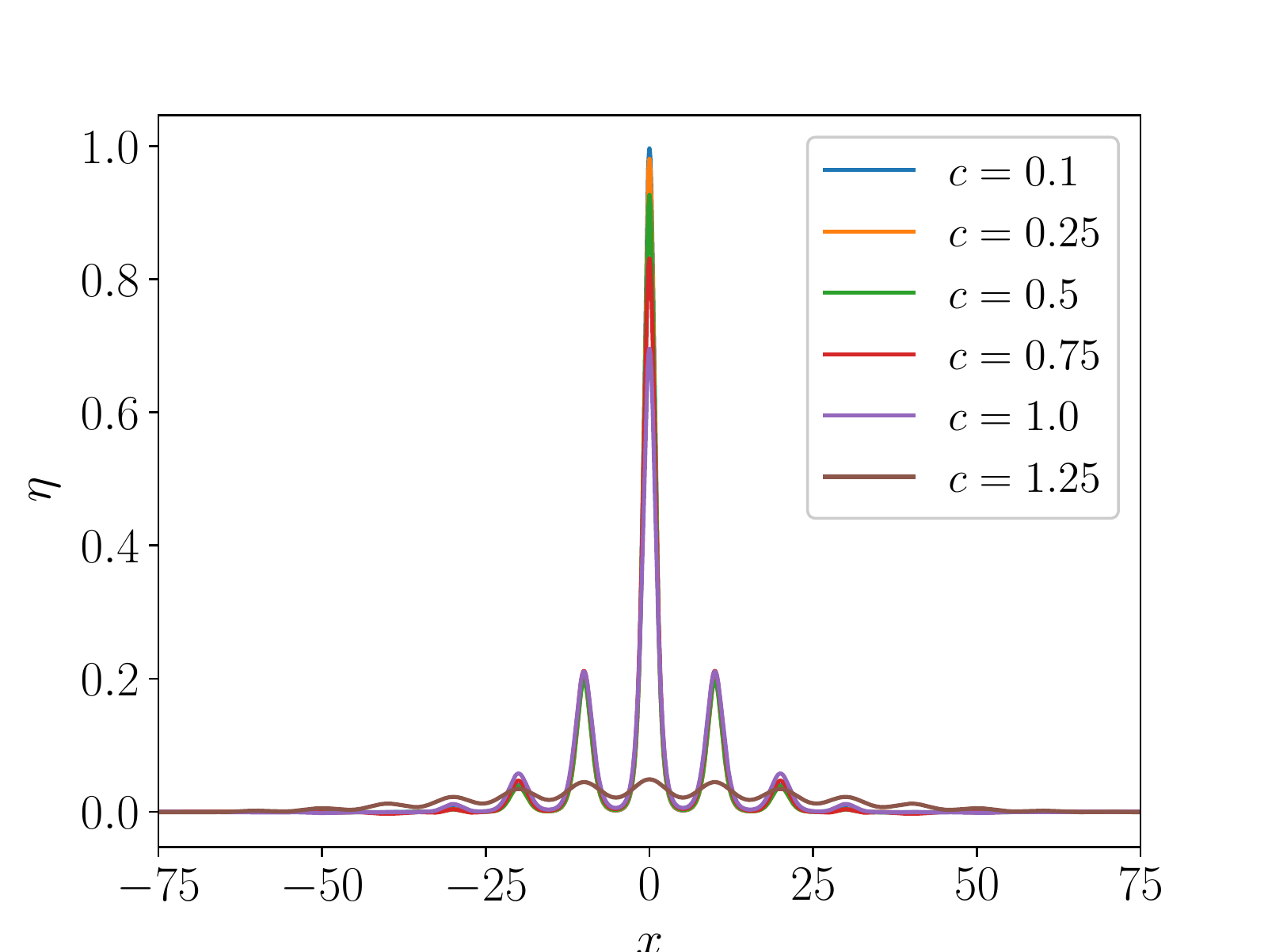}}%
&
\resizebox{0.5\textwidth}{!}{\includegraphics{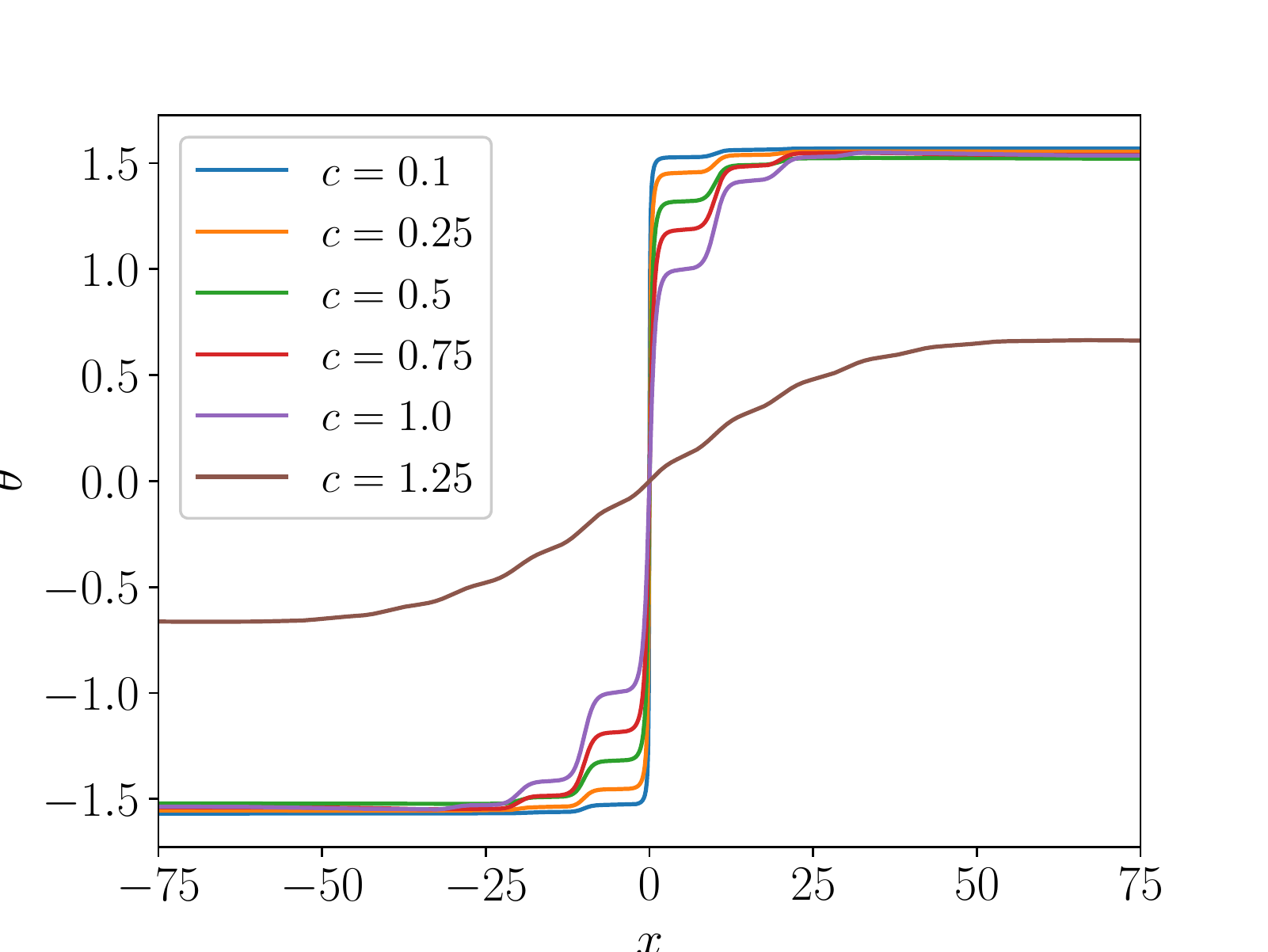}}
\end{tabular}
\caption{Numerically computed solitons for potential \eqref{eq:pot3Dirac}
with $\lambda=10.0$, showing 
 $\bm\eta_k$ (left) and ${\bm \theta}_k$ (right)
as a function of $x_k$.
}
\label{fig:Pot3Dirac-lambda=10.0-etatheta}
\end{figure}
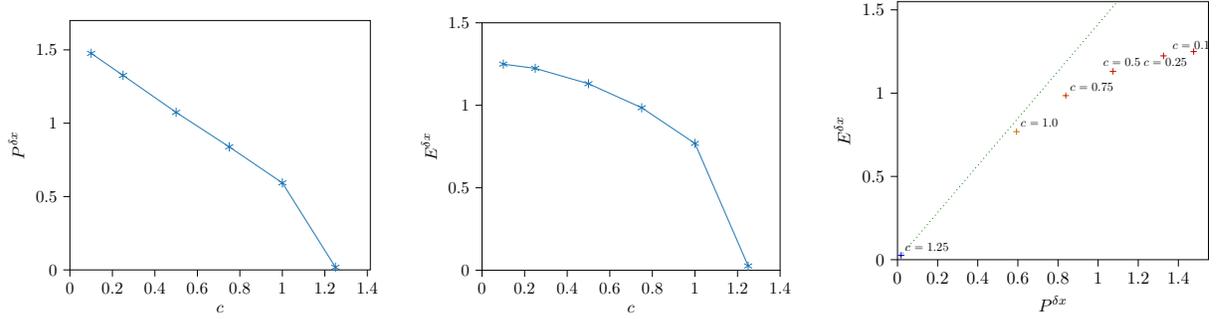
\begin{figure}[!ht]
\centering
\begin{tabular}{ccc}
\resizebox{0.33\textwidth}{!}{
\begin{tikzpicture}

\definecolor{color0}{rgb}{0.12156862745098,0.466666666666667,0.705882352941177}

\begin{axis}[
tick align=outside,
tick pos=left,
x grid style={white!69.0196078431373!black},
xlabel={\(\displaystyle c\)},
xmin=0, xmax=1.4142135623731,
xtick style={color=black},
y grid style={white!69.0196078431373!black},
ylabel={\(\displaystyle P^{\delta x}\)},
ymin=0, ymax=1.7,
ytick style={color=black}
]
\addplot [semithick, color0, mark=asterisk, mark size=3, mark options={solid}]
table {%
0.1 1.47606777517304
0.25 1.32616463841537
0.5 1.07401376169404
0.75 0.838845538219547
1 0.593192746104716
1.25 0.0180665023790393
};
\end{axis}

\end{tikzpicture}} &                                                                            \resizebox{0.33\textwidth}{!}{
\begin{tikzpicture}

\definecolor{color0}{rgb}{0.12156862745098,0.466666666666667,0.705882352941177}

\begin{axis}[
tick align=outside,
tick pos=left,
x grid style={white!69.0196078431373!black},
xlabel={\(\displaystyle c\)},
xmin=0, xmax=1.4,
xtick style={color=black},
y grid style={white!69.0196078431373!black},
ylabel={\(\displaystyle E^{\delta x}\)},
ymin=0, ymax=1.5,
ytick style={color=black}
]
\addplot [semithick, color0, mark=asterisk, mark size=3, mark options={solid}]
table {%
0.1 1.24956301693226
0.25 1.22389762224068
0.5 1.1308544060172
0.75 0.985197870276432
1 0.769070102446899
1.25 0.0260814387616776
};
\end{axis}

\end{tikzpicture}} & \resizebox{0.33\textwidth}{!}{
\begin{tikzpicture}

\begin{axis}[
tick align=outside,
tick pos=left,
x grid style={white!69.0196078431373!black},
xlabel={\(\displaystyle P^{\delta x}\)},
xmin=0, xmax=1.55,
xtick style={color=black},
y grid style={white!69.0196078431373!black},
ylabel={\(\displaystyle E^{\delta x}\)},
ymin=0, ymax=1.55,
ytick style={color=black}
]
\addplot [draw=blue, fill=blue, mark=+, only marks, scatter]
table{%
x  y
1.47606777517304 1.24956301693226
1.32616463841537 1.22389762224068
1.07401376169404 1.1308544060172
0.838845538219547 0.985197870276432
0.593192746104716 0.769070102446899
0.0180665023790393 0.0260814387616776
};
\addplot [semithick, green!50.1960784313725!black, dotted]
table {%
0 0
1.47606777517304 2.08747506663159
};
\draw (axis cs:1.35006777517304,1.26456301693226) node[
  scale=0.7,
  anchor=base west,
  text=black,
  rotate=0.0
]{$c=0.1$};
\draw (axis cs:1.20616463841537,1.16089762224068) node[
  scale=0.7,
  anchor=base west,
  text=black,
  rotate=0.0
]{$c=0.25$};
\draw (axis cs:1.00401376169404,1.1608544060172) node[
  scale=0.7,
  anchor=base west,
  text=black,
  rotate=0.0
]{$c=0.5$};
\draw (axis cs:0.838845538219547,1.01519787027643) node[
  scale=0.7,
  anchor=base west,
  text=black,
  rotate=0.0
]{$c=0.75$};
\draw (axis cs:0.593192746104716,0.799070102446899) node[
  scale=0.7,
  anchor=base west,
  text=black,
  rotate=0.0
]{$c=1.0$};
\draw (axis cs:0.0180665023790393,0.0560814387616776) node[
  scale=0.7,
  anchor=base west,
  text=black,
  rotate=0.0
]{$c=1.25$};
\end{axis}

\end{tikzpicture}}
\end{tabular}
\caption{Numerically computed solitons for potential \eqref{eq:pot3Dirac}
with $\lambda=10.0$, showing 
 $P^{\delta x}$ (left) and $E^{\delta x}$ (right) as a function
of $c$. On the right, $E^{\delta x}$ as a function of $P^{\delta x}$.
}
\label{fig:Pot3Dirac-lambda=10.0-NRJ}
\end{figure}

 \subsection{Solitons with potential \eqref{eq:potBR}}
In this subsection, we consider the Bochner-Riesz potential $\mathcal W_\lambda$ defined in \eqref{eq:potBR}, for $\lambda=2$ and $\lambda=4$,
as explained in Example~5.

We compute first nontrivial minimizers of the function $F_\lambda^{\delta x}$
for $\lambda=1.0$, as before, taking $L=200$ and $N=3999$, so that $\delta x = 0.05$, and $\varepsilon^2=\delta x/4$. In this case, $\delta \xi=2\pi/L\sim 0.13$. The results are displayed in Figures~\ref{fig:PotBR-lambda=1.0-etatheta} and \ref{fig:PotBR-lambda=1.0-NRJ}.
Clearly, $\lambda=1$ is a critical case for the dispersion curve, since it is given by  a straight line in interval $[0,\sqrt 2]$, as seen in Figure~\ref{dispersion:E5}. Even though $\omega_\lambda$ has no inflection point, the non-differentiability at $\xi=\sqrt 2$ could be related to the oscillations of solitons in Figure~\ref{fig:PotBR-lambda=1.0-etatheta}. 
These solitons should be stable in view of  Figure~\ref{fig:PotBR-lambda=1.0-NRJ}.

We run another experiment with $\lambda=4.0$, $L=300$, $N=5999$, so that $\delta x = 0.05$.
The results are displayed in Figures \ref{fig:PotBR-lambda=4.0-etatheta} and \ref{fig:PotBR-lambda=4.0-NRJ}. Even though in this case $c_L\#0.7$, 
we obtain again solitons for every subsonic speed, that seem stable
since $c\mapsto P^{\delta x}$ is decreasing.  
The more oscillating behavior of the soliton, could be explained by the gap between $c_L$ and $\sqrt 2$, as in Subsection~\ref{subsec:E3}.



\begin{figure}[!ht]
	\centering
	\begin{tabular}{cc}
		\resizebox{0.5\textwidth}{!}{\includegraphics{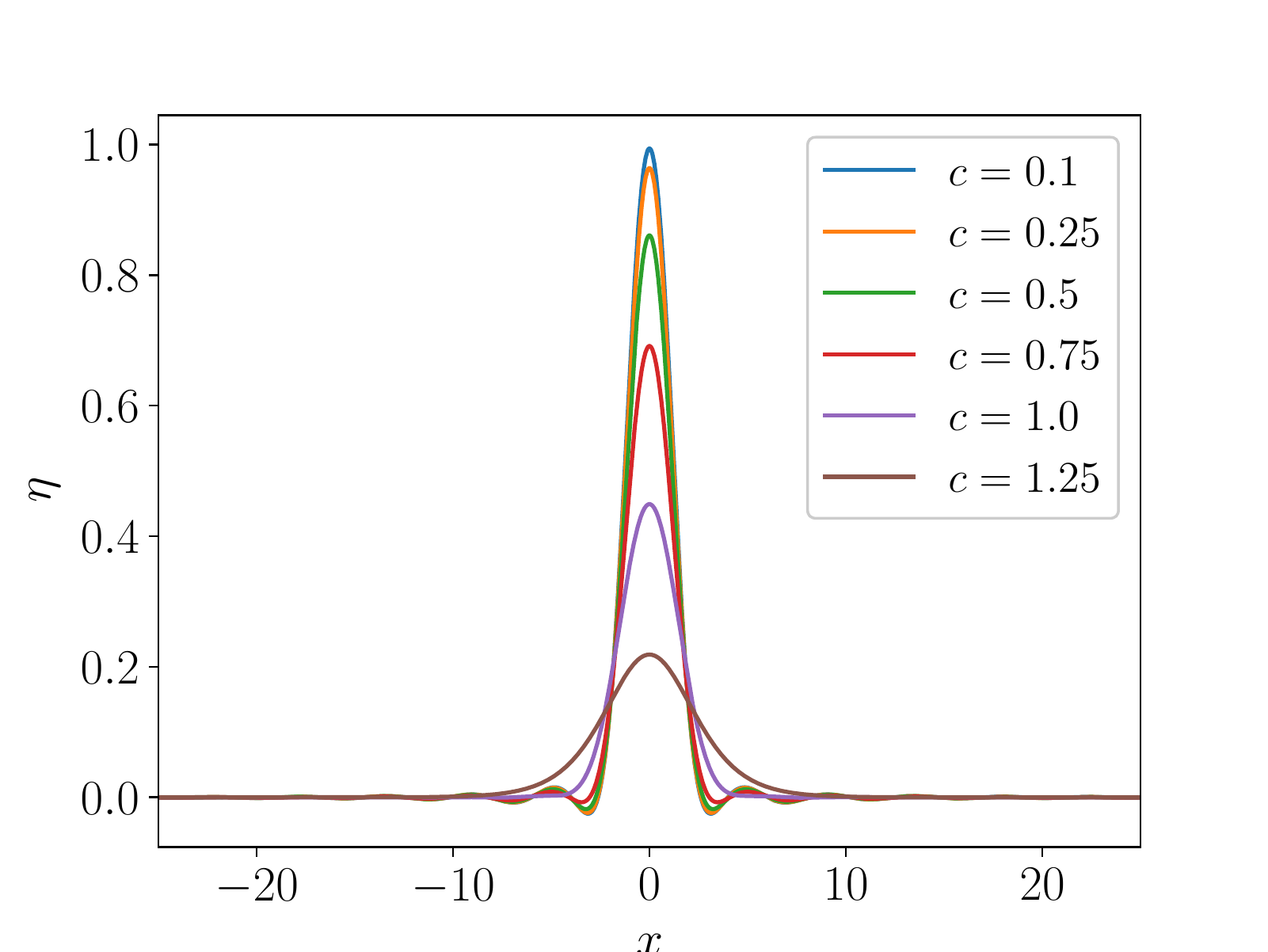}}
		&                                                                            
		\resizebox{0.5\textwidth}{!}{\includegraphics{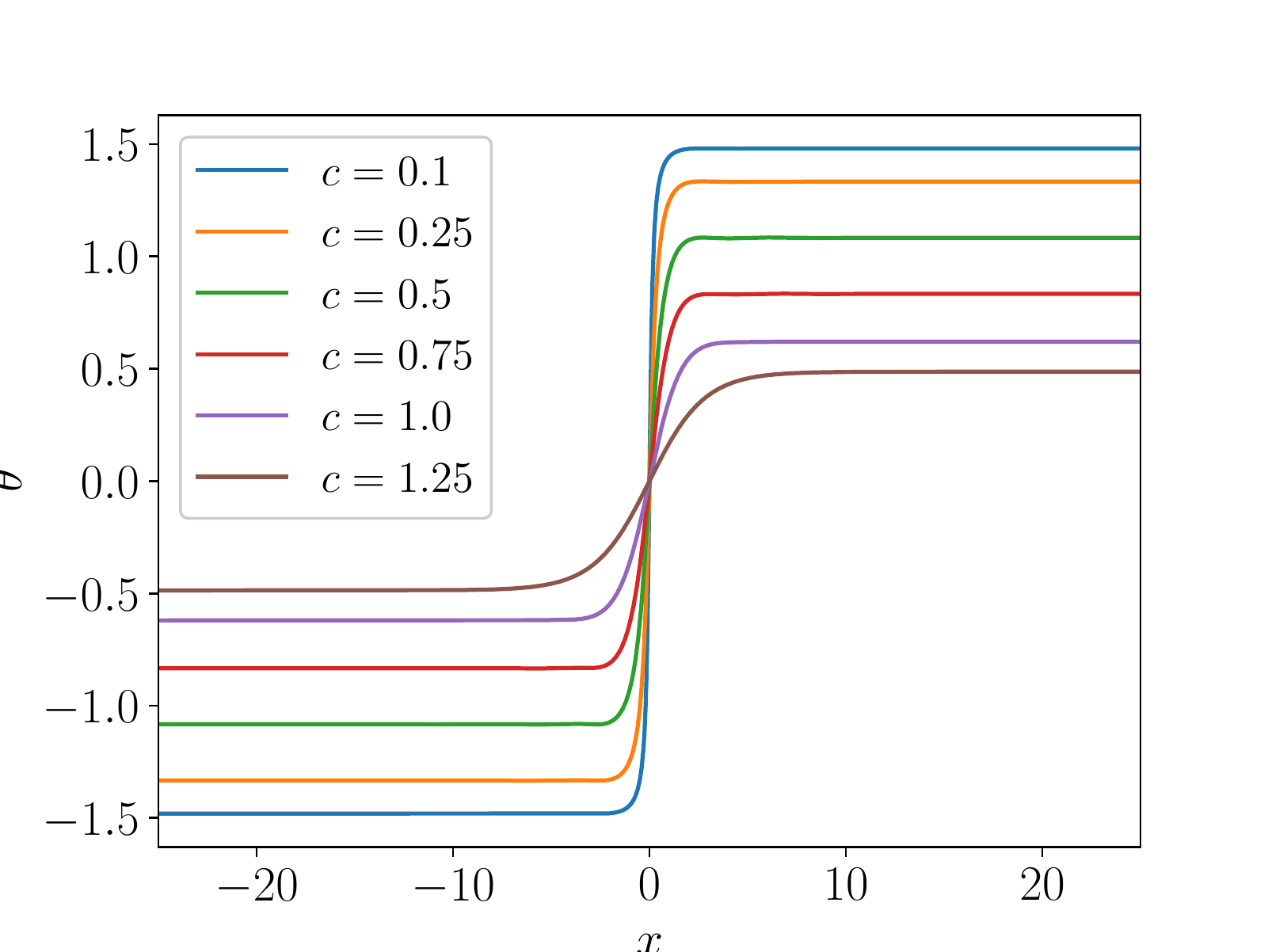}}
\end{tabular}
\caption{Numerically computed solitons for potential \eqref{eq:potBR}
	with $\lambda=1.0$, showing 
$\bm\eta_k$ (left) and ${\bm \theta}_k$ (right) as
	a function of $x_k$. 
}
\label{fig:PotBR-lambda=1.0-etatheta}
\end{figure}

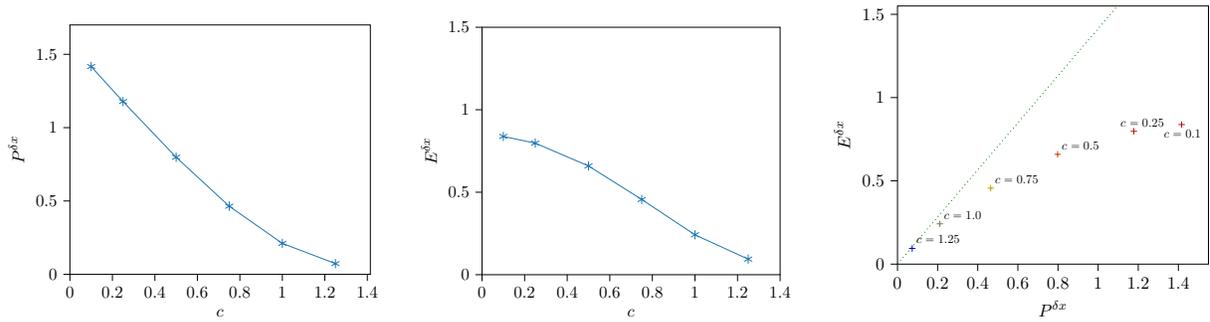
\begin{figure}[!ht]
\centering
\begin{tabular}{ccc}
	\resizebox{0.33\textwidth}{!}{
\begin{tikzpicture}

\definecolor{color0}{rgb}{0.12156862745098,0.466666666666667,0.705882352941177}

\begin{axis}[
tick align=outside,
tick pos=left,
x grid style={white!69.0196078431373!black},
xlabel={\(\displaystyle c\)},
xmin=0, xmax=1.4142135623731,
xtick style={color=black},
y grid style={white!69.0196078431373!black},
ylabel={\(\displaystyle P^{\delta x}\)},
ymin=0, ymax=1.7,
ytick style={color=black}
]
\addplot [semithick, color0, mark=asterisk, mark size=3, mark options={solid}]
table {%
0.1 1.41627000619312
0.25 1.17798770509715
0.5 0.79821840813834
0.75 0.464556642213112
1 0.21099240214293
1.25 0.0732963127209309
};
\end{axis}

\end{tikzpicture}}
	&
	\resizebox{0.33\textwidth}{!}{
\begin{tikzpicture}

\definecolor{color0}{rgb}{0.12156862745098,0.466666666666667,0.705882352941177}

\begin{axis}[
tick align=outside,
tick pos=left,
x grid style={white!69.0196078431373!black},
xlabel={\(\displaystyle c\)},
xmin=0, xmax=1.4,
xtick style={color=black},
y grid style={white!69.0196078431373!black},
ylabel={\(\displaystyle E^{\delta x}\)},
ymin=0, ymax=1.5,
ytick style={color=black}
]
\addplot [semithick, color0, mark=asterisk, mark size=3, mark options={solid}]
table {%
0.1 0.838119622246435
0.25 0.797522333448849
0.5 0.659307936465034
0.75 0.456135183351586
1 0.241947240224075
1.25 0.0943453372797519
};
\end{axis}

\end{tikzpicture}} & \resizebox{0.33\textwidth}{!}{
\begin{tikzpicture}

\begin{axis}[
tick align=outside,
tick pos=left,
x grid style={white!69.0196078431373!black},
xlabel={\(\displaystyle P^{\delta x}\)},
xmin=0, xmax=1.55,
xtick style={color=black},
y grid style={white!69.0196078431373!black},
ylabel={\(\displaystyle E^{\delta x}\)},
ymin=0, ymax=1.55,
ytick style={color=black}
]
\addplot [draw=blue, fill=blue, mark=+, only marks, scatter]
table{%
x  y
1.41627000619312 0.838119622246435
1.17798770509715 0.797522333448849
0.79821840813834 0.659307936465034
0.464556642213112 0.456135183351586
0.21099240214293 0.241947240224075
0.0732963127209309 0.0943453372797519
};
\addplot [semithick, green!50.1960784313725!black, dotted]
table {%
0 0
1.41627000619312 2.00290825074053
};
\draw (axis cs:1.30627000619312,0.758119622246435) node[
  scale=0.7,
  anchor=base west,
  text=black,
  rotate=0.0
]{$c=0.1$};
\draw (axis cs:1.09098770509715,0.827522333448849) node[
  scale=0.7,
  anchor=base west,
  text=black,
  rotate=0.0
]{$c=0.25$};
\draw (axis cs:0.79821840813834,0.689307936465034) node[
  scale=0.7,
  anchor=base west,
  text=black,
  rotate=0.0
]{$c=0.5$};
\draw (axis cs:0.464556642213112,0.486135183351586) node[
  scale=0.7,
  anchor=base west,
  text=black,
  rotate=0.0
]{$c=0.75$};
\draw (axis cs:0.21099240214293,0.271947240224075) node[
  scale=0.7,
  anchor=base west,
  text=black,
  rotate=0.0
]{$c=1.0$};
\draw (axis cs:0.0732963127209309,0.124345337279752) node[
  scale=0.7,
  anchor=base west,
  text=black,
  rotate=0.0
]{$c=1.25$};
\end{axis}

\end{tikzpicture}}
\end{tabular}
\caption{Numerically computed solitons for potential \eqref{eq:potBR}
	with $\lambda=1.0$, showing 
$P^{\delta x}$ (left) and $E^{\delta x}$ (center) as a function
	of $c$. On the right, $E^{\delta x}$ as a function of $P^{\delta x}$.
}
\label{fig:PotBR-lambda=1.0-NRJ}
\end{figure}


\begin{figure}[!ht]
	\centering
	\begin{tabular}{cc}
		\resizebox{0.5\textwidth}{!}{\includegraphics{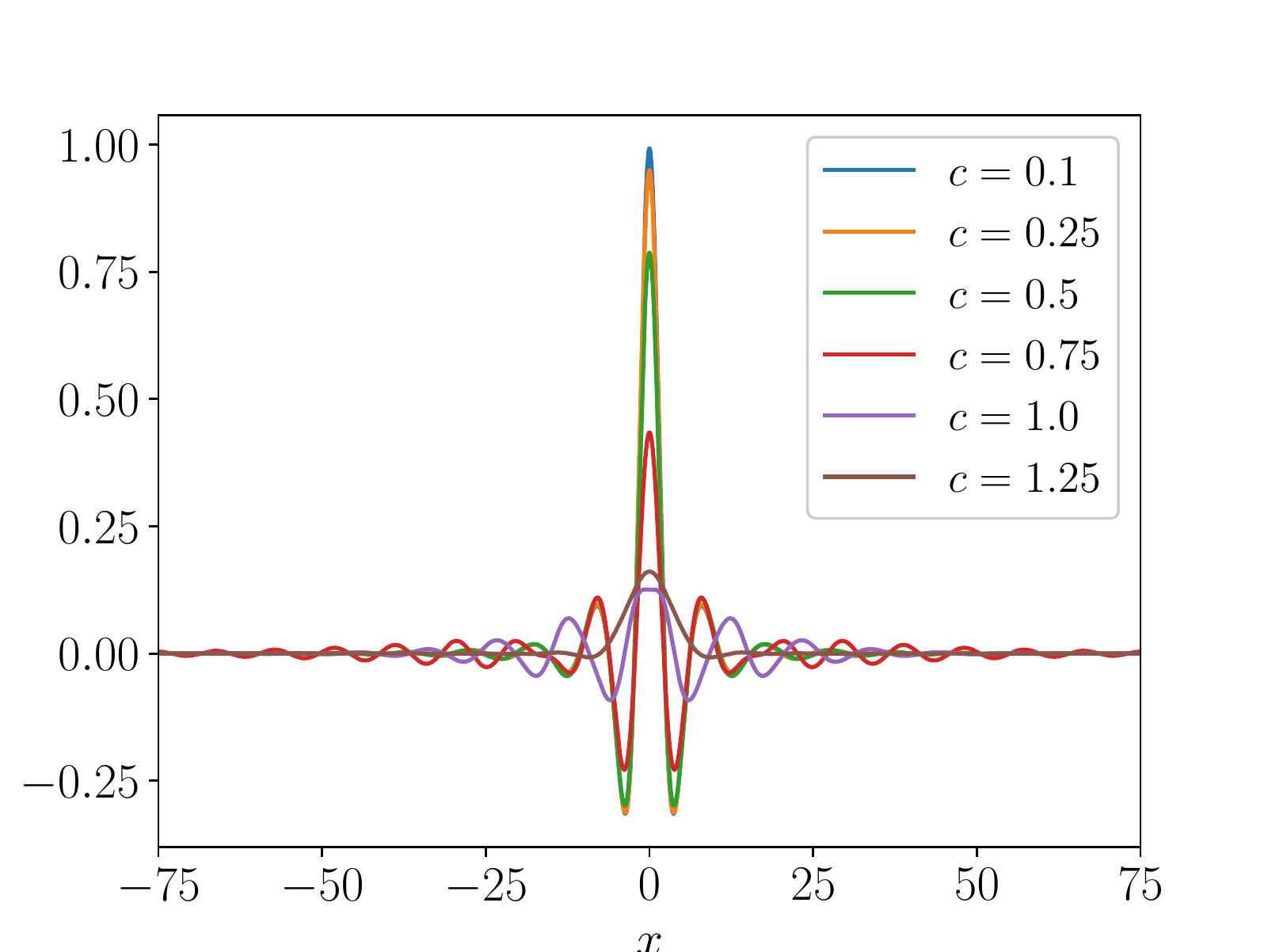}}
		&                                                                            
		\resizebox{0.5\textwidth}{!}{\includegraphics{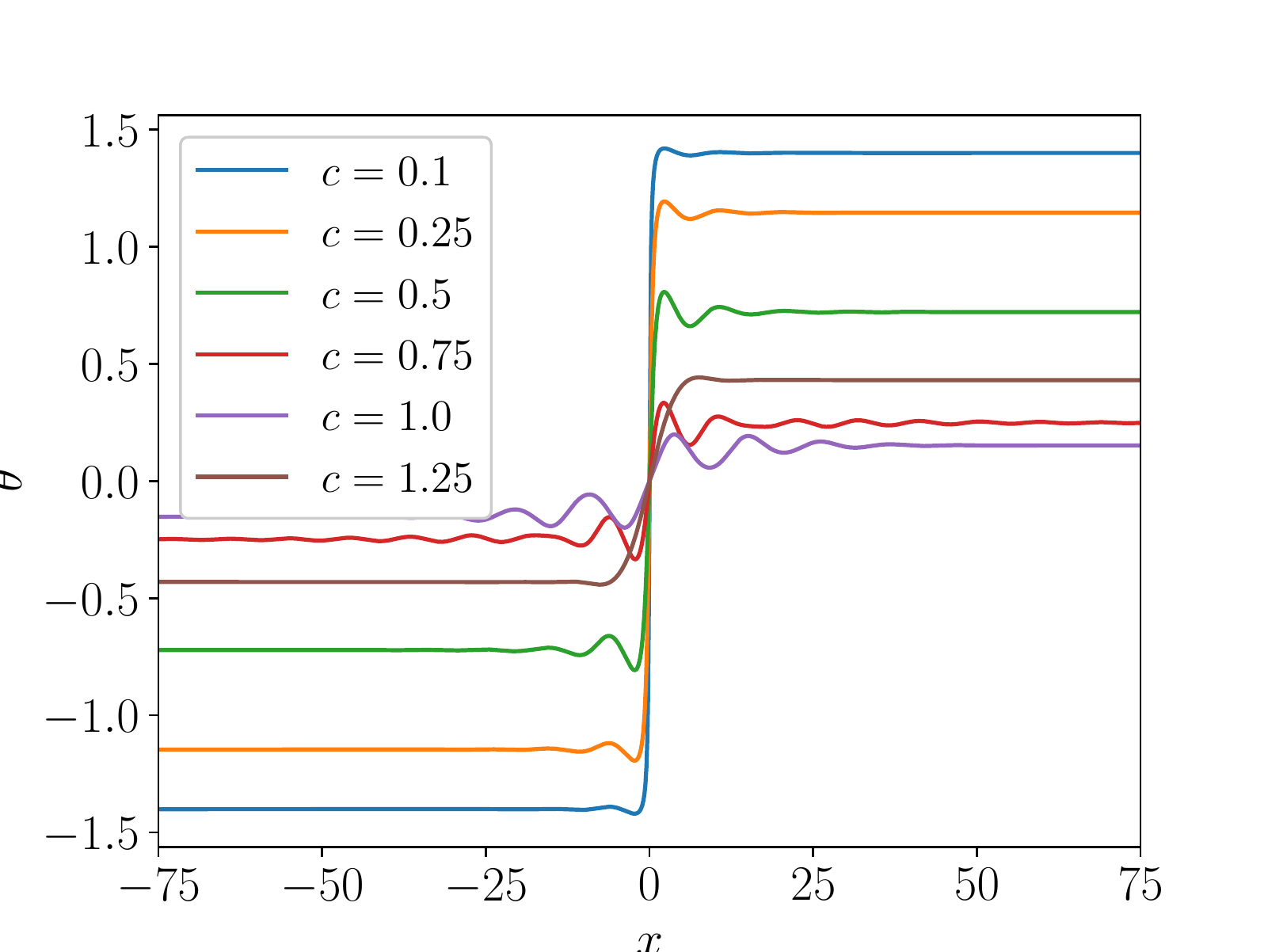}}
\end{tabular}
\caption{Numerically computed solitons for potential \eqref{eq:potBR}
	with $\lambda=4.0$, showing 
$\bm\eta_k$ (left) and  ${\bm \theta}_k$ (right) as	a function of $x_k$. 
}
\label{fig:PotBR-lambda=4.0-etatheta}
\end{figure}

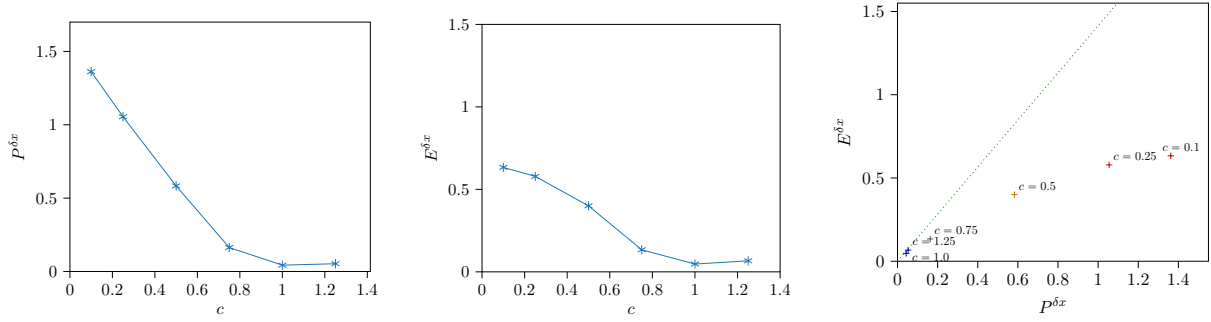
\begin{figure}[!ht]
\centering
\begin{tabular}{ccc}
	\resizebox{0.33\textwidth}{!}{
\begin{tikzpicture}

\definecolor{color0}{rgb}{0.12156862745098,0.466666666666667,0.705882352941177}

\begin{axis}[
tick align=outside,
tick pos=left,
x grid style={white!69.0196078431373!black},
xlabel={\(\displaystyle c\)},
xmin=0, xmax=1.4142135623731,
xtick style={color=black},
y grid style={white!69.0196078431373!black},
ylabel={\(\displaystyle P^{\delta x}\)},
ymin=0, ymax=1.7,
ytick style={color=black}
]
\addplot [semithick, color0, mark=asterisk, mark size=3, mark options={solid}]
table {%
0.1 1.36217262265445
0.25 1.05492578022971
0.5 0.583333852007676
0.75 0.163118389057553
1 0.0431171624952924
1.25 0.0528045429487633
};
\end{axis}

\end{tikzpicture}}
	&
	\resizebox{0.33\textwidth}{!}{
\begin{tikzpicture}

\definecolor{color0}{rgb}{0.12156862745098,0.466666666666667,0.705882352941177}

\begin{axis}[
tick align=outside,
tick pos=left,
x grid style={white!69.0196078431373!black},
xlabel={\(\displaystyle c\)},
xmin=0, xmax=1.4,
xtick style={color=black},
y grid style={white!69.0196078431373!black},
ylabel={\(\displaystyle E^{\delta x}\)},
ymin=0, ymax=1.5,
ytick style={color=black}
]
\addplot [semithick, color0, mark=asterisk, mark size=3, mark options={solid}]
table {%
0.1 0.633312788726185
0.25 0.578730318785132
0.5 0.40017346683664
0.75 0.133899422521728
1 0.0473244827632513
1.25 0.0668113229411979
};
\end{axis}

\end{tikzpicture}} & \resizebox{0.33\textwidth}{!}{
\begin{tikzpicture}

\begin{axis}[
tick align=outside,
tick pos=left,
x grid style={white!69.0196078431373!black},
xlabel={\(\displaystyle P^{\delta x}\)},
xmin=0, xmax=1.55,
xtick style={color=black},
y grid style={white!69.0196078431373!black},
ylabel={\(\displaystyle E^{\delta x}\)},
ymin=0, ymax=1.55,
ytick style={color=black}
]
\addplot [draw=blue, fill=blue, mark=+, only marks, scatter]
table{%
x  y
1.36217262265445 0.633312788726185
1.05492578022971 0.578730318785132
0.583333852007676 0.40017346683664
0.163118389057553 0.133899422521728
0.0431171624952924 0.0473244827632513
0.0528045429487633 0.0668113229411979
};
\addplot [semithick, green!50.1960784313725!black, dotted]
table {%
0 0
1.36217262265445 1.92640299725125
};
\draw (axis cs:1.30017262265445,0.663312788726185) node[
  scale=0.7,
  anchor=base west,
  text=black,
  rotate=0.0
]{$c=0.1$};
\draw (axis cs:1.05492578022971,0.608730318785132) node[
  scale=0.7,
  anchor=base west,
  text=black,
  rotate=0.0
]{$c=0.25$};
\draw (axis cs:0.583333852007676,0.43017346683664) node[
  scale=0.7,
  anchor=base west,
  text=black,
  rotate=0.0
]{$c=0.5$};
\draw (axis cs:0.163118389057553,0.163899422521728) node[
  scale=0.7,
  anchor=base west,
  text=black,
  rotate=0.0
]{$c=0.75$};
\draw (axis cs:0.051171624952924,0.0023244827632513) node[
  scale=0.7,
  anchor=base west,
  text=black,
  rotate=0.0
]{$c=1.0$};
\draw (axis cs:0.0528045429487633,0.0968113229411979) node[
  scale=0.7,
  anchor=base west,
  text=black,
  rotate=0.0
]{$c=1.25$};
\end{axis}

\end{tikzpicture}}
\end{tabular}
\caption{Numerically computed solitons for potential \eqref{eq:potBR}
	with $\lambda=4.0$, showing
$P^{\delta x}$ (left) and $E^{\delta x}$ (center) as a function
	of $c$. On the right, $E^{\delta x}$ as a function of $P^{\delta x}$.
}
\label{fig:PotBR-lambda=4.0-NRJ}
\end{figure}

\subsection{Solitons with potential \eqref{eq:potPierre}}
We consider the potential $\mathcal W$ defined in \eqref{eq:potPierre}.
The minimizers are computed with $L=800$ and $N=7999$ so that $\delta x = 0.1$.
The results are displayed in Figures \ref{fig:PotPierre-L=800-etatheta}
and \ref{fig:PotPierre-L=800-NRJ}.
As seen in Figure~\ref{dispersion:E6}, there is a roton minimum and $c_L \# 0.596$. This is no 
obstacle for the computation of soliton for every subsonic speed. Since  $P^{\delta x}$ is a 
strictly decreasing function, these solitons should be stable. The highly oscillating behavior 
could be explained again by the gap between the Landau speed and the speed of sound.




\begin{figure}[!ht]
  \centering
  \begin{tabular}{cc}
    \resizebox{0.5\textwidth}{!}{
      \includegraphics{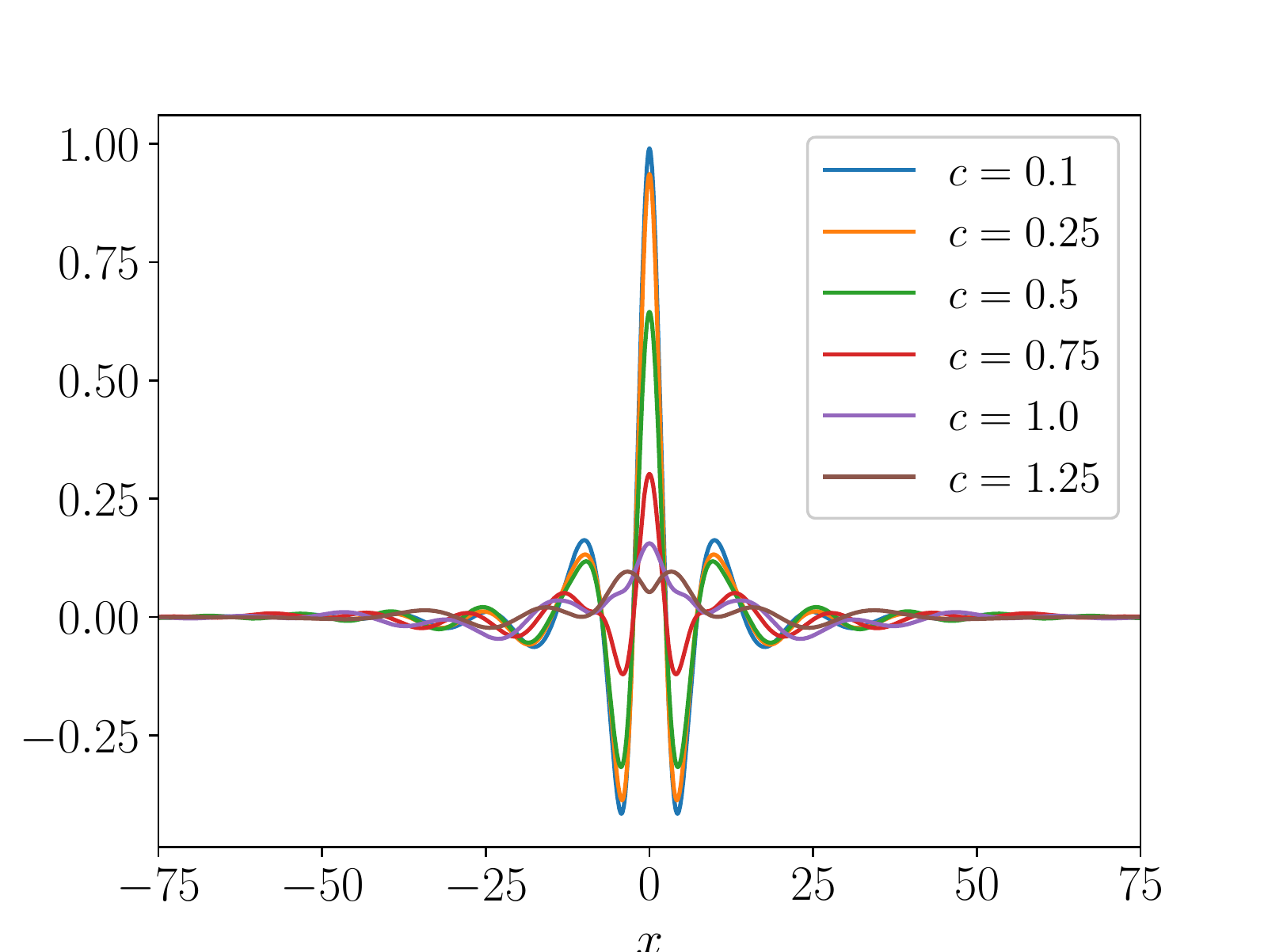}}
    &
\resizebox{0.5\textwidth}{!}{\includegraphics{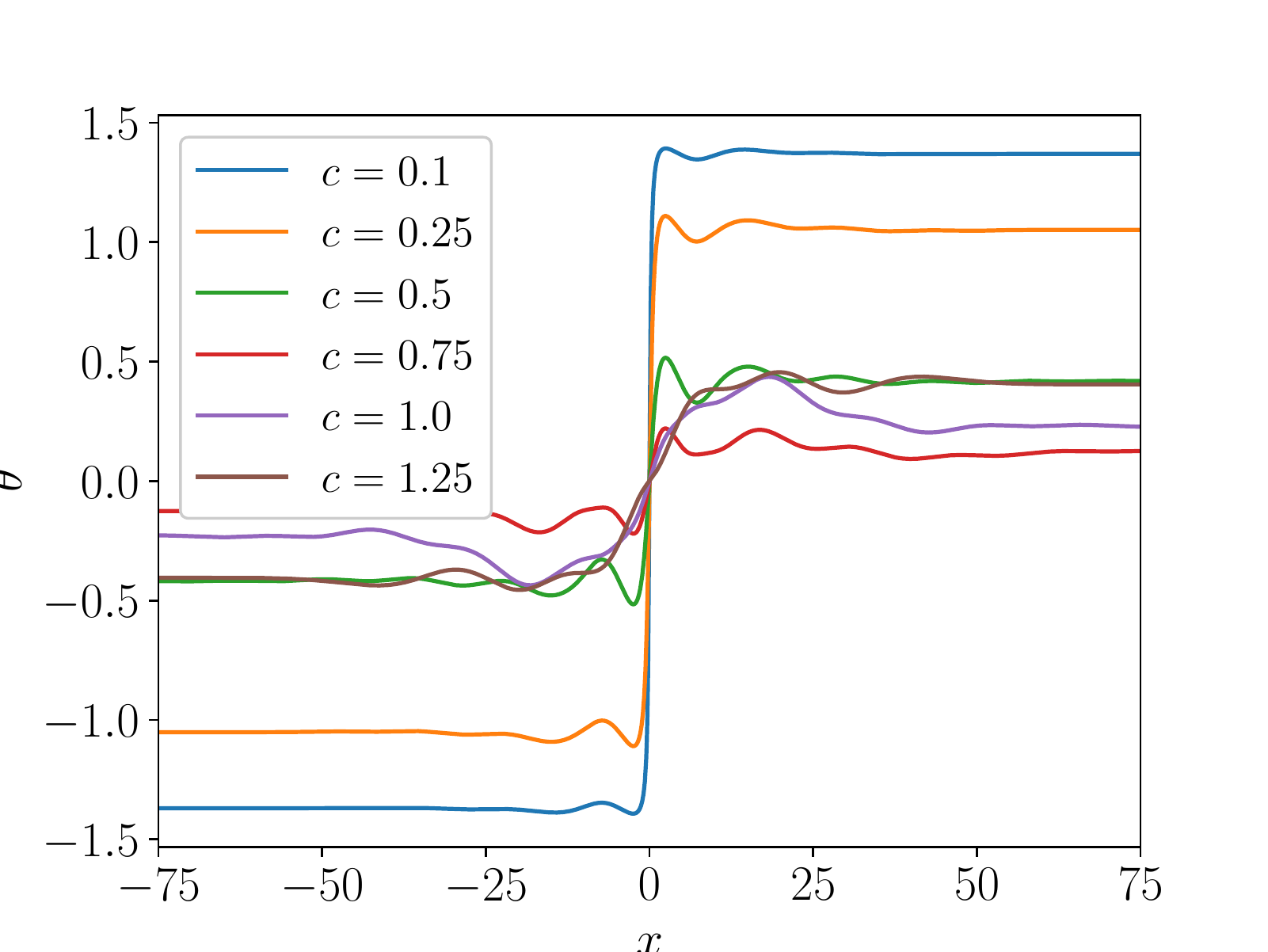}}
  \end{tabular}
  \caption{Numerically computed solitons for potential \eqref{eq:potPierre},
  	showing $\bm\eta_k$ (left) and ${\bm \theta}_k$ (right) as
    a function of $x_k$.}
  \label{fig:PotPierre-L=800-etatheta}
\end{figure}

\begin{figure}[!ht]
  \centering
  \begin{tabular}{ccc}
     \resizebox{0.33\textwidth}{!}{
\begin{tikzpicture}

\definecolor{color0}{rgb}{0.12156862745098,0.466666666666667,0.705882352941177}

\begin{axis}[
tick align=outside,
tick pos=left,
x grid style={white!69.0196078431373!black},
xlabel={\(\displaystyle c\)},
xmin=0, xmax=1.4142135623731,
xtick style={color=black},
y grid style={white!69.0196078431373!black},
ylabel={\(\displaystyle P^{\delta x}\)},
ymin=0, ymax=1.7,
ytick style={color=black}
]
\addplot [semithick, color0, mark=asterisk, mark size=3, mark options={solid}]
table {%
0.1 1.33954103142172
0.25 0.98112192540132
0.5 0.338216411582617
0.75 0.0703879126613871
1 0.0418160310941601
1.25 0.0326218281123937
};
\end{axis}

\end{tikzpicture}}
    &
    \resizebox{0.33\textwidth}{!}{
\begin{tikzpicture}

\definecolor{color0}{rgb}{0.12156862745098,0.466666666666667,0.705882352941177}

\begin{axis}[
tick align=outside,
tick pos=left,
x grid style={white!69.0196078431373!black},
xlabel={\(\displaystyle c\)},
xmin=0, xmax=1.4,
xtick style={color=black},
y grid style={white!69.0196078431373!black},
ylabel={\(\displaystyle E^{\delta x}\)},
ymin=0, ymax=1.5,
ytick style={color=black}
]
\addplot [semithick, color0, mark=asterisk, mark size=3, mark options={solid}]
table {%
0.25 0.480390630267514
0.5 0.216944036940047
0.75 0.0556394509809669
1 0.0483560466395352
1.25 0.0401838526000534
};
\end{axis}

\end{tikzpicture}} & \resizebox{0.33\textwidth}{!}{
\begin{tikzpicture}

\begin{axis}[
tick align=outside,
tick pos=left,
x grid style={white!69.0196078431373!black},
xlabel={\(\displaystyle P^{\delta x}\)},
xmin=0, xmax=1.55,
xtick style={color=black},
y grid style={white!69.0196078431373!black},
ylabel={\(\displaystyle E^{\delta x}\)},
ymin=0, ymax=1.55,
ytick style={color=black}
]
\addplot [draw=blue, fill=blue, mark=+, only marks, scatter]
table{%
x  y
1.33954103142172 0.541055620584826
0.98112192540132 0.480390630267514
0.338216411582617 0.216944036940047
0.0703879126613871 0.0556394509809669
0.0418160310941601 0.0483560466395352
0.0326218281123937 0.0401838526000534
};
\addplot [semithick, green!50.1960784313725!black, dotted]
table {%
0 0
1.33954103142172 1.89439709399183
};
\draw (axis cs:1.33954103142172,0.571055620584826) node[
  scale=0.7,
  anchor=base west,
  text=black,
  rotate=0.0
]{$c=0.1$};
\draw (axis cs:0.98112192540132,0.510390630267514) node[
  scale=0.7,
  anchor=base west,
  text=black,
  rotate=0.0
]{$c=0.25$};
\draw (axis cs:0.338216411582617,0.246944036940047) node[
  scale=0.7,
  anchor=base west,
  text=black,
  rotate=0.0
]{$c=0.5$};
\draw (axis cs:0.0703879126613871,0.0856394509809669) node[
  scale=0.7,
  anchor=base west,
  text=black,
  rotate=0.0
]{$c=0.75$};
\draw (axis cs:0.0418160310941601,0.0783560466395352) node[
  scale=0.7,
  anchor=base west,
  text=black,
  rotate=0.0
]{$c=1.0$};
\draw (axis cs:0.0326218281123937,0.0701838526000534) node[
  scale=0.7,
  anchor=base west,
  text=black,
  rotate=0.0
]{$c=1.25$};
\end{axis}

\end{tikzpicture}}
  \end{tabular}
  \caption{Numerically computed solitons for potential \eqref{eq:potPierre}, showing 
 $P^{\delta x}$ (left) and $E^{\delta x}$ (center) as a function
    of $c$. On the right, $E^{\delta x}$ as a function of $P^{\delta x}$.
    }
  \label{fig:PotPierre-L=800-NRJ}
\end{figure}
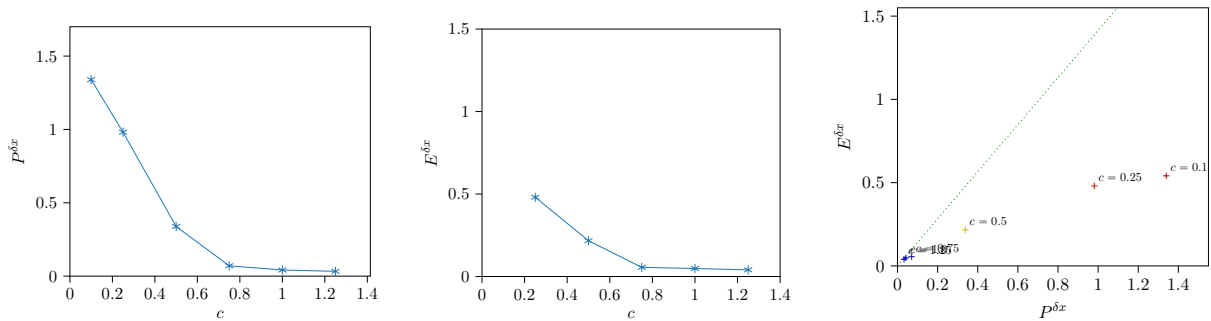

\section{Conclusion}
 This paper introduces a numerical method to compute dark solitons to the nonlocal Gross--Pitaevskii 
 equation with nonlocal realistic potentials, with a prescribed speed.
 Using an {\it ad hoc} formulation of the problem (see \eqref{eq:zeroJ}-\eqref{eq:defJc},
 and the discrete analogue \eqref{eq:defJnum}), one has to solve a nonlinear nonlocal problem
 ($J_c^{\delta x}({\bm \eta},\lambda)=0_{\R^N}$ for a fixed parameter value $\lambda\in\R$)
 in the neighborhood of a guess, which is linked with the known continuous soliton with same speed
 for $\lambda=0$. This problem is solved numerically by a gradient descent method minimizing
 the residue, and allows for numerical simulations for several physically realistic
 nonlocal interaction potentials.
 
 Theses numerical simulations provide insight into the behavior of solitons and allow us to
 discuss the influence of the nonlocal interactions
 on the shape of the dark solitons, as well as on their stability.
 Moreover, they suggest that the speed of sound and the Landau speed, 
 given by the dispersion relation, are important values for  understanding  the properties of these 
 dark 
 solitons. 

 In future research, we hope to find more precise criteria for determining the behavior of dark solitons.
For instance, we aim to obtain conditions on the interaction potential to ensure that the corresponding dark solitons exhibit either monotonic or oscillating behavior.
 
\begin{merci}
	The authors acknowledge support from the Labex CEMPI (ANR-11-LABX-0007-01).
	A.~de Laire was also supported by the ANR project ODA (ANR-18-CE40-0020-01).
	S.~L\'opez-Mart\'inez was supported by the Madrid Government (Comunidad de Madrid – Spain) 
	under the multiannual Agreement with UAM in the line for the Excellence of the University 
	Research Staff in the context of the V PRICIT (Regional Program of Research and Technological 
	Innovation).
\end{merci}

 \clearpage

\section*{Appendix}

\begin{proof}[Proof of Theorem~\ref{thm:dLMar}]
	Let $m=\inf_{\xi\in\R}(\widehat\W(\xi)^-)$, where $s^-=\min\{s,0\}$, and let $R=\sqrt{2(\sigma-m)}$. For any $\tilde{\sigma}\in (0,\sigma)$, we choose $k=(\tilde{\sigma}-m)/R^2=(\tilde{\sigma}-m)/(2(\sigma-m))$. Clearly, $k\in (0,1/2)$. 
	
	On the one hand, for a.e.\ $|\xi|\geq R$, one trivially has
	\[\widehat\W(\xi)\geq m=\tilde\sigma - kR^2\geq\tilde\sigma-k\xi^2.\]
	On the other hand, observe that $\sigma-\xi^2/2\geq\tilde\sigma-k\xi^2$ if, and only if, $|\xi|\leq R$. Therefore, for a.e.\ $|\xi|\leq R$, we get 
	\[\widehat\W(\xi)\geq \sigma-\frac{\xi^2}{2}\geq\tilde\sigma-k\xi^2.\]
	We may now apply Theorem~{1.1} in \cite{dLMar2022} and get a solution to \eqref{eq:TWc} for a.e.\ $c\in (0,\sqrt{2\tilde\sigma})$. Taking $\tilde\sigma$ arbitrarily close to $\sigma$ we obtain a solution for a.e.\ $c\in (0,\sqrt{2\sigma})$, which completes the proof.
\end{proof}

\bibliographystyle{abbrv}   

\end{document}